\documentclass[12pt]{amsart}
\usepackage{setspace}
\usepackage{verbatim}

\usepackage{amsfonts,fancyhdr,ifthen,bm}
\usepackage{amscd,amssymb} 
\usepackage{times}
\usepackage{graphicx}
\usepackage{color}
\usepackage{epsfig}
\usepackage{mathrsfs}
\usepackage{paralist}
\usepackage{comment}
\usepackage{mathtools}
\usepackage{multirow}
\usepackage[foot]{amsaddr}

\newtheorem{thm}{Theorem}[section]
\newtheorem*{thm*}{Theorem}
\newtheorem{prop}[thm]{Proposition}
\newtheorem*{prop*}{Proposition}
\newtheorem{cor}[thm]{Corollary}
\newtheorem*{cor*}{Corollary}

\newtheorem{lem}[thm]{Lemma}
\newtheorem*{lem*}{Lemma}

\newtheorem*{oquest*}{Open Question}

\theoremstyle{remark}

\theoremstyle{remark}
\newtheorem*{rmk*}{Remark}

\theoremstyle{definition}
\newtheorem{defn}[thm]{Definition}

\theoremstyle{definition}

\theoremstyle{definition}
\newtheorem*{defn*}{Definition}

\theoremstyle{definition}

\numberwithin{equation}{section}%numbers equations by section

\newcommand{\vect}[1]{\bm{#1}}

\newcommand{\Z}{\mathbb{Z}}
\newcommand{\QQ}{\mathbb{Q}}

\newcommand{\Gal}{\text{Gal}}

\DeclareFontFamily{U}{wncy}{}
\DeclareFontShape{U}{wncy}{m}{n}{<->wncyr10}{}
\DeclareSymbolFont{mcy}{U}{wncy}{m}{n}
\DeclareMathSymbol{\Sha}{\mathord}{mcy}{"58}

\newcommand\restr[2]{{% we make the whole thing an ordinary symbol
  \left.\kern-\nulldelimiterspace % automatically resize the bar with \right
  #1 % the function
  \vphantom{\big|} % pretend it's a little taller at normal size
  \right|_{#2} % this is the delimiter
  }}

\begin{document}
% This is cover.tex

% Insert your information as appropriate.

\title[$2^{\infty}$-Selmer groups and $2^{\infty}$-class groups]{
$2^{\infty}$-Selmer groups, $2^{\infty}$-class groups, and Goldfeld's conjecture
}

\author{Alexander Smith}
\address{Department of Mathematics, Harvard University}
\email{adsmith@math.harvard.edu}

\thanks{I would like to thank Noam Elkies, Jordan Ellenberg, and Melanie Matchett Wood for their support over the course of this project. I would also like to thank Brian Conrad, Dorian Goldfeld, Ye Tian, and David Yang for their comments on prior versions of this paper.}

\begin{abstract}
We prove that the $2^{\infty}$-class groups of the imaginary quadratic fields have the distribution predicted by the Cohen-Lenstra heuristic. Given an elliptic curve $E/\QQ$ with full rational $2$-torsion and no rational cyclic subgroup of order four, we analogously prove that the $2^{\infty}$-Selmer groups of the quadratic twists of $E$ have distribution as predicted by Delaunay's heuristic. In particular, among the twists $E^{(d)}$ with $|d| < N$, the number of curves with rank at least two is $o(N)$.
\end{abstract}

\maketitle
\tableofcontents

\section{Introduction}
\label{sec:intro}

Recall that a positive integer is called a congruent number if it is the area of some right triangle with rational side lengths. This paper was born as an  eventually-successful attempt to prove the following theorem.
\begin{thm*}
\label{yay}
The set of congruent numbers equal to $1$, $2$, or $3$ mod $8$ has zero natural density in $\mathbb{N}$.
\end{thm*}
Previously, the best upper bound on this density was due to Heath-Brown, who found the limit as $N$ approaches $\infty$ of the distribution of $2$-Selmer groups in the quadratic twist family
\[E^{(d)}: \,dy^2 = x^3 - x \quad\text{with } |d| < N.\]
He found that, among $d$ equal to $1$, $2$, or $3$ mod $8$, the minimal $2$-Selmer rank of two was attained in the limit by about $41.9\%$ of curves \cite{Heat94}. It is well known that $d$ is congruent if and only if $E^{(d)}$ has positive rank, and we always have an inequality
\[\text{rank}(E^{(d)}) \le -2 \,+\, \dim \text{Sel}^2 \,E^{(d)},\]
with $\dim$ denoting the dimension of the $2$-Selmer group as an $\mathbb{F}_2$-vector space. Then, from his Selmer group computations, Heath-Brown could show that at most $58.1\%$ of $d$ equal to $1$, $2$, or $3$ mod $8$ were congruent.

The work of Heath-Brown was extended by Kane to families of the form
\[E^{(d)}:\, dy^2 = x(x+a)(x+b)\quad\text{with } |d| < N,\]
where $a$ and $b$ are distinct nonzero rational numbers; that is to say, Kane assumed that $E/\QQ$ had full rational $2$-torsion. With the additional hypothesis that $E$ had no rational cyclic subgroup of order four, Kane proved that the limit of the distribution of the $2$-Selmer groups in this family approached the distribution found by Heath-Brown \cite{Kane13}. With these results, Kane was able to find upper bounds on the density of twists in this family with rank $\ge 2$.

Now, the $2$-Selmer rank provides a coarse upper bound for the rank of an elliptic curve. This bound can be improved by instead considering the ranks of the $2^k$-Selmer groups, with larger $k$ giving finer estimates for the rank of the elliptic curve. In fact, if the Shafarevich-Tate conjecture is true, we expect that the $\Z_2$-Selmer corank
\[\text{corank }\text{Sel}^{2^{\infty}} E = \lim_{k \rightarrow \infty} \dim 2^{k-1}\text{Sel}^{2^k}\,E\]
should equal the rank of $E$ for any elliptic curve $E/\QQ$.

With this in mind, the first goal of the paper is to find the distribution of the $2^k$-Selmer groups in the quadratic twist family of a curve $E/\QQ$. To write down the result, we will need some notation:
\begin{defn*}
For $n \ge j \ge 0$, take $P^{\text{Alt}}(j \,|\, n)$
to be the probability that a a uniformly selected alternating $n \times n$ matrix with coefficients in $\mathbb{F}_2$ has kernel of rank exactly $j$.

In addition, given an elliptic curve $E/\QQ$ and integers $n \ge 0$ and $k \ge 1$, take $R_{E,\, k}(n)$ to be the set of squarefree $d$ for which
\[\dim 2^{k-1}\text{Sel}^{2^k}(E^{(d)}) = \begin{cases} n + 2&\text{if } k = 1 \\ n &\text{otherwise.}\end{cases}\]
\end{defn*}

\begin{thm}
\label{thm:Sel_main}
Take $E/\QQ$ to be an elliptic curve with full rational $2$-torsion. Assume that $E$ has no rational cyclic subgroup of order four. Choose $m \ge 1$, and choose any sequence of nonnegative integers $n_1 \ge n_2 \ge \dots \ge n_{m+1}$ for which the $n_k$ are either all even or all odd. Then
\[\lim_{N \rightarrow \infty} \frac{\big|\{1, \dots, N\} \, \cap\,  R_{E,\, 1}(n_1) \cap \dots \cap R_{E,\, m}(n_m) \cap R_{E,\, m+1}(n_{m+1})\big|}{\big|\{1, \dots, N\} \, \cap\, R_{E,\, 1}(n_1) \cap \dots \cap R_{E,\, m}(n_m)\big|\qquad\qquad\quad\,}\]
\[ = \,P^{\text{\emph{Alt}}}(n_{m+1}\,|\, n_{m}).\]
\end{thm}
Together with Kane's results, this Markov-chain behavior establishes that the $2^{\infty}$-Selmer groups of the twists of such an elliptic curve $E/\QQ$ have the distribution predicted by Delaunay \cite{Del01} and Bhargava et al. \cite{BKLPR15}. This theorem also gives us very fine control on the rank of elliptic curves in this family. 
\begin{cor}
\label{cor:main}
Take $E/\QQ$ to be an elliptic curve with full rational $2$-torsion. Assume that $E$ has no rational cyclic subgroup of order four. Then, for any $N > 1$, we have
\[\big| \big\{ 1 \le d \le N\,:\,\, \text{\emph{corank }} \text{\emph{Sel}}^{2^{\infty}} E^{(d)} \ge 2 \big\}\big| = o(N).\]
\end{cor}
By applying this corollary to $E$ the curve $y^2 = x^3 - x$, we derive the zero density result that opened this paper. More generally, recall that Goldfeld's conjecture states that, given an elliptic curve $E/\QQ$, $50\%$ of the quadratic twists of $E$ have analytic rank $0$, $50\%$ have analytic rank $1$, and $0\%$ have higher analytic rank \cite{Gold79}. From global root number calculations, we know that $50\%$ of the twists will have even $\Z_2$-Selmer corank, and $50\%$ have odd $\Z_2$-Selmer corank. In light of this, we have the following.
\begin{cor}
Take $E/\QQ$ to be an elliptic curve with full rational $2$-torsion. Assume that $E$ has no rational cyclic subgroup of order four. Then, if the Birch and Swinnerton-Dyer conjecture holds for the set of twists of $E$, Goldfeld's conjecture holds for $E$.
\end{cor}

We will prove an explicit form of Theorem \ref{thm:Sel_main} as Theorem \ref{thm:Sel_main_exp} and an explicit form of Corollary \ref{cor:main} as Corollary \ref{cor:drown}. Neither of these results is likely to be sharp, with Corollary \ref{cor:drown} particularly egregious in this manner. As detailed in \cite{Park16}, most heuristics for ranks of elliptic curves suggest that, for any $\epsilon > 0$ and any elliptic curve $E$, there is some $N_0$ so, for $N > N_0$, we have
\[\big| \big\{ 1 \le d \le N\,:\,\, \text{rank}( E^{(d)}) \ge 2 \big\}\big| < N^{3/4 + \epsilon}.\]
We instead prove that, for any elliptic curve $E/\QQ$ as in Corollary \ref{cor:main} and any $c < \frac{\log 2}{\log 6}$, there is some $N_0$ so, for $N > N_0$, we have
\[\big| \big\{ 1 \le d \le N\,:\,\, \text{rank}( E^{(d)}) \ge 2 \big\}\big| < \frac{N}{(\log \log \log \log \log N)^c}.\]
If we assume the grand Riemann hypothesis, we can remove about three of these logarithms. The remaining two logarithms come from the use of Ramsey theory in our arguments, and are likely unremovable without a new proof strategy. 
 
Our final main result concerns the class groups of quadratic fields. For a positive integer $k$, the $2^k$-Selmer groups of quadratic twists of an elliptic curve and the $2^{k+1}$-class groups of imaginary quadratic fields are analogous families of objects. The strength of this analogy can be seen in the work of Fouvry and Kl{\"u}ners in \cite{Fouv07}. By modifying the strategy used by Heath-Brown to find the distribution of $2$-Selmer groups, this pair found that the distribution of the $4$-class groups in the family of imaginary quadratic fields was consistent with Gerth's extension of the Cohen-Lenstra heuristic to $p = 2$ \cite{Gert84, CoLe84}. Similarly, by modifying our approach to $2^k$-Selmer groups in Theorem \ref{thm:Sel_main}, we can find the distribution of $2^{k+1}$-class groups in the family of imaginary quadratic fields. We start by introducing the notation we will use.
\begin{defn*}
For $n \ge j \ge 0$, take $P^{\text{Mat}}(j\,|\,n)$ to be the probability that a uniformly selected $n \times n$ matrix with coefficients in $\mathbb{F}_2$ has kernel of rank exactly $j$.

In addition, given $k \ge 2$ and  $n \ge 0$, take $R_{\text{Im},\, k}(n)$ to be the set of squarefree $d$ for which
\[\dim 2^{k-1} \text{Cl}\, \QQ\big(\sqrt{-d}\big)[2^k] = n.\]
\end{defn*}
\begin{thm}
\label{thm:Cl_main}
Take $m \ge 2$, and choose any sequence of nonnegative integers $n_2 \ge n_3 \ge \dots \ge n_{m+1}$. Then
\[\lim_{N \rightarrow \infty} \frac{\big|\{1, \dots, N\} \, \cap\,  R_{\text{\emph{Im}},\, 2}(n_2) \cap \dots \cap R_{\text{\emph{Im}},\, m}(n_m) \cap R_{\text{\emph{Im}},\, m+1}(n_{m+1})\big|}{\big|\{1, \dots, N\} \, \cap\, R_{\text{\emph{Im}},\, 2}(n_2) \cap \dots \cap R_{\text{\emph{Im}},\, m}(n_m)\big|\qquad\qquad\quad\,\,\,}\]
\[ = \,P^{\text{\emph{Mat}}}(n_{m+1}\,|\, n_{m}).\]
\end{thm}
This Markov-chain behavior is consistent with the Cohen-Lenstra heuristic and represents the third major result towards this heuristic for quadratic fields, following the result of Fouvry-Kl{\"u}ners for $4$-torsion and the substantially earlier results of Davenport-Heilbronn for $3$-torsion \cite{DaHe71}.

Theorem \ref{thm:Sel_main} and \ref{thm:Cl_main} are generalizations of prior conditional results from \cite{Smith16b}, a paper by the author on $8$-class groups and $4$-Selmer groups. That paper was based on a result of R{\'e}dei that, for any negative squarefree $d$, there is some number field $M$ so, for any odd prime $p$ not dividing $d$, the $8$-class rank of $\QQ(\sqrt{dp})$ is determined by the splitting behavior of $p$ in $M/\QQ$ \cite{Rede39}. In \cite{CoLa83}, it was conjectured that this result can be extended to higher class groups. More specifically, it was conjectured that, for any $k > 1$ and any squarefree negative $d$, the structure of  $\text{Cl} \, \QQ(\sqrt{dp})[2^k]$ could be determined from the splitting behavior of $p$ in some governing field $M/\QQ$ determined from $d$ and $k$.

For $k > 3$, this conjecture is likely to be false for all $d$, with compelling evidence found by Milovic in \cite{Milo15}. 
However, the concept of a governing field remains useful for $k > 3$, as we can use splitting behavior to determine some relative information about class groups. In particular, for $d$ negative squarefree, and for $\{p_{10}, p_{11}\}, \dots,  \{p_{m0}, p_{m1}\}$ some sequence of pairs of distinct primes, we can sometimes derive the $2^m$-class structure of
\[\QQ\left(d^{1/2} \prod_{i \le m} p_{i0}^{1/2}\right)\]
from the $2^m$-class structures of the $2^m -1 $ fields
\[\QQ\left(d^{1/2} \prod_{i \le m} p_{if(i)}^{1/2}\right)\quad\text{with}\quad f \in \mathbb{F}_2^{\{1, \dots, m\}} - \{0\}\]
together with the splitting behavior of $p_{10}$ and $p_{11}$ in a governing field determined from the primes $p_{20}, p_{21}, \dots, p_{m0}, p_{m1}$.  Thinking of the quadratic fields as lying at the vertices of some $m$-dimensional cube, we can rephrase this result as finding the $2^m$-class structure at one vertex of the cube from the $2^m$-class structures at all the other vertices. We have a similar result for predicting $2^m$-Selmer structure at one vertex of an $m + 1$ dimensional cube from the $2^m$-Selmer structures at all the other vertices in the cube.

Making this relative governing field idea concrete takes up most of Sections \ref{sec:alg1} and \ref{sec:add_res}. The governing fields we need are constructed as the fields of definition of certain Galois cochains that we call \emph{governing expansions}. In Section \ref{ssec:gov_exp}, we prove the existence and basic properties of these cochains.

Next, in Section \ref{ssec:raw_exp}, we study sets of Galois cocycles on cubes of quadratic twists of a given Galois module. We find that the na{\"i}ve way of summing this set of cocycles gives a cocycle under one set of hypotheses (Proposition \ref{prop:psi_coh}) and gives a governing expansion under another set of hypotheses (Proposition \ref{prop:R_G_agree}). These two simple propositions are the most fundamental results in this paper. On the class side, the results are used in Section \ref{ssec:class_exp} to prove Theorem \ref{thm:Cl_gov}, which allows us to control the Artin pairing on the $2^k$-class groups. On the Selmer side, the results are used in Section \ref{ssec:Sel_exp} to prove Theorem \ref{thm:Sel_gov}, which allows us to control the Cassels-Tate pairing on the $2^k$-Selmer groups.

The conditions under which we can use either of these theorems are extremely limited. In Section \ref{sec:add_res}, we axiomatize some of the conditions with a structure that we denote an \emph{additive-restrictive system}. Over the course of this technical section, we find additive-restrictive systems that handle sets of governing expansions and systems that handle sets of cocycles coming from either class structure or Selmer structure. Using this new terminology, we reduce Theorems \ref{thm:Cl_gov} and \ref{thm:Sel_gov} to Proposition \ref{prop:AR_main}.

We have almost no control on the shape of these additive-restrictive systems. That said, as we will show in Proposition \ref{prop:ars_density}, we do have some control on their sizes. We can then prove the equidistribution results we want on these arbitrarily-shaped additive-restrictive systems using Ramsey theory. This is the main goal of Section \ref{sec:rams}, a section that cumulates in the proof of Proposition \ref{prop:staff}. As a first step towards this proposition, we prove the following lovely result:
\begin{prop*}
Take $d \ge 2$ to be an integer, take $2^{-d-1} > \delta > 0$, and take $X_1, \dots, X_d$ to be finite sets with cardinality at least $n > 1$. Suppose that $Y$ is a subset of $X = X_1 \times \dots \times X_d$ of cardinality at least $\delta \cdot |X|$. Then, for any positive $r$ satisfying
\[r \le \left(\frac{\log n}{5 \log \delta^{-1}}\right)^{1/(d-1)},\]
there exists a choice of sets $Z_1, \dots, Z_d$, each of cardinality $r$, such that
\[Z_1 \times \dots \times Z_d \subseteq Y.\]
\end{prop*}
This bound on $r$ can be shown to be sharp up to a change of constant using the probabilistic method.

Through Sections \ref{sec:add_res} and \ref{sec:rams}, we are working with a grid of quadratic twists. We cannot explicitly find the $2^k$-Selmer structure or $2^{k+1}$-class structure at any point in this grid. At the same time, under the condition that the corresponding grid of governing Artin symbols behaves generically, the results of these two sections let us say that the $2^k$-Selmer groups and $2^{k+1}$-class groups have the distribution we expect anyways.

The next goal is to find a situation where this grid of Artin symbols usually behaves generically. If we had assumed the grand Riemann hypothesis, this step would be straightforward.  As we are not using this hypothesis, it takes a three-logarithm detour to deal with this grid of Artin symbols. To understand the issue, choose some large $N$, and choose $n$ uniformly among the positive squarefree integers less than $n$. Write  $p_1 < \dots < p_r$ for the sequence of prime factors of $n$. Choose $k < r$, and take $M$ to be a number field of discriminant near $p_1 \cdot \dots \cdot p_k$. Suppose we wish to control the splitting of $p_{k+1}, \dots, p_{r}$ in $M$ as we vary these primes in small intervals $X_{k+1}, \dots, X_r$. Using the strongest form of unconditional Chebotarev available to us (Proposition \ref{prop:2Cheb}), we find that we can only do this if the gap 
\[\log \log p_{k+1} - \log \log p_k\]
is unusually large. We call tuples with such a sufficiently large gap \emph{extravagantly spaced}. By carefully analyzing the Poisson point process that models prime divisors, we are able to show that most tuples are (just barely!) extravagantly spaced. This is the main focus of Section \ref{sec:pois}, with our main result being Theorem \ref{thm:SkND}.

To avoid thinking about more complicated objects, we usually understand $2^k$-Selmer structure via the natural inclusion
\[2^{k-1}\text{Sel}^{2^k}E \,\subseteq\, \text{Sel}^2 E\big/E[2] \quad\text{for}\quad k > 1,\]
and we think about $2^{k+1}$-class structure similarly. For this to work, we need good ways to control $2$-Selmer and $4$-class structures. The distributional results of Heath-Brown, Fouvry and Kl{\"u}ners, and Kane are based on moment calculations for these groups and are difficult to use for more specialized sets of integers. As an alternative, we take the following tack. Choosing $p_1 < \dots < p_r$ as in the previous paragraph, we can define an $r \times r$ matrix $M$ whose off-diagonal coefficient $M_{ij}$ is given by the Legendre symbol $\left(\frac{p_i}{p_j}\right)$. With some extra quadratic-residue information, this matrix can be used to determine the $2$-Selmer structure of $E^{(d)}$; our main aim of Section \ref{sec:Leg} is to prove that the matrix $M$ is almost equidistributed among all possibilities satisfying quadratic reciprocity after some of the $p_i$ are permuted. The major analytic ingredients for this work are Chebotarev's density theorem and the large sieve. With these tools and a subtle induction argument, we prove Proposition \ref{prop:eqd_lgn_raw}, a weak equidistribution result for Legendre symbol matrices. By accounting for the effect of permuting primes via some basic combinatorics, we can strengthen this result to the form given in Theorem \ref{thm:Lgn_perm}. In Section \ref{ssec:box}, we make the transition from the set of all integers to certain product spaces of integers that we call \emph{boxes}. By applying Theorem \ref{thm:Lgn_perm} to these boxes, we then rederive Kane's results directly as Corollary \ref{cor:Kane}.

Finally, in Section \ref{sec:proofs}, we use the results of Sections \ref{sec:pois} and \ref{sec:Leg} to shave the set of integers $\{1, \dots, N\}$ to grids on which the additive-restrictive systems of Section \ref{sec:add_res} can be defined, and on which the governing grid of Artin symbols behaves generically. Using this, we prove Theorem \ref{thm:Sel_main_exp} and Corollary \ref{cor:drown}, and these give our main results on the Selmer side. We omit the analogous arguments used for the results on the class side, as no new ideas are needed for the translation.

\section{Algebraic tools}
\label{sec:alg1}
We will use the following notation:
\begin{itemize}
\item $X$ will denote a product
\[X = X_1 \times X_2 \times \dots \times X_d,\] 
where each $X_i$ is a finite set. In all our applications, the $X_i$ will be disjoint collections of odd primes.
\item For $a$ a positive integer, $[a]$ will denote the set $\{1, \dots, a\}$.
\item For $S \subseteq [d]$, we define
\[\overline{X}_S = \prod_{i \in S} (X_i \times X_i)  \times \prod_{i \in [d] - S} X_i.\]
We use $\pi_i$ to denote projection to the $i^{th}$ factor.
\item We denote the projections of $X_i \times X_i$ to $X_i$ by $\pi_0$ and $\pi_1$.
\item For $S, S_0 \subseteq [d]$, we take
\[\pi_{S_0}: \overline{X}_S \rightarrow \prod_{i \in S \cap S_0} (X_i \times X_i) \times \prod_{i \in ([d] - S) \cap {S_0}} X_i\]
to be the natural projection.
\item Given an element $\bar{x} \in \overline{X}_S$ and a subset $T$ of $S$, and writing $U = S - T$, we define a subset $\widehat{x}(T)$ of $\overline{X}_T$ by
\[ \left\{\bar{y} \in \overline{X}_T \,:\,\,  \pi_i(\bar{y}) \in \pi_i(\bar{x}) \text{ for } i \in U \text{ and }  \pi_{[d]- U}(\bar{y}) = \pi_{[d] - U}(\bar{x})\right\}.\] 
\end{itemize}

Fix some algebraic closure $\overline{\QQ}$ of $\QQ$, and take $G_{\mathbb{\QQ}} = \text{Gal}(\overline{\QQ}/\QQ)$. All the number fields used in this paper will be Galois extensions of $\QQ$ inside of this algebraic closure.

Our main results will fall as a consequence of the Chebotarev density theorem. We begin by constructing the sets of governing fields we need to do this.
\subsection{Sets of governing expansions}
\label{ssec:gov_exp}

Take $X_1, \dots, X_d$ to be disjoint collections of odd primes, and take $X$ to be their product. Given a subset $S$ of $[d]$ and $\bar{x} \in \pi_S(\overline{X}_S)$, define
\[K(\bar{x}) = \prod_{i \in S} \QQ\left(\sqrt{\pi_0(\pi_i(\bar{x}))\cdot \pi_1(\pi_i(\bar{x}_i))}\right)\]
where we use the $\prod$ symbol to denote a composition of number fields.

For $T \subseteq S \subseteq [d]$ and $\bar{x} \in \overline{X}_S$, take $\chi_{T,\, \bar{x}}: G_{\QQ} \rightarrow \mathbb{F}_2$ to be defined by
\[\chi_{T, \,\bar{x}}(\sigma) = \begin{cases} 1 &\text{if } \sigma(\sqrt{p_{0i}p_{1i}}) = -\sqrt{p_{0i}p_{1i}} \text{ for } i \in T \\ 0 &\text{otherwise,}\end{cases}\]
where we have taken $(p_{0i}, p_{1i})$ to be the coordinate $\pi_i(\bar{x})$ for $i \in S$.

From the equation
\[\chi_T(\sigma\tau) = \prod_{i \in T} \big(\chi_{\{i\}}(\sigma) + \chi_{\{i\}}(\tau)\big) = \sum_{U \subseteq T} \chi_{U}(\sigma)\chi_{T - U}(\tau),\]
we calculate the coboundary of $\chi_T$ to be
\begin{equation}
\label{eq:chi_cob}
d\chi_T(\sigma, \tau) = \sum_{\emptyset \ne U \subsetneq T} \chi_U(\sigma) \cdot \chi_{T-U}(\tau).
\end{equation}
This equation is the backbone for the following definition.

\begin{defn*}
Choose $S_0 \subseteq [d]$, choose $\bar{x}$ in $\overline{X}_{S_0}$, and choose some homomorphism
\[\phi_{\emptyset}: G_{\QQ} \rightarrow \mathbb{F}_2.\]
Suppose that we have a set of maps $\phi_S$ indexed by the subsets of $S_0$ such that we have the coboundary relation
\begin{equation}
\label{eq:expns}
d\phi_S (\sigma, \tau) = \sum_{\emptyset \ne T \subseteq S}  \chi_{T, \,\bar{x}}(\sigma) \cdot \phi_{S - T}(\tau)
\end{equation}
for each subset $S$ of $S_0$. Then, if $\phi_S$ is defined, we call it a $(S, \,\bar{x})$-expansion of $\phi_{\emptyset}$.
\end{defn*}

Using \eqref{eq:chi_cob}, we can verify that the right hand side of \eqref{eq:expns} has zero coboundary, so this definition is reasonable.

There are two main ways to construct expansions. In this section, we will use class field theory to construct $(S, \,\bar{x})$ expansions from a set of smaller expansions. The ramification of these \emph{governing expansions} can be precisely controlled, so their fields of definitions can be used as governing fields. In Section \ref{ssec:raw_exp}, we instead find expansions by summing cocycles representing Selmer or class elements over the points of $\widehat{x}(\emptyset)$. Such expansions are less nicely behaved.  However, if we calculate enough of these expansions within a small space, we can force some of the expansions to equal a governing expansion. This gives us enough control over the Selmer groups and class groups to prove our main theorems.

We start with the class field theory we will need. We assume that  the reader is familiar with the material in \cite{Serre79}.
\begin{prop}
\label{prop:S_expns}
Take $X_1, \dots, X_d$ to be disjoint collections of odd primes, and write $X$ for their product. Choose a subset $S \subseteq [d]$ and a member $\bar{x}$ of $\overline{X}_S$. Take $\phi_{\emptyset} \in H^1(G_{\QQ}, \mathbb{F}_2)$. Suppose we have $(S - \{i\}, \, \bar{x})$ expansions of $\phi_{\emptyset}$ for all $i$ in $S$, and take $M_i$ to be the field of definition for $\phi_{S - \{i\}}$. Write
\[M = K(\bar{x}) \prod_{i \in S} M_i.\]

Write $(p_{0i}, p_{1i}) = \pi_i(\bar{x})$. Suppose that, for all $i$ in $S$,
\begin{itemize}
\item  $p_{0i}$ and $p_{1i}$ split completely in the extension $M_i/\QQ$, and
\item $p_{0i}p_{1i}$ is a square at $2$ and at all primes where $M_i/\QQ$ is ramified.
\end{itemize}
Then $\phi_{\emptyset}$ has an $(S,\, \bar{x})$ expansion $\phi_S$ whose field of definition is unramified above $M$ at all finite places.
\end{prop}
\begin{proof}
We need to check that the cocycle given on the right hand side of \eqref{eq:expns} is zero in $H^2(G_{\QQ}, \mathbb{F}_2)$. Call this cocycle $\psi$. Identifying $\mathbb{F}_2$ with $\pm 1$ and using the exact sequence
\[1 \xrightarrow{\quad\,} \pm 1 \xrightarrow{\quad\,} \overline{\QQ}^{\times} \xrightarrow{\,\,\,2\,\,\,} \overline{\QQ}^{\times} \xrightarrow{\quad\,} 1,\]
we find an exact sequence
\[0 = H^1\big(G_{\QQ}, \overline{\QQ}^{\times}\big) \rightarrow H^2\big(G_{\QQ}, \mathbb{F}_2\big) \rightarrow H^2\big(G_{\QQ}, \overline{\QQ}^{\times}\big),\]
with the left equality by Hilbert 90. But we know that the map
\[ H^2\big(G_{\QQ}, \overline{\QQ}^{\times}\big) \rightarrow \prod_{v} H^2\big(\text{Gal}(\overline{\QQ}_v/\QQ_v), \overline{\QQ}_v^{\times}\big)\]
is injective, where the product is over all places of $\QQ$. Furthermore, the conditions of the proposition imply that $\text{inv}_v(\psi)$ is zero at all places. Then $\psi$ is the image of some $1$-cochain. This cochain corresponds to a $\mathbb{F}_2$ central extension of $M$.

Write this extension as $M(\sqrt{\alpha})/M$. This extension is Galois over $\QQ$, so if $M(\sqrt{\alpha})/\QQ$ is ramified at some place $p$ other than $2$ or $\infty$ where $M/\QQ$ is unramified, we can lose the ramification by multiplying $\alpha$ by $p$. Now, suppose $M/\QQ$ is ramified at $p$. We see that the local conditions force $\psi$ to be trivial on $\text{Gal}(\overline{\QQ}_p/\QQ_p)$, so $M_p(\sqrt{\alpha})/\QQ_p$ has Galois group 
\[(\Z/2\Z) \times \text{Gal}(M_p/\QQ_p)\]
 if $M_p(\sqrt{\alpha})$ does not equal $M_p$. But the inertia group cannot contain $(\Z/2\Z)^2$ for $p \ne 2$, so $M_p(\sqrt{\alpha})/M_p$ is unramified. At $p=2$, we can avoid ramification by multiplying $\alpha$ by $\pm 2$ or $\pm 1$.
\end{proof}

With this out of the way, we can define systems of governing expansions.

\begin{defn}
Take $X_1, \dots, X_d$ to be disjoint collections of odd primes, and write $X$ for their product. Fix $i_a \le d$ and $\overline{Y}_{\emptyset} \subseteq X$. Suppose we choose the following objects:
\begin{itemize}
\item For each subset $S \subseteq [d]$ containing $i_a$, we choose a subset
\[\overline{Y}_S \subseteq \overline{X}_S.\]
\item For each $S \subseteq [d]$ containing $i_a$ and each $\bar{x} \in \overline{Y}_S$, we choose a continuous function
\[\phi_{\bar{x}}: G_{\QQ} \rightarrow \mathbb{F}_2,\]
taking $M(\bar{x})$ to be the minimal field of definition of this $\phi_{\bar{x}}$.
\end{itemize}
We call the collection of $\phi_{\bar{x}}$ a \emph{set of governing expansions} if the following criteria are satisfied.
\begin{enumerate}
\item If $\bar{x}$ is in $\overline{Y}_{\{i_a\}}$, then
\[\phi_{\bar{x}} = \chi_{i_a,\, \bar{x}}.\]
\item If $S$ contains $i_a$ and $\bar{x}$ is in $\overline{Y}_S$, then
\[\widehat{x}(T) \subset \overline{Y}_T \quad \text{for}\quad i_a \in T \subseteq S \quad\text{or for}\quad T = \emptyset.\]
Choosing $\bar{x}_{S- T}$ arbitrarily in $\widehat{x}(S - T)$, we have
\[d \phi_{\bar{x}}(\sigma, \tau) = \sum_{i_a \not\in T \subset S} \chi_{T,\, \bar{x}}(\sigma) \cdot \phi_{\bar{x}_{S - T}}(\tau).\]

\item Suppose $\bar{x}_1, \bar{x}_2$ are in $\overline{X}_S$, and suppose that
\[\bigg\{\pi_0\big(\pi_i(\bar{x}_1)\big),\,\, \pi_1\big(\pi_i(\bar{x}_1)\big)\bigg\} = \bigg\{\pi_0\big(\pi_i(\bar{x}_2)\big),\,\, \pi_1\big(\pi_i(\bar{x}_2)\big)\bigg\}\]
for all $i \in S$. Then, if
\[\widehat{x}_1(\emptyset) \cup \widehat{x}_2(\emptyset) \subseteq \overline{Y}_{\emptyset},\]
we have an equivalence
\[\bar{x}_1 \in \overline{Y}_S \,\,\Longleftrightarrow\,\, \bar{x}_2 \in \overline{Y}_S.\]
If both lie in $\overline{Y}_S$, then they satisfy
\[\phi_{\bar{x}_1} = \phi_{\bar{x}_2}.\]
\item \emph{(Additivity)} Taking $i \in S \subseteq [d]$, suppose $\bar{x}_1$, $\bar{x}_2$, $\bar{x}_3 \in \overline{Y}_S$ satisfy
\[\pi_{S -\{i\}}(\bar{x}_1) = \pi_{S -\{i\}}(\bar{x}_2) =\pi_{S -\{i\}}(\bar{x}_3)\]
and
\[\pi_i(\bar{x}_1) = (p_1, p_2),\,\,\,\pi_i(\bar{x}_2) = (p_2, p_3), \,\,\,\pi_i(\bar{x}_3) = (p_1, p_3).\]
Then
\[\phi_{\bar{x}_1}  + \phi_{ \bar{x}_2} = \phi_{ \bar{x}_3}.\]
\item If $\bar{x} \in \overline{Y}_S$, then $M(\bar{x})K(\bar{x})/K(\bar{x})$ is unramified at all finite places.

\item Take $\bar{x} \in \overline{X}_S$. Suppose that
\[\widehat{x}(\emptyset) \subseteq \overline{Y}_{\emptyset}\]
and that, for all $i \in S - \{i_a\}$, we have
\[\widehat{x}(S - \{i\}) \subseteq \overline{Y}_{S-\{i\}}.\]
Choosing $\bar{x}_i \in \widehat{x}(S - \{i\})$, suppose further that, for each $i \in S$, $\pi_0(\pi_i(\bar{x}))$ and $\pi_1(\pi_i(\bar{x}))$ split completely in $M(\bar{x}_i)$ and
\[\pi_0(\pi_i(\bar{x}))\pi_1(\pi_i(\bar{x}))\]
is a quadratic residue at $2$ and at all primes ramifying in $K(\bar{x}_i)/\QQ$. Then
\[\bar{x} \in \overline{Y}_S.\]
\end{enumerate}

We will use the letter $\mathfrak{G}$ to denote a set of governing expansions, writing $\overline{Y}_S(\mathfrak{G})$, $i_a(\mathfrak{G})$, etc. to denote the data associated with $\mathfrak{G}$.
\end{defn}

Additivity reflects a natural tensor product structure present in a set of governing expansions. We can explicitly uncover this linear structure via iterated commutators.
\begin{defn*}
Given a set of governing expansions, choose any $S \ni i_a$ and any $\bar{x} \in \overline{Y}_S$. Write $k = |S|$, and define
\begin{equation}
\label{eq:beta_def}
\beta_k \phi_{\bar{x}}(\sigma_1, \dots, \sigma_k) = \phi_{\bar{x}}\big([\sigma_1, \,[\sigma_2,\, [ \dots, \,[\sigma_{k-1}, \,\sigma_k]\dots]]]\big)
\end{equation}
\end{defn*}
Note that
\[\phi_{\bar{x}}\big([\sigma, \tau]\big) = \phi_{\bar{x}}(\sigma \tau) + \phi_{\bar{x}}(\tau \sigma) + d\phi_{\bar{x}}\big([\sigma, \tau],\, \,\tau\sigma\big).\]
From \eqref{eq:expns}, we see that the coboundary above is zero since each $\chi_T$ has abelian field of definition. But we have
\[\phi_{\bar{x}}(\sigma\tau) + \phi_{\bar{x}}(\tau\sigma) = d\phi_{\bar{x}}(\sigma, \tau) + d\phi_{\bar{x}}(\tau, \sigma).\]
Take $\text{Bij}^*([k],\, S)$ to be the set of bijective maps $g$ from $[k]$ to $S$ such that either $g(k-1)$ or $g(k)$ equals $i_a$. Then, in light of the above equation and \eqref{eq:expns}, we can calculate
\[\beta_k \phi_{\bar{x}}(\sigma_1, \dots, \sigma_k) = \sum_{g \in \text{Bij}^*([k],\, S)} \prod_{i \le k} \chi_{g(i), \, \bar{x}}(\sigma_i).\]

Write
\[K(X) = \prod_{\bar{x} \in \overline{X}_{[d]}}K(\bar{x}),\]
and write $V$ for the $\mathbb{F}_2$ vector space $\text{Gal}(K(X)/\QQ)$. Then $\beta_k$ can be considered as a linear operator from the space generated by the $\phi_{\bar{x}}$ to
\[\bigotimes_{i \in S} \text{Hom}(V,\, \mathbb{F}_2).\]

If $\bar{x}_1$, $\bar{x}_2$, and $\bar{x}_3$ are as in part (4) of the definition above, we see that
\begin{equation}
\label{eq:bk_add}
\beta_k \phi_{\bar{x}_1} + \beta_k \phi_{\bar{x}_2} = \beta_k \phi_{\bar{x}_3}.
\end{equation}
This turns out to be a good way to force additivity on our set of governing expansions.

\begin{prop}
\label{prop:gov_set_ex}
For any choice of a product $X$ of disjoint sets $X_1, \dots, X_d$ of odd primes, for any choice of $i_a \in [d]$, and for any choice of $\overline{Y}_{\emptyset}$, there is a set of governing expansions $\mathfrak{G}$ defined on $X$ with $i_a(\mathfrak{G}) = i_a$ and $\overline{Y}_{\emptyset}(\mathfrak{G}) = \overline{Y}_{\emptyset}$.
\end{prop}
\begin{proof}
We actually will prove something slightly stronger. Take $W_S$ to be the space generated by the $\phi_{\bar{x}}$ for $\bar{x} \in \overline{Y}_S$. In light of \eqref{eq:bk_add}, we can prove additivity by showing that we can choose the $\phi_{\bar{x}}$ so that $\beta_{|S|}$ is injective on $W_S$.

This is clear for $S = \{i_a\}$. Now, suppose we had found $\phi_{\bar{y}}$ satisfying this property for all $\bar{y} \in \overline{Y}_T$ and proper subsets $T$ of $S$ that contain $i_a$, and we wish to prove the result for $S$. In light of Proposition \ref{prop:S_expns}, we certainly can find expansions $\phi_{\bar{x}}$ for each $\bar{x} \in \overline{Y}_S$. The only question is whether we can make the map from $W_S$ injective.

Take $M$ to be the narrow Hilbert class field of $K(X)$. For each prime $p$ that ramifies in $K(X)/\QQ$, choose $\mathfrak{P}$ to be a prime of $M$ over $p$, and take $\sigma_p$ to be the nontrivial inertia element corresponding to $\mathfrak{P}$.  By adjusting the $\phi_{\bar{x}}$, $\bar{x} \in \overline{Y}_S$ by the quadratic character $\chi_{\pm p}$ as needed, we can force $\phi_{\bar{x}}(\sigma_p) = 0$; we choose the sign for $\pm p$ to keep $\phi_{\bar{x}}$ unramified at $2$.

So suppose the set of $\phi_{\bar{x}}$ are zero at each $\sigma_p$. We claim that this is sufficient for $\beta_{|S|}$ to be injective on $W_S$.

Take $k = |S|$. Suppose the map were not injective, with $\beta_k \phi = 0$ for
\[\phi = \sum_j c_j \phi_{\bar{x}_j}\]
for some set of constants $c_j$. We have
\[0 = \beta_k \phi(\sigma_1, \dots, \sigma_k) = d\phi(\sigma_1, \tau) + d\phi(\tau, \sigma_1)\]
where $\tau$ is the iterated commutator of $\sigma_2, \dots, \sigma_k$. 

This splits into two cases depending on $k$. 
If $k > 2$, we always have that $d\phi(\tau, \sigma_1)$ is zero, so
\[\beta_k \phi = d\phi(\sigma_1, \tau) = \sum_j \sum_{i \in S - \{i_a\}} c_j \chi_{i, \bar{x}_j}(\sigma_1) \cdot \phi_{\bar{x}_j, \, S - \{i\}}(\tau).\]
Using the independence of the sets of characters corresponding to each $i$, we get that
\[\sum_j c_j \chi_{i, \bar{x}_j}(\sigma_1)  \cdot \phi_{\bar{x}_j, \, S - \{i\}}\big([\sigma_2,\, [ \dots, \,[\sigma_{k-1}, \,\sigma_k]\dots]]\big) = 0\]
for any choice of $i \in S - \{i_a\}$. This can be reexpressed as
\[\sum_j c_j\chi_{i, \bar{x}_j}(\sigma_1)  \cdot  \beta_{k-1}\phi_{\bar{x}_j, \, S - \{i\}}(\sigma_2, \dots, \sigma_k) = 0.\]
By the induction hypothesis, we thus have
\[\sum_j c_j\chi_{i, \bar{x}_j}(\sigma)  \cdot \phi_{\bar{x}_j, \, S - \{i\}}(\tau) = 0\]
for any choice of $\sigma$ and $\tau$. Taking coboundaries then gives
\[\sum_j c_j  \sum_{i  \in T \subseteq S - \{i_a\}} \chi_{T, \bar{x}_j}(\sigma) \cdot \phi_{\bar{x}_j,\, S - T}(\tau) = 0.\]
Again using the independence of these characters, we find
\[\sum_j c_j \chi_{T, \bar{x}_j}(\sigma) \cdot \phi_{\bar{x}_j,\, S - T}(\tau) = 0\]
for any $T \subseteq S - \{i_a\}$. Adding these together then gives that $d \phi = 0$.

On the other hand, if $k = 2$, we have $\tau = \sigma_2$, and we still find $d \phi = 0$.

Then
\[\sum_j c_j \phi_{\bar{x}_j}\]
is a Galois cocycle and hence corresponds to a quadratic extension of $\QQ$. But, from the $\phi(\sigma_p) =0$ conditions, we find that it is unramified at all finite primes, so $\phi = 0$. Then $\beta_{|S|}$ is injective on $W_S$, and this set of governing extensions is additive at level $S$. This gives the proposition by induction.
\end{proof}

There is one final result we need for sets of governing expansions. To prove our main theorems, we apply Chebotarev's density theorem to the composition of fields $M(\bar{x})$ over a special set of $\bar{x} \in \overline{X}_S$. For this reason, it is essential to have a sense of when a given field $M(\bar{x}_0)$ is not contained in the composition of all the other $M(\bar{x})$. The next proposition gives us the independence result we need.
\begin{prop}
\label{prop:gov_set_ind}
Take $X = X_1 \times \dots \times X_d$ to be a product of disjoint sets of odd primes, and take $i_a \in S \subseteq [d]$. For $i \in S$, take 
\[Z_i \subseteq X_i \times X_i\]
to be the set of edges of some ordered tree in $X_i$. Suppose we have a set of governing expansions on $X$ such that
\[\pi_S \left(\overline{Y}_S\right) \supseteq Z = \prod_{i \in S} Z_i.\]
For $z$ in the latter product, choose $\bar{x}(z)$ so $\pi_S(\bar{x}(z)) = z$. Then, for any choice of $z_0 \in \prod_{i \in S} Z_i$, and writing $\bar{z}_0 = \bar{x}(z_0)$, we have that
\[M_S(X)\prod_{z_0 \ne z \in Z} M(\bar{x}(z))\]
does not contain the field $M(\bar{z}_0)$, where $M_S(X)$ is as in the proof of the Proposition \ref{prop:gov_set_ex}.
\end{prop}
\begin{proof}
We need to check that $d\phi_{\bar{z}_0}$ is not in the span of the other $d\phi_{\bar{x}(z)}$ inside of 
\[H^2\big(\text{Gal}(M_S(X)/\QQ), \,\,\mathbb{F}_2\big).\]
Since $\text{Gal}(M_S(X)/\QQ)$ has nilpotence degree $|S| -1$, we see that the map
\[\beta(\psi)(\sigma_1, \dots, \sigma_k) \]
\[= \psi\big(\sigma_1,\,\, [\sigma_2,\, [ \dots, \,[\sigma_{k-1}, \,\sigma_k]\dots]]\big) + \psi\big([\sigma_2,\, [ \dots, \,[\sigma_{k-1}, \,\sigma_k]\dots]], \,\,\sigma_1\big)\]
is trivial on any $2$-coboundary, where we have taken $k =|S|$. That is, $\beta$ is defined on this cohomology group.

Then we just need to check that 
\[\beta(d\phi_{\bar{z}_0}) = \beta_k \phi_{\bar{z}_0}\]
is not in the span of the other $\beta_k \phi_{\bar{z}}$. Taking 
\[K_i(X) = \prod_{\bar{x} \in \overline{X}_{[d] - \{i\}}} K(\bar{x}),\]
we define $V_i$ to be the associated $\mathbb{F}_2$ vector space $\text{Gal}(K(X)/K_i(X))$. With this notation, we can consider $\beta_{|S|}\phi_{\bar{z}}$ restricted to
\[V_{i_1} \times \dots \times V_{i_{k - 1}} \times V_{i_a},\]
where $S = \{i_1, \dots, i_{k-1}, i_a\}$. Restricted to this domain, we find
\[\beta_{|S|} \phi_{\bar{z}} = \chi_{i_1,\, \bar{z}} \otimes \dots \otimes \chi_{i_{k-1},\, \bar{z}} \otimes \chi_{i_a,\, \bar{z}}\]
for any $z$ in $Z$. The tree assumption implies that, for any $i \in S$, the set
\[\big\{ \chi_{i, \bar{z}}\,:\,\, \bar{z} \in Z\big\}\] 
is a linearly independent set; that is, once all the duplicate entries are removed, the remaining characters are linearly independent. Since each $\bar{z}$ corresponds to a distinct tuple of characters, the structure of tensor products implies that $\phi_{\bar{z}_0}$ must be independent from the other $\phi_{\bar{z}}$. This proves the proposition.
\end{proof}

%%%%%%%
%%%%%%%

\subsection{Sets of raw expansions}
\label{ssec:raw_exp}
Take $N$ to be a $G_{\QQ}$ module that is isomorphic to some power of $\QQ_2/\Z_2$ if the $G_{\QQ}$ structure is forgotten. Take $X_1, \dots, X_d$ to be disjoint sets of odd primes where $N$ is not ramified, and take $X$ to be their product. For $x \in X$, we use $N(x)$ to denote quadratic twist of $N$ by the quadratic character of 
\[\QQ\left(\sqrt{\pi_1(x) \cdot \dots \cdot \pi_d(x)}\right)/\QQ.\]

Note that, for any $x \in X$,
\[N(x)[2] = N[2].\]
We write $\beta(x_0, x_1)$ for the isomorphism $N(x_0) \rightarrow N(x_1)$ that preserves Galois structure above 
\[K(x_0, x_1) = \QQ\big(\sqrt{\pi_1(x_0)\pi_1(x_1) \dots \pi_d(x_0) \pi_d(x_1)}\big).\]
Call the associated multiplicative quadratic character $\chi(x_0, x_1)$

For our next definition, we will need that $N$ contains a copy of $\Z/2\Z$.
\begin{defn*}
Given $N$ and $X$ as above, take
\[\text{rk}: X \rightarrow \Z^+ \cup \{\infty\}\]
to be any function. For $x$ in $X$ and $k \le \text{rk}(x)$ an integer, take
\[\psi_k(x) \in C^1\left(G_{\QQ},\, N(x)[2^k]\right),\]
where $C^1$ denotes the set of $1$-cocycles of $G_{\QQ}$ in $N(x)[2^k]$. This data will be called a \emph{set of raw cocycles} on $X$ if, for $x \in X$ and $k < \text{rk}(x)$, we have
\[2\psi_{k+1}(x) = \psi_k(x).\]
Given $S \subseteq [d]$, we will call $\mathfrak{R}$ \emph{consistent over} $S$ if
\[\psi_1(x) = \psi_1(x') \quad\text{whenever }x, x' \in X \text{ satisfy } \pi_{[d] - S}(x) = \pi_{[d] - S}(x')\]
Given $i_a \in S$, we call our set of raw cocycles \emph{$i_a$-consistent} over $S$ if there is some injection of Galois modules $\iota: \mathbb{F}_2 \rightarrow N[2]$ such that
\[\psi_1(x) - \psi_1(x') = \iota \circ \chi_{\pi_{i_a}(x)\pi_{i_a}(x')}.\]

We will use the letter $\mathfrak{R}$ to refer to a set of raw cocycles, writing $\text{rk}(\mathfrak{R})$ and $\psi_k(\mathfrak{R}, x)$ for the data associated to $\mathfrak{R}$. We will also use the notation $i_a(\mathfrak{R})$ and $\iota(\mathfrak{R})$ for the corresponding data of $i_a$-consistent $\mathfrak{R}$.
\end{defn*}

The goal of this subsection is to compare sets of raw cocycles with sets of governing expansions. In our final results for sets of raw cocycles, we will be interested in situations where 
\[i_a(\mathfrak{R})= i_a(\mathfrak{G}).\]
First, we look at the simpler situation where $i_a$ plays no role.
\begin{defn*}
Take $\mathfrak{R}$ to be a set of raw cocycles on $X$, and take $S$ to be a nonempty subset of $[d]$, and take $\bar{x} \in \overline{X}_S$.
Take $\bar{x} \in \overline{X}_S$, and suppose that $\text{rk}(\mathfrak{R})(x) \ge |S|$ for $x \in \widehat{x}(\emptyset)$. Choosing $x_0 \in \widehat{x}(\emptyset)$, we then define
\[\psi(\mathfrak{R},\, \bar{x}) = \sum_{x \in \widehat{x}(\emptyset)} \beta(x, x_0) \circ \psi_{|S|}(\mathfrak{R}, \, x).\]

Supposing that $\mathfrak{R}$ is consistent over $S$, we say that $\mathfrak{R}$ is \emph{minimal} at $\bar{x}$ if $\psi(\mathfrak{R},\, \bar{x}) = 0$.
\end{defn*}

We need one crucial calculation. For $i \in S$, take $H_i$ to be the subset of $x \in \widehat{x}(\emptyset)$ with $\pi_i(x_0) \ne \pi_i(x)$. For $T \subseteq S$, take
\[H_T = \bigcap_{i \in T} H_i.\]
We can write any $\sigma$ in $\text{Gal}(K(\bar{x})/\QQ)$ in the form.
\[\sigma = \sum_{i \in T_{\sigma}} \sigma_i,\]
where $\sigma_i$ is the unique nontrivial element of this Galois group that fixes 
\[\sqrt{\pi_0(\pi_j(\bar{x}))\pi_1(\pi_j(\bar{x}))}\]
for all $j \ne i$ in $S$, and where $T_{\sigma}$ is a subset of $S$. We claim that, for $x$ in $\widehat{x}(\emptyset)$,

\begin{equation}
\label{eq:to_hplanes}
\sum_{\substack{\emptyset \ne T \subseteq T_{\sigma} \\ H_T \ni x}} (-2)^{|T| - 1} = \begin{cases} 1 &\text{if } \chi(x, x_0)(\sigma) = -1 \\ 0 &\text{otherwise.}\end{cases}.
\end{equation}
For take $T_x$ to be the maximal $T$ so that $x \in H_T$. Then the right hand side is one if 
\[|T_x \cap T_{\sigma}|\] 
is odd. Calling this cardinality $m$, the left hand side is
\[\sum_{\emptyset \ne T \subseteq T_{\sigma} \cap T_x}  (-2)^{|T| - 1}\]
\[= \sum_{k = 1}^{m} \binom{m}{k}(-2)^{k-1} = \frac{1}{2}\big(1 - (-1)^m\big),\]
via the binomial theorem. This equals the right hand side, establishing \eqref{eq:to_hplanes}.

\begin{prop}
\label{prop:psi_coh}
Take $\mathfrak{R}$ to be a set of raw cocycles on $X$. Take $S$ to be a nonempty subset of $[d]$ where $\mathfrak{R}$ is consistent. Choose some $\bar{x} \in \overline{X}_S$ where $\psi(\bar{x}) = \psi(\mathfrak{R}, \, \bar{x})$ is defined.

Suppose that, for any $T \subsetneq S$ and any $\bar{x}_1 \in \widehat{x}(T)$, $\mathfrak{R}$ is minimal at $\bar{x}_1$. Then $\psi(\bar{x})$ maps into $N[2]$, and its coboundary is zero. That is, it corresponds to an element in $C^1(G_{\QQ}, N[2])$.
\end{prop}
\begin{proof}
It is clear that $2\psi(\bar{x})$ is zero by the minimality assumptions, so $\psi(\bar{x})$ maps into $N(x_0)[2]$. To show it is a coboundary, we calculate 
\begin{align*}
&d\big(\beta(x, x_0) \circ \psi_{|S|}(x)\big) (\sigma, \tau) \\
&\qquad\qquad= \big(\sigma \beta(x, x_0) - \beta(x, x_0) \sigma\big) \circ \psi_{|S|}(x)(\tau)\\
& \qquad\qquad= \begin{cases} \sigma \beta(x, x_0)\circ \psi_{|S| - 1}(x) (\tau) &\text{ if } \chi(x_0, x)(\sigma) = -1 \\ 0 &\text{ otherwise.}\end{cases}
\end{align*}
Then
\[d \psi(\bar{x})(\sigma, \tau) =  \sum_{\chi(x, x_0)(\sigma) = -1} \sigma\beta(x, x_0) \circ \psi_{|S| - 1}(x) (\tau) \]

Taking $T_{\sigma}$ as above, we can use \eqref{eq:to_hplanes} to write this sum as
\begin{equation}
\label{eq:raw_cobnd}
\sum_{\emptyset \ne T \subseteq T_{\sigma}} (-1)^{|T| - 1}\sigma \left(\sum_{x \in H_T} \beta(x, x_0) \circ \psi_{|S| - |T|}(x)(\tau)\right).
\end{equation}
But the inner sum is zero for each $T$ by the minimality hypothesis, so the coboundary is zero. This gives the proposition
\end{proof}
In light of the coboundary calculation of this proposition, we see that, if $\psi$ is minimal at $\bar{x} \in \overline{X}_S$, it is minimal at any $\bar{y} \in \widehat{x}(T)$ for any $T \subseteq S$.

We now start comparing sets of governing expansions with sets of raw cocycles.
\begin{defn*}
Take $\mathfrak{R}$ to be a set of raw cocycles on $X$, and take $\mathfrak{G}$ to be a set of governing expansions on $X$. Choose a subset $S$ of $[d]$, and choose $\bar{x} \in \overline{X}_S$.

If $\mathfrak{R}$ is $i_a(\mathfrak{G})$-consistent over $S$, 
we say that $\mathfrak{R}$ \emph{agrees} with $\mathfrak{G}$ at $\bar{x}$ if $\psi(\mathfrak{R},\, \bar{x})$ and $\phi_{\bar{x}}(\mathfrak{G})$ exist and
\[\psi(\mathfrak{R},\, \bar{x})\, - \,\iota(\mathfrak{R}) \circ \phi_{\bar{x}}(\mathfrak{G}) = 0.\]
If $S$ does not contain $i_a(\mathfrak{G})$ and if $\mathfrak{R}$ is consistent over $S$, we say that $\mathfrak{R}$ \emph{agrees} with $\mathfrak{G}$ at $\bar{x} \in \overline{X}_S$ if it is minimal at $\bar{x}$.
\end{defn*}

\begin{prop}
\label{prop:R_G_agree}
Take $\mathfrak{R}$ to be a set of raw cocycles on $X$, and take $\mathfrak{G}$ to be a set of governing expansions on $X$.  Choose $S \subseteq [d]$ so that $\mathfrak{R}$ is  $i_a(\mathfrak{G})$-consistent over $S$, and take $\bar{x} \in \overline{X}_S$ so that $\psi(\mathfrak{R},\, \bar{x})$ and $\phi_{\bar{x}}(\mathfrak{G})$ both exist. Suppose that, for any $T \subsetneq S$ and any $\bar{x}_1 \in \widehat{x}(T)$, $\mathfrak{R}$ agrees with $\mathfrak{G}$ at $\bar{x}_1$. Then
\[\psi(\bar{x})  - \iota \circ \phi_{\bar{x}}\, \in\, C^1\left(G_{\QQ}, \, N[2]\right)\]
\end{prop}
\begin{proof}

As before, $2\psi(\bar{x}) = 0$ by the minimality hypotheses, so we just need to check the cocycle condition. We can rewrite \eqref{eq:raw_cobnd} as
\[d\psi(\bar{x})(\sigma, \tau) = \sum_{\emptyset \ne T \subseteq S} \chi_{T,\,\bar{x}}(\sigma) \cdot \left(\sum_{x \in H_T} \beta(x, x_0) \circ \psi_{|S| - |T|}(x)(\tau)\right).\]
From the hypothesis on $\bar{x}_1$, we find that this equals
\[\iota \circ \sum_{i_a \not\in T \subseteq S} \chi_{T,\,\bar{x}}(\sigma) \cdot \phi_{\bar{x}_{S - T}}(\tau) =  \iota \circ d\phi_{\bar{x}}(\sigma, \tau).\]
Then $\psi(\bar{x}) - \iota \circ\phi_{\bar{x}}$ has zero coboundary, giving the proposition.
\end{proof}

%%%%%
%%%%%
\subsection{Raw expansions for class groups}
\label{ssec:class_exp}

Take $K/\QQ$ to be an imaginary quadratic field $\QQ(\sqrt{-n_0})$. Supposing $X_1, \dots, X_d$ are disjoint sets of odd primes that are unramified in this extension, we define
\[K(x) = \QQ\left(\sqrt{-n_0 \prod_{i \le d} \pi_i(x)}\right)\]
for $x \in X$. Throughout this section, we will presume that, for any $i \le d$, the value of $p \text{ mod }  4$ is the same for all $p$ in $X_i$. We will take $N(x)$ to be the module $\QQ_2/\Z_2$ twisted by the quadratic character corresponding to the extension $K(x)/\QQ$.

Choose $T_a \subseteq [d]$, and take $\Delta_a$ to be a squarefree integer dividing $2n_0$. From this information, we define a character $\psi_1(x): G_{\QQ} \rightarrow N[2]$ by 
\[\psi_1(x) = \chi_{\Delta_a} + \sum_{i \in T_a} \chi_{\pi_i(x)}.\]
We assume that the field of definition of $\psi_1(x)$ is unramified above $K(x)$ for all $x$. In this case, $\psi_1(x)$ corresponds to an element of the dual class group $\text{Cl}^{\vee} K(x)[2]$.

\begin{prop}
Take $\psi_1(x)$ as above, and take $K(x)^{\text{\emph{ur}}}$ to be the maximal extension of $K(x)$ that is unramified everywhere. Then, for $k > 0$, we have that
\[\restr{\psi_1(x)}{\text{\emph{Gal}}\left(\overline{\QQ}/K(x)\right)} \in 2^{k-1}\text{\emph{Cl}}^{\vee} K(x)[2^k]\]
if and only if, for some
\[\psi_k(x) \in C^1\left(\text{\emph{Gal}}(K(x)^{\text{\emph{ur}}}/\QQ),\, N(x)[2^k]\right),\]
we have
\[\psi_1(x) = 2^{k-1}\psi_k(x).\]
\end{prop}

\begin{proof}
We see that $\psi_k(x)$ restricted to the absolute Galois group of $K(x)$ is in $\text{Cl}^{\vee} K(x)[2^k]$, so the sufficiency of finding such a $\psi_k(x)$ is easy. Conversely, given a map
\[\psi_k(x)' \in \text{Cl}^{\vee}K(x)[2^k],\]
we know that the field of definition $L$ of $\psi_k(x)'$ is dihedral over $\QQ$, with its unique order $2^k$ cyclic subgroup corresponding to the intermediate field $K(x)$. To prove the converse, we need to extend the character $\psi_k(x)'$ from $\text{Gal}(L/K(x))$ to a cocycle $\psi_k(x)$ on $\text{Gal}(L/\QQ)$. Choosing some  $F$ in this Galois group so that we have a coset decomposition
\[\text{Gal}(L/\QQ) = \text{Gal}(L/K(x)) \,+\, F\cdot \text{Gal}(L/K(x)),\]
and choosing some $\alpha \in N(x)$ with $2^{k-1}\alpha = \psi_1(F)$, we can define such a $\psi_k(x)$ by setting
\[\psi_k(x)(\sigma) = \psi_k(x)'(\sigma)\quad\text{and}\quad \psi_k(x)(F\cdot\sigma) =  \alpha  -\psi_k(x)'\]
for all $\sigma \in \text{Gal}(L/K(x))$. We can verify that $\psi_k(x)$ obeys the cocycle condition, giving the proposition.
\end{proof}

In light of this, we define
\[\overline{\text{Cl}}^{\vee} K(x)[2^k] \,=\, C^1\left(\text{\Gal}(K(x)^{\text{ur}}/\QQ),\, N(x)[2^k]\right).\]
We always have
\[\overline{\text{Cl}}^{\vee} K(x)[2^k] \, \cong\, \text{Cl}^{\vee} K(x)[2^k] \oplus (\Z/2^k\Z.)\]
For $w_a = (T_a,  \,\Delta_a)$ corresponding to an element of $\overline{\text{Cl}}^{\vee}K(x)[2^k]$, we define $\mathfrak{R}(w_a)$ to be a set of raw cocycles on $X$ so that, for all $x \in X$,
\[\psi_1(\mathfrak{R},\, x) = \psi_1(x)\]
and so that $\text{rk}(\mathfrak{R})(x)$ is the maximal integer $k$ such that $\psi_1(x)$ corresponds to an element of
\[2^{k-1} \overline{\text{Cl}}^{\vee} K(x)[2^k],\]
with 
\[\psi_k(\mathfrak{R},\, x) \,\in\,\overline{\text{Cl}}^{\vee} K(x)[2^k]\]
whenever the left hand side is defined.

Now, take $w_b = (T_b, \, \Delta_b)$, where $T_b$ is any subset of $[d]$ and $\Delta_b$ is a positive squarefree divisor of $n_0$ (or, if $K(x)$ has even discriminant, $2n_0$). For any $x \in X$, we define an ideal $w_b(x)$ of the integers of $K(x)$ by
\[\prod_{p | \Delta_b} \mathfrak{P}(p) \cdot \prod_{i \in T_b} \mathfrak{P}\big(\pi_i(x)\big)\]
where $\mathfrak{P}(p)$ is the unique prime dividing $p$ in $K(x)$. Taking $\overline{\text{Cl}}\, K(x)[2]$ to be the set of ideals with squarefree norm dividing the discriminant of $K(x)/\QQ$, we see that the map
 \[\overline{\text{Cl}}\, K(x)[2] \rightarrow \text{Cl}\, K(x)[2]\]
is a surjective and has kernel isomorphic to $\Z/2\Z$. We write $2^{k-1}\overline{\text{Cl}}\,K(x)[2^k]$ for the preimage of $2^{k-1}\text{Cl}\,K(x)[2^k]$ under this map.

For a cocycle $\psi_k$, write $L(\psi_k)$ for the field of definition of $\psi_k$ over $K(x)$. If $\psi_k(x)$ exists, we see that the Artin symbol
\[\left[\frac{L(\psi_k)/K(x)}{\mathfrak{P}}\right]\]
lies in the order $2$ subgroup of $\text{Gal}(L(\psi_k)/K(x))$ at any $\mathfrak{P}$ dividing the discriminant of $K(x)/\QQ$. Identifying this subgroup with $\Z/2\Z$, we have the following result.

\begin{thm}
\label{thm:Cl_gov}
Take $X$ and $n_0$ as above, and choose $w_a$ and $w_b$ as above that correspond to elements of $\overline{\text{\emph{Cl}}}^{\vee} K(x)[2]$ and $\overline{\text{\emph{Cl}}}\,K(x)[2]$ respectively. Choose $S \subseteq [d]$ of cardinality at least three, and choose $\bar{x} \in \overline{X}_S$. Take $\mathfrak{G}$ to be a set of governing expansions on $X$, writing $i_a = i_a(\mathfrak{G})$, and take $\mathfrak{R}$ to be the set of raw cocycles $\mathfrak{R}(w_a)$. We assume $i_a$ is in $S$.

We next assume that
\[w_b(x) \in 2^{|S|-2}\overline{\text{\emph{Cl}}}\, K(x)[2^{|S|-1}] \quad\text{for all } x \in \widehat{x}(\emptyset)\]
and that there is some $i_b \in S$ other than $i_a$ so that
\[S \cap T(w_b) \subseteq \{i_b\} \quad\text{and}\quad S \cap T(w_a) \subseteq \{i_a\}.\]
Take $i_{ab}$ to equal $i_a$ if $T(w_b)$ does not meet $S$, otherwise taking $i_{ab}$ to equal $i_b$. For $i$ in $S$ other than $i_{ab}$, choose $\bar{z}_i$ in $\widehat{x}(S - \{i\})$.

\begin{enumerate}
\item Suppose either that $T(w_a)$ does not meet $S$ or that $T(w_b)$ does not meet $S$. Assume that, for each $i$ in  $S$ other than $i_{ab}$ and each $\bar{y}$ in  $\widehat{z}_i(S - \{i, i_{ab}\})$, we have that $\mathfrak{R}$ is minimal at $\bar{y}$. Then $\psi_{|S|-1}(\mathfrak{R},\, x)$ exists for all $x \in \widehat{x}(\emptyset)$ and
\[\sum_{x \in \widehat{x}(\emptyset)} \left[\frac{L\left(\psi_{|S| - 1}(\mathfrak{R},\, x)\right)/K(x)}{w_b(x)} \right] = 0.\]
\item Now assume that both $T(w_a)$ and $T(w_b)$ meet $S$. Choose $\bar{z}$ in $\widehat{x}(S - \{i_b\})$, and assume that $\phi_{\bar{z}}(\mathfrak{G})$ exists. Assume further that, for every $i \in S$ other than $i_b$ and each $\bar{y}$ in  $\widehat{z}_i(S - \{i, i_{b}\})$, we have that $\mathfrak{R}$ agrees with $\mathfrak{G}$ at $\bar{y}$.  Then $\psi_{|S|-1}(\mathfrak{R},\, x)$ exists for all $x \in \widehat{x}(\emptyset)$. Furthermore, writing 
\[(p_{0b},\, p_{1b}) = \pi_{i_b}(\bar{x}),\]
we have
\[\sum_{x \in \widehat{x}(\emptyset)} \left[\frac{L\left(\psi_{|S| - 1}(\mathfrak{R},\, x)\right)/K(x)}{w_b(x)} \right] =  \phi_{\bar{z}}(\mathfrak{G}) \big(\text{\emph{Frob}}(p_{0b}) \cdot \text{\emph{Frob}}(p_{1b})\big). \]

\end{enumerate}
\end{thm}
\begin{proof}
For both parts, choose $\bar{z}$ in $\widehat{x}(S - \{i_{ab}\})$. For $x \in \widehat{z}(\emptyset)$, note that the assumption on $w_b(x)$ means that
 \[\left[\frac{L\left(\psi_{|S| - 1}(\mathfrak{R},\, x)\right)/K(x)}{w_b(x)} \right]\]
 depends only on $w_a$, $w_b$, and $x$, and not on the choice of raw cocycles $\mathfrak{R}$.

Write $b$ for the norm of the ideal $w_b(x)$ for any $x \in\widehat{z}(\emptyset)$; our restrictions on $T_b$ mean that $b$ does not depend on $x$. Take $p$ to be a prime divisor of $b$, and take $\mathfrak{P}$ to be the prime dividing $p$ in $K(x)$. Write $\Delta(x)$ for the discriminant of $K(x)/\QQ$. Assuming $\psi_{|S| - 1}(\mathfrak{R},\, x)$ exists, we can write it locally at $p$ as $\chi$ or $\chi + \chi_{\Delta(x)}$, where $\chi$ is unramified. We then can say
\[\left[\frac{L\left(\psi_{|S| - 1}(\mathfrak{R},\, x)\right)/K(x)}{\mathfrak{P}} \right] = \text{inv}_p(\chi \cup \chi_b).\]
We also have $\text{inv}_p (\chi_{\Delta(x)} \cup \chi_b) = 0$ from our requirements on $w_b$, so we find
\begin{equation}
\label{eq:Art_to_inv}
\left[\frac{L\left(\psi_{|S| - 1}(\mathfrak{R},\, x)\right)/K(x)}{\mathfrak{P}} \right] = \text{inv}_p(\psi_{|S|-1}(x) \cup \chi_b).
\end{equation}

Take $x_0$ to be the element of $\widehat{z}(\emptyset)$ outside of all the sets $\widehat{z}_i(\emptyset)$, and write $\bar{y}_i$ for the element in $\widehat{z}(S - \{i, i_{ab}\}) \cap \widehat{z}_i(S - \{i, i_{ab}\})$. For the first part, consider
\[\psi = - \sum_{x \in \widehat{z}(\emptyset) - \{x_0\}} \beta(x, \,x_0) \circ \psi_{|S|-1}(\mathfrak{R},\, x).\]
From Proposition \ref{prop:psi_coh}, we know that this is a cocycle mapping to $N(x_0)$, and we find
\[2^{|S|-2}\psi = \psi_1(x_0).\]
From the minimality assumption, we have
\[2\psi  =   -\sum_{x \in \widehat{z}(\emptyset) - \widehat{y}_i(\emptyset) -\{ x_0\}} \beta(x, \,x_0) \circ \psi_{|S|-2}(\mathfrak{R}, \, x)\]
for each $i \in S - \{i_{ab}\}$. From this, we must have that the field of definition of $2\psi$ is unramified at each $\pi_i(\bar{z}_i)$ for $i \in S - \{i_{ab}\}$. Then $2\psi$ must have field of definition unramified above $K(x_0)$, so some quadratic twist of $\psi$ is unramified above $K(x_0)$, and $\psi_{|S|-1}(\mathfrak{R}, x_0)$ exists. Then, via \eqref{eq:Art_to_inv}, we find
\[\sum_{x \in \widehat{z}(\emptyset)} \left[\frac{L\left(\psi_{|S| - 1}(\mathfrak{R},\, x)\right)/K(x)}{w_b(x)} \right]  =\sum_{p |b}   \text{inv}_p\big(\psi(\bar{z}) \cup \chi_b\big).\]
The assumption on $w_b$ means the choice of $\psi_{|S|-1}(\mathfrak{R}, x_0)$ does not affect the value of this sum, so we can take $\psi(\bar{z})$ to be a quadratic character. By Hilbert reciprocity, this equals
\[\sum_{p \nmid b}   \text{inv}_p\big(\psi(\bar{z}) \cup \chi_b\big).\]
But $\chi_b$ is locally trivial at all primes ramifying in any $K(x)$ that do not divide $b$, so this is zero. This gives the first part of the theorem.

For the second part, we instead take
\[\psi =\iota \circ \phi_{\bar{z}} - \sum_{x \in \widehat{z}(\emptyset) - \{x_0\}} \beta(x, \,x_0) \circ \psi_{|S|-1}(\mathfrak{R},\, x).\]
From Proposition \ref{prop:R_G_agree}, we see that this is a cocycle mapping to $N(x_0)$, and we again find $2^{|S|-2}\psi = \psi_1(x_0)$. Furthermore, we have
\[2\psi = \iota \circ \phi_{\bar{y}_i} - \sum_{x \in \widehat{z}(\emptyset)- \widehat{y}_i(\emptyset) - \{x_0\}}  \beta(x, \,x_0) \circ \psi_{|S|-2}(\mathfrak{R}, \, x)\]
for each  $i \in S - \{i_b\}$, where we are taking $\phi_{\bar{y}_{i_a}} = 0$. Then $2\psi$ must have field of definition unramified above $K(x_0)$. Then $\psi_{|S|-1}(\mathfrak{R}, x_0)$ exists and can be taken to be a quadratic twist of this $\psi$. Following the logic of the first part, we can ignore the quadratic twist, and we find
\[\sum_{x \in \widehat{z}(\emptyset)} \left[\frac{L\left(\psi_{|S| - 1}(\mathfrak{R},\, x)\right)/K(x)}{w_b(x)} \right] =  \sum_{p | b} \text{inv}_p(\phi_{\bar{z}} \cup \chi_b).\]
Repeating this for the other $\bar{z} \in \widehat{x}(S - \{i_b\})$, we find
\[\sum_{x \in \widehat{x}(\emptyset)} \left[\frac{L\left(\psi_{|S| - 1}(\mathfrak{R},\, x)\right)/K(x)}{w_b(x)} \right]  = \text{inv}_{p_{0b}}\big(\phi_{\bar{z}} \cup \chi_{p_{0b}}\big) +  \text{inv}_{p_{1b}}\big(\phi_{\bar{z}} \cup \chi_{p_{1b}}\big),\]
a synonym for what is claimed. This gives the part and the theorem.
 \end{proof}

%%%%%%
%%%%%%
\subsection{Raw expansions for Selmer Groups}
\label{ssec:Sel_exp}
Take $E/\QQ$ to be an elliptic curve with full rational $2$-torsion; that is to say, we have an isomorphism of Galois modules
\[E[2] \cong (\Z/2\Z)^2\]
defined over $\QQ$. Take $N_0$ to be the conductor of $E$, and take $X_1, \dots, X_d$ to be disjoint sets of odd primes not dividing $N_0$. We assume that, for each $i \le d$, the value of $p \text{ mod } 4$ is the same for all $p \in X_i$. We also assume that $E[4]$ has no order four cyclic subgroup defined over $\QQ$.

We define
\[E(x) = E^{(p_1\cdot \dots \cdot p_d)} \quad \text{where} \quad p_i = \pi_i(x),\]
with $E^{(n)}$ denoting the quadratic twist of $E$ in $\QQ(\sqrt{n})$. The $2^k$-Selmer group of $E(x)$ is defined to be
\[\text{Sel}^{\, 2^k}(E(x)) = \ker \left(H^1\big(G_{\QQ}, E[2^k]\big) \longrightarrow \prod_v H^1\big(\text{Gal}(\overline{\QQ}_v/\QQ_v), E\big)\right),\]
the product being over all rational places $v$. Our main group of study will instead be the corresponding set of cocycles
\[\ker \left(C^1\big(G_{\QQ}, E[2^k]\big) \longrightarrow \prod_v H^1\big(\text{Gal}(\overline{\QQ}_v/\QQ_v), E\big)\right),\]
a group we will denote by $\overline{\text{Sel}}^{\,2^k}E(x)$. In our case, if we write $\text{Sel}^{2^k}E(x)$ in the form $\text{Im}(E[2]) \oplus H$, we can find a corresponding isomorphism
\[\overline{\text{Sel}}^{\,2^k}E(x) \cong (\Z/2^k\Z)^2 \oplus H.\]
In particular, $\overline{\text{Sel}}^{\,2}E(x)$ equals $\text{Sel}^{\,2}E(x)$.

Writing  $E[2] \cong (\Z/2\Z)e_1 + (\Z/2\Z)e_2$, we have a (non-canonical) isomorphism
\[H^1\left(G_{\QQ}, \,E[2]\right) \cong H^1\left(G_{\QQ}, \Z/2\Z\right) \times H^1\left(G_{\QQ}, \Z/2\Z\right).\]
From this, we can write any $2$-Selmer element of $E(x)$ as a pair of quadratic characters $(\chi_1, \chi_2)$, with the $\chi_i$ ramified only at bad primes of $E(x)$. As in the previous section, each $\chi_i$ corresponds  to a choice of divisor $\Delta_i$ of $2N_0$ and a choice of subset $T_i$ of $[d]$. We will use the letter $w$ to denote a choice of tuple $(T_1, T_2, \Delta_1, \Delta_2)$ and write $w(x)$ for the cocycle in
\[C^1\left(G_{\QQ},\, E(x)[2]\right)\]
corresponding to $w$ at $x$.

Taking
\[N(x) = E(x)[2^{\infty}],\]
we define $\mathfrak{R}(w)$ to be a set of raw cocycles for which
\[\psi_1(\mathfrak{R},\, x) = w(x) \quad\text{for all} x \in X\]
and for which $\text{rk}(\mathfrak{R})(x)$ is maximum of one and the maximal $k$ such that $w(x)$ is in
\[2^{k-1}\overline{\text{Sel}}^{\,2^k}E(x),\]
with $\psi_k(\mathfrak{R},\, x)$ lying in $\overline{\text{Sel}}^{\,2^k}E(x)$ whenever it is defined for all $k \ge 2$.

There is a natural alternating pairing defined on the Selmer group called the Cassels-Tate pairing; Milne's standard text is our reference for the pairing's construction \cite{Milne86}. Suppose $w_a(x)$ and $w_b(x)$ are both $2$-Selmer elements at $x$. Suppose further that $\text{rk}(\mathfrak{R}(w_a))(x)$ is at least $k$, and take $\psi = \psi_k(\mathfrak{R}(w_a),\, x)$. Take $\psi'$ to be any map from $G_{\QQ}$ to $E(x)$ satisfying $2\psi' = \psi$, and take $\epsilon$ to be a $2$-cochain from $G_{\QQ}$ to $\Z/2\Z$ satisfying
\[d\epsilon = d\psi' \cup w_b(x),\]
 where the cup product comes from the natural Weil pairing on $E[2]$. Finally, for each rational place $v$, take 
 \[\psi^{\circ}_v \in \text{ker}\bigg(C^1\big(\text{Gal}(\overline{\QQ}_v/\QQ_v),\, E[2^{k+1}]\big) \rightarrow H^1(\text{Gal}\big(\overline{\QQ}_v/\QQ_v),\,E\big)\bigg)\]
 satisfying $2\psi^{\circ}_v = \psi_v$. Then we can define the Cassels-Tate pairing as
\[\big\langle \psi_k(\mathfrak{R}(w_a),\, x),\, w_b(x)\big\rangle_{CT} = \sum_v \text{inv}\big((\psi^{\circ}_v - \psi'_v) \cup w_b(x) + \epsilon_v \big).\]

We have one final piece of notation:
\begin{defn*}
Given $E$ and $X$ as above, and given $S \subseteq [d]$ and $\bar{x} \in \overline{X}_S$, we call $\bar{x}$ quadratically consistent if, for all $i \in S$ and $x \in \widehat{x}(\emptyset)$, we have that
\[\pi_0(\pi_i(\bar{x}))\pi_1(\pi_i(\bar{x}))\]
is a quadratic residue at $2$ and at all the bad primes of $E(x)$ besides $\pi_i(x)$. (We also define this for $K$ as in the previous section, where we require the above to be a quadratic residue at $2$ and all ramified primes of $K(x)/\QQ$ besides the $\pi_i(x)$).
\end{defn*}
\begin{thm}
\label{thm:Sel_gov}
Take $E/\QQ$ and $X$ as above.  Choose $S \subseteq [d]$ of cardinality at least three, choose some quadratically consistent $\bar{x} \in \overline{X}_S$, and choose tuples $w_a$ and $w_b$ as above corresponding to $2$-Selmer elements of $E(x)$ for $x \in \widehat{x}(\emptyset)$. Take $\mathfrak{G}$ to be a set of governing expansions on $X$, writing $i_a = i_a(\mathfrak{G})$, and take $\mathfrak{R}$ to be the set of raw cocycles $\mathfrak{R}(w_a)$. We assume $i_a$ is in $S$.

We next assume that
\[w_b(x) \in 2^{|S| - 3} \overline{\text{\emph{Sel}}}^{\,2^{|S|-2}} E(x)\quad \text{for all } x \in \widehat{x}(\emptyset)\]
and that there is some $i_b \in S$  other than $i_a$ so that
\[T_1(w_a) \cap S \subseteq \{i_a\}, \quad T_2(w_b) \cap S \subseteq  \{i_b\},\]
\[\text{and}\quad T_2(w_a) \cap S = T_1(w_b) \cap S = \emptyset.\]
Take $i_{ab}$ to equal $i_a$ if $T(w_b)$ does not meet $S$, otherwise taking $i_{ab}$ to equal $i_b$. For $i$ in $S$ other than $i_{ab}$, choose $\bar{z}_i$ in $\widehat{x}(S - \{i\})$.

\begin{enumerate}
\item Suppose either that $T_1(w_a)$ does not meet $S$ or that $T_2(w_b)$ does not meet $S$. Assume that, for each $i$ in $S$ other than $i_{ab}$ and each $\bar{y}$ in $\widehat{z}_i(S - \{i, i_{ab}\})$, we have that $\mathfrak{R}$ is minimal at $\bar{y}$. Then $\text{\emph{rk}}(\mathfrak{R})(x) \ge |S| - 2$ for each $x \in \widehat{x}(\emptyset)$, and
\[\sum_{x \in \widehat{x}(\emptyset)} \big\langle \psi_{|S|-2}(\mathfrak{R},\, x),\, w_b(x)\big\rangle_{CT} = 0.\]
\item Now assume that $T_1(w_a)$ and $T_2(w_b)$ both meet $S$. Choose $\bar{z} \in \widehat{x}(S - \{i_b\})$, and assume that $\phi_{\bar{z}}(\mathfrak{G})$ exists. Assume further that, for every $i \in S$ other than $i_b$ and each $\bar{y}$ in $\widehat{z}_i(S - \{i, i_b\})$, we have that $\mathfrak{R}$ agrees with $\mathfrak{G}$ at $\bar{y}$.  Then $\text{\emph{rk}}(\mathfrak{R})(x) \ge |S| - 2$ for each $x \in \widehat{x}(\emptyset)$. Furthermore, writing 
\[(p_{0b},\, p_{1b}) = \pi_{i_b}(\bar{x}),\]
we have
\[\sum_{x \in \widehat{x}(\emptyset)} \big\langle \psi_{|S|-2}(\mathfrak{R},\, x),\, w_b(x)\big\rangle_{CT} = \phi_{\bar{z}_{i_b}}(\mathfrak{G}) \big(\text{\emph{Frob}}(p_{0b}) \cdot \text{\emph{Frob}}(p_{1b})\big).\]
\end{enumerate}

\end{thm}
\begin{proof}
For both parts, choose $\bar{z} \in \widehat{z}(S - \{i_{ab}\})$. Take $x_0$ to be the element of $\widehat{z}(\emptyset)$ outside of all the sets $\widehat{z}_i(\emptyset)$, and write $\bar{y}_i$ for the element in $\widehat{z}(S - \{i, i_{ab}\}) \cap \widehat{z}_i(S - \{i, i_{ab}\})$. For the first part, consider
\[\psi = -\sum_{x \in \widehat{z}(\emptyset) - \widehat{y}_i - \{x_0\}} \beta(x, \,x_0) \circ \psi_{|S| - 2}(\mathfrak{R},\, x)\]
for $i \in S - \{i_{ab}\}$. The minimality hypotheses mean that this does not depend on the choice of $i$, and Proposition \ref{prop:psi_coh} implies that it is a cocycle with values in $N(x_0)$. We also see that $2^{|S|-3} \psi = w_a(x_0)$. 

From quadratic consistency, we know that $\beta(x, x_0)$ is an isomorphism locally at $\infty$ and at each prime that is simultaneously bad for $E(x)$ and $E(x_0)$. Because of this, for each $i \in S - \{i_{ab}\}$, we can show that $\psi$ obeys local conditions at all places other than at the primes in $\pi_j(\bar{x})$ with $j \in S - \{i\}$. By varying $i$, we then find that $\psi$ is a Selmer element, so $\text{rk}(\mathfrak{R})(x_0) \ge |S| - 2$, and we rechoose $\mathfrak{R}$ at $x_0$ to make $\mathfrak{R}$ minimal at all $\bar{y} \in \widehat{z}(S - \{i, i_{ab}\})$ for each $i$ in $S - \{i_{ab}\}$. Modifying $\mathfrak{R}$ does not affect the Cassels-Tate pairing by the assumptions on $w_b$, so rechoosing $\mathfrak{R}$  in this way will not affect the sum we are calculating.

Next, choose an $\epsilon(x)$ and a $\psi'(x)$ above $\psi_{|S|-2}(x)$  at all $x \in \widehat{z}(\emptyset) - \{x_0\}$ as in the definition of the Cassels-Tate pairing. Then take
\[\psi'(x_0) = -\sum_{x \in \widehat{z}(\emptyset) - \{x_0\}} \beta(x, \, x_0) \circ \psi'(x) \,\,\text{ and}\]
\[\epsilon(x_0) = -\sum_{x \in \widehat{z}(\emptyset) - \{x_0\}}  \epsilon(x).\]
Via the coboundary calculation of Proposition \ref{prop:psi_coh}, we find that
\[d\epsilon(x_0) = d\psi'(x_0) \cup w_b(x_0).\]

For the second part, we instead take
\[\psi =\iota \circ \phi_{\bar{y}_i} -  \sum_{x \in \widehat{z}(\emptyset) - \widehat{y}_i - \{x_0\}} \beta(x, \,x_0) \circ \psi_{|S| - 2}(\mathfrak{R},\, x).\]
for $i \in S - \{i_b\}$. As before, this is a cocycle with values in $N(x_0)$ that satisfies $2^{|S|-3} \psi = w_a(x_0)$, and the agreement hypotheses mean that it does not depend on the choice of $i$. The $\phi_{\bar{y}_i}$ are locally trivial at the primes of $\pi_i(\bar{z})$ since $\phi_{\bar{z}}$ exists, so we again find that $\psi$ is a Selmer element for $E(x_0)$. Then $\text{rk}(\mathfrak{R})(x_0) \ge |S| - 2$, and we rechoose $\mathfrak{R}$ at $x_0$ to make $\mathfrak{R}$ minimal at all $\bar{y} \in \widehat{z}(S - \{i, i_{ab}\})$ for each $i$ in $S - \{i_{ab}\}$.

Next, choose an $\epsilon(x)$ and a $\psi'(x)$ above $\psi_{|S|-2}(x)$  at all $x \in \widehat{z}(\emptyset) - \{x_0\}$ as in the definition of the Cassels-Tate pairing, and take
\[\psi'(x_0) = -\sum_{x \in \widehat{z}(\emptyset) - \{x_0\}} \beta(x, \, x_0) \circ \psi'(x) \,\,\text{ and}\]
\[\epsilon(x_0) =  \phi_{\bar{z}}(\mathfrak{G}) \cup \chi_2(w_b)(x_0)\, -\sum_{x \in \widehat{z}(\emptyset) - \{x_0\}}  \epsilon(x).\]
From a coboundary calculation as in Proposition \ref{prop:R_G_agree}, we find that
\[d\epsilon(x_0) = d\psi'(x_0) \cup w_b(x_0).\]

For both parts, we know that $\beta(x,\, x_0)$ is locally an isomorphism at $\infty$ and at each prime that is simultaneously bad for $E(x)$ and $E(x_0)$. At such places, and at all simultaneously good places, we set
\[\psi^{\circ}_v(x_0) = -\sum_{x \in \widehat{z}(\emptyset) - \{x_0\}} \beta(x, \, x_0) \circ \psi^{\circ}_v(x).\]
At a place of the form $\pi_0(\pi_i(\bar{z}))$ or $\pi_1(\pi_i(\bar{z}))$, we instead choose some vertex $y_0$ of each $\bar{y} \in \widehat{z}(S - \{i\})$ to define
\[\psi^{\circ}_v(y_0) = -\sum_{x \in \widehat{y}(\emptyset) - \{y_0\}} \beta(y, \, y_0) \circ \psi^{\circ}_v(y).\]
Minimality implies that these choices for $\psi^{\circ}_v$ have the properties required of them.

For the first part, we see that $w_b(x)$ is constant on $\widehat{z}(\emptyset)$, while the sums $\psi_1$, $\psi^{\circ}_v$, and $\epsilon$ over this all sum to zero. This gives the result on the sum of Cassels-Tate pairings.

For the second part, we instead find
\[\sum_{x \in \widehat{z}(\emptyset)} \big\langle \psi_{|S|-2}(\mathfrak{R},\, x),\, w_b(x)\big\rangle_{CT} = \sum_{v\text{ bad for some } E(x)} \text{inv}_v\big(\phi_{\bar{z}}(\mathfrak{G}) \cup \chi_2(w_b)(x_0)\big).\]
Repeating this for the other $\bar{z} \in \widehat{x}(S - \{i_b\})$ then gives the theorem.

\end{proof}

\section{Additive-Restrictive systems}
\label{sec:add_res}

We now introduce the notion of an additive-restrictive system, a construction that abstracts some of the details for sets of governing expansions and sets of raw cocycles. Take $X$ to be a product of disjoint sets $X_1, \dots, X_d$, and take all other notation as at the beginning of Section \ref{sec:alg1}.

\begin{defn}
\label{def:add_res}
An \emph{additive-restrictive system} is a sequence of objects 
\[(\overline{Y}_S,\, \overline{Y}_S^{\,\circ},\, F_S,\, A_S)\]
 indexed by $S \subseteq [d]$ so that
\begin{itemize}
\item For each $S \subseteq [d]$, $A_S$ is an abelian group, $\overline{Y}_S$ and $\overline{Y}^{\,\circ}_S$ are sets satisfying
\[\overline{Y}_S^{\,\circ} \subseteq \overline{Y}_S \subseteq \overline{X}_S,\]
and $F_S$ is a function
\[F_S: \overline{Y}_S \rightarrow A_S\]
with kernel $\overline{Y}_S^{\,\circ}$.
\item If $S$ is nonempty,
\[\overline{Y}_S = \big\{\bar{x} \in \overline{X}_S \,:\,\, \widehat{x}(T) \subset \overline{Y}^{\,\circ}_T \text{ for all } T \subsetneq S \big\}.\]
\item \emph{(Additivity)} Choose $s \in S$, and suppose $\bar{x}_1, \bar{x}_2, \bar{x}_3$ are elements of $\overline{Y}_S$ satisfying
\[\pi_{[d] - \{s\}}(\bar{x}_1) = \pi_{[d] - \{s\}}(\bar{x}_2) = \pi_{[d] - \{s\}}(\bar{x}_3)\]
and
\[\pi_s(\bar{x}_1) = (p_1, p_2),\quad \pi_s(\bar{x}_2) = (p_2, p_3),\quad \pi_s(\bar{x}_3) = (p_1, p_3)\]
for some $p_1, p_2, p_3 \in X_s$.
Then
\[F_S(\bar{x}_1) + F_S(\bar{x}_2) = F_S(\bar{x}_3).\]
\end{itemize}

We will use the letter $\mathfrak{A}$ to denote an additive-restrictive system, writing $\overline{Y}_S(\mathfrak{A})$, $F_S(\mathfrak{A})$, etc. to denote the data associated with $\mathfrak{A}$.
\end{defn}

The crucial property of additive-restrictive sequences is that we can bound how quickly the sets $\overline{Y}_S^{\,\circ}$ shrink as $S$ increases. We do this with the following proposition.
\begin{prop}
\label{prop:ars_density}
Suppose $X = X_1 \times \dots \times X_d$ is a product of finite sets, and suppose
\[\big((\overline{Y}_S, \overline{Y}^{\,\circ}_S, F_S, A_S)\,: \,\, S \subset [d]\big)\]
is an additive-restrictive system on $X$. Write $\delta$ for the density of $\overline{Y}^{\,\circ}_{\emptyset}$ in $X$, and write $|A|$ for the maximum size of a group $A_S$. Then, for any $S \subseteq [d]$, the density of $\overline{Y}^{\,\circ}_{S}$ in $\overline{X}_{S}$ is at least
\[\delta^{2^{|S|}}|A|^{-3^{|S|}}.\]
\end{prop}
\begin{proof}
Write $\delta_T$ for the density of $\overline{Y}^{\,\circ}_T$ in $\overline{X}_T$. For $s \in S$ and $\bar{x}_0 \in \overline{X}_S$, define
\[M(\bar{x}_0) = \pi^{-1}_{[d] - \{s\}} \big(\pi_{[d] - \{s\}} (\bar{x}_0)\big)\]
and consider
\[V =\overline{Y}^{\,\circ}_{S - \{s\}} \,\cap\, M(\bar{x}_0)\]
and
\[W = \overline{Y}^{\,\circ}_S  \,\,\cap\,\,  M(\bar{x}_0).\]
We see that $W$ naturally injects into $V \times V$. Furthermore, by the additivity of our additive-restrictive sequence, $W$ takes the form of an equivalence relation on $V$. Given $\bar{x}_1, \bar{x}_2$ in $V$ and $T$ a subset of $S$ containing $s$, write
\[\bar{x}_1 \sim_T \bar{x}_2\]
if 
\begin{itemize}
\item $\bar{x}_1 \sim_{T'} \bar{x}_2$ for all proper subsets $T'$ of $T$ that contain $s$, and
\item $F_T$ is zero on all elements of $\widehat{x}(T)$ if $\bar{x}$ satisfies
\[\widehat{x}(S - \{s\}) = \{\bar{x}_1,\, \bar{x}_2\}.\]
\end{itemize}

The relation $\sim_S$ splits $V$ into
\[\prod_{s \in T \subseteq S} \big| A_S \big|^{2^{|S| - |T|}} \le \prod_{i = 0}^{|S| - 1} |A|^{\binom{|S| - 1}{i}2^i} = |A|^{3^{|S| - 1}}\]
equivalence classes, and $W$ describes this equivalence relation.

Write $\delta_{\bar{x}_0}$ for the density of $V$ in $\overline{X}_{S - \{s\}} \cap M(\bar{x}_0)$. Then the density of $V \times V$ in $\overline{X}_{S} \cap M(\bar{x}_0)$ is $\delta_{\bar{x}_0}^2$, and the density of $W$ in this space is then at least
\[|A|^{-3^{|S| - 1}} \cdot \delta_{\bar{x}_0}^2\]
The average of the $\delta_{\bar{x}_0}$ is $\delta_{S - \{s\}}$, and $\overline{Y}^{\,\circ}_S$ is given by the union of the $W$ over all $\bar{x}_0$, so Cauchy's inequality gives
\[\delta_S \ge |A|^{-3^{|S| - 1}} \cdot \delta_{S - \{s\}}^2.\]
Repeating this argument gives
\[\delta_{S} \ge |A|^{-3^{|S|-1}(1 + \frac{2}{3} + \frac{4}{9} + \frac{8}{27} + \dots )} \cdot \delta_{\emptyset}^{2^{|S|}} = \delta^{2^{|S|}}|A|^{-3^{|S|}},\]
as claimed.
\end{proof}

We now turn to constructing additive-restrictive systems. We first do this for sets of governing expansions.

\begin{prop}
\label{prop:ARS_gov}
Take $\mathfrak{G}$ to be a set of governing expansions on a space $X = X_1 \times \dots \times X_d$, and choose a nonempty subset $S_{\max}$ of $[d]$ that contains $i_a = i_a(\mathfrak{G})$. There is then an additive-restrictive system $\mathfrak{A}$ on $X$ so that, for $S \subseteq S_{\max}$, we have
\[\overline{Y}_S(\mathfrak{A}) = \overline{Y}_S(\mathfrak{G}).\]
Furthermore, for all $S \subseteq [d]$, this additive restrictive system satisfies
\[\big|A_S(\mathfrak{A})\big| \le 2^{|S_{\max}| + 1}.\]

\end{prop}
\begin{proof}
We will construct maps $F_S(\mathfrak{A}): \overline{Y}_S(\mathfrak{G}) \rightarrow A_S(\mathfrak{A})$ as in this proposition statement for all $S \subseteq S_{\max}$. We will do this based on the structure of $S$.

First, suppose $S$ is a singleton $\{j\}$. Then we take $F_S(\bar{x}) = 0$ if and only if
\[\pi_0(\pi_j(\bar{x}))\pi_1(\pi_j(\bar{x}))\]
is a quadratic residue at at $2$ and at all primes in $\pi_{S_{\max} - \{j\}}(\bar{x})$. Two bits encode the residue information at $2$, and one bit encodes it at the remaining $|S_{\max}| - 1$ primes, so we can take $\overline{Y}_S^{\,\circ}$ as the kernel of a map to
\[A_S(\mathfrak{A}) = (\Z/2\Z)^{|S_{\max}| + 1}.\]

Suppose instead that $|S| > 1$ and that $S$ contains $i_a$. Then we want to take $F_S(\bar{x}) = 0$ if and only if $\phi_{\bar{x}}(\mathfrak{G})$ is a trivial map at the place $\pi_i(\bar{x})$ for all $i \in S_{\max} - S$. We know that $\phi_{\bar{x}}$ is an unramified quadratic character at each such place, so this information can be encoded at one bit per place in $S_{\max} - S$. Then we can take $\overline{Y}_S^{\,\circ}$ as a kernel of a map to
\[A_S(\mathfrak{A}) = (\Z/2\Z)^{|S_{\max}| - |S|}.\]
Outside of these two cases, we take $A_S$ to be the trivial group. This defines our additive-restrictive system, and we can verify from the definition of a set of governing expansions that it satisfies $\overline{Y}_S(\mathfrak{A}) = \overline{Y}_S(\mathfrak{G})$ for $S \subseteq S_{\max}$.
\end{proof}

\subsection{Additive-Restrictive systems for class and Selmer groups}
We now turn to the construction of an additive-restrictive system which can be used to control the sizes of class groups and Selmer groups. The constructions are similar for Selmer groups and class groups, so we define them at the same time. We first define the data needed to construct the additive-restrictive system.
\begin{defn}
\label{def:AR_input}
Take $K/\QQ$ to be a quadratic imaginary field, or take $E/\QQ$ to be an elliptic curve with full rational two torsion and no rational order four cyclic subgroup; the former case will be called the \emph{class side}, the latter the \emph{Selmer side}. We write a basis for $E[2]$ as $e_1, e_2$.

Take $X_1, \dots, X_d$ to be disjoints sets of odd primes where $K/\QQ$ is unramified on the class side, and where $E/\QQ$ is good on the Selmer side. Write $X$ for the products of the $X_i$.

In the notation of Section \ref{sec:alg1}, we suppose that every $\bar{x} \in \overline{X}_{[d]}$ is quadratically consistent. We then define the \emph{additive-restrictive input} as an assignment of the following six pieces of interconnected data:

\textbf{(1) A choice of lower pairings:} Choose some $x_0 \in X$. On the class side, find the set $D^{\vee}_{(2)}$ of tuples $w_a = (T_a, \Delta_a)$ with $w_a(x_0)$ in $2\overline{\text{Cl}}^{\vee} K(x_0)[4]$; similarly, find the set $D_{(2)}$ of tuples mapping to $2\overline{\text{Cl}}\, K(x_0)[4]$. By quadratic consistency, we see that these sets do not depend on the choice of $x_0$. Write $t_a$ for the nontrivial element of the kernel of
\[D^{\vee}_{(2)} \rightarrow \text{Cl}^{\vee} K(x_0)[4].\]
and similarly define $t_b$ in $D_{(2)}$. Choose an integer $m \ge 2$ and filtrations
\[D^{\vee}_{(2)} \supseteq D^{\vee}_{(3)} \supseteq \dots \supseteq D^{\vee}_{(m)} \ni t_a \text{ and}\]
\[D_{(2)} \supseteq D_{(3)} \supseteq \dots \supseteq D_{(m)} \ni t_b\]
of vector spaces. For $k < m$, choose a bilinear pairing
\[\text{Art}_{(k)}: D^{\vee}_{(k)} \times D_{(k)} \rightarrow \mathbb{F}_2\]
whose left kernel is $D^{\vee}_{(k+1)}$ and whose right kernel is $D_{(k+1)}$.

On the Selmer side, take $D_{(1)}$ to be the set of tuples mapping to the $2$-Selmer group of $E(x_0)$. By quadratic consistency, this set does not depend on the choice of $x_0$. We write $\text{Im}(E[2])$ for the image of the $2$-torsion in the $2$-Selmer group, and specifically write $t_2$ for the image of $e_2$ in the $2$-Selmer group. Choose an integer $m \ge 1$ and a filtration
\[D_{(1)} \supseteq D_{(2)} \supseteq D_{(3)} \supseteq \dots \supseteq D_{(m)} \supseteq \text{Im}(E[2])\]
of vector spaces. For $k < m$, choose an alternating pairing
\[\text{Ctp}_{(k)}: D_{(k)} \times D_{(k)} \rightarrow \mathbb{F}_2\]
whose kernel is $D_{(k+1)}$.

\textbf{(2) A choice of basis:} On the class side, take $n_k$ to be the dimension of $D^{\vee}_{(k)}/\langle t_a \rangle$ for $2 \le k \le m$. Then choose $w_{a1}, \dots, w_{an_2} \in D^{\vee}_{(2)}$  and $w_{b1},\dots, w_{bn_2} \in D_{(2)}$ so that, for $2 \le k \le m$, the first $n_k$ vectors in the first sequence are a basis for $D^{\vee}_{(k)}/\langle t_a \rangle$, and the first $n_k$ vectors in the second sequence are a basis for $D_{(k)}/\langle t_b \rangle$.

On the Selmer side, take $n_k$ to be the dimension of $D_{(k)}/\text{Im}(E[2])$ for $1 \le k \le m$. Take $w_1, \dots, w_{n_1} \in D_{(1)}$ so that, for $1 \le k \le m$, the first $n_k$ vectors in this sequence generate $D_{(k)}/\text{Im}(E[2])$.

\textbf{(3) A choice of variable indices:} Choose $i_b \le d$, and for $j_1, j_2 \le n_0$, choose an element $i_a(j_1, j_2)$ and a set $S(j_1, j_2)$ containing both $i_b$ and $i_a(j_1, j_2)$. We require these sets to obey different properties on the class and Selmer side.

On the class side, for all $j \le n$, we assume that $T(w_{aj})$ and $T(w_{bj})$ do not contain $i_b$. In addition, for all $j_1, j_2 \le n_0$, we assume the following:
\begin{itemize}
\item We assume that $S(j_1, j_2)$ has cardinality $m + 1$.
\item  We assume that $S(j_1, j_2)$ is disjoint from $T(w_{aj})$ and $T(w_{bj})$ for all $j  \le n$ other than $j_1$ and $j_2$.
\item We assume that
\begin{align*}
&T(w_{bj_1}) \cap S(j_1, j_2) = T(w_{aj_2}) \cap S(j_1, j_2)  = \emptyset, \\
&\qquad T(w_{aj_1}) \cap S(j_1, j_2) \,=\, \big\{i_a(j_1, j_2)\big\} \quad\text{and}\\
& \qquad S(j_1, j_2) \,\subseteq\, T(w_{bj_2}) \cup  \{i_b\}.
\end{align*}
\end{itemize}

On the Selmer side, for all $j \le n$, we assume that $T_1(w_j)$ and $T_2(w_j)$ do not contain $i_b$. In addition, if $j_1$ equals $j_2$, we assume that $S(j_1, j_2)$ is the empty set; and if $j_1$ is greater than $j_2$, we assume that $S(j_1, j_2)$ equals $S(j_2, j_1)$. In addition, for all $j_1 < j_2 \le n_0$, we assume the following:
\begin{itemize}
\item We assume that $S(j_1, j_2)$ has cardinality $m+2$.
\item We assume that $S(j_1, j_2)$ is disjoint from $T_1(w_j)$ and $T_2(w_j)$ for all $j \le n$ other than $j_1$ or $j_2$.
\item We assume that
\begin{align*}
&T_2(w_{j_1})  \cap S(j_1, j_2) \,=\,T_1(w_{j_2})  \cap S(j_1, j_2) \,=\, \emptyset,\\
&\quad\qquad T_1(w_{j_1}) \cap S(j_1, j_2) \,=\, \big\{i_a(j_1, j_2)\big\} \quad\text{and}\\
&\quad \qquad S(j_1, j_2) \,\subseteq\, T_2(w_{j_2}) \cup  \{i_b\}.
\end{align*}
\end{itemize}

We use the term \emph{variable indices} to describe the $S(j_1, j_2)$ because, when we actually prove our equidistribution results in Proposition \ref{prop:proof_C}, we will have fixed a choice of prime in each $X_i$ other than at the $i$ in $S(j_1, j_2)$.

\textbf{(4) A choice of raw cocycles:}
On the class side, we find a set of raw cocycles $\mathfrak{R}(w_{aj})$ for each $j \le n_2$, where the set of raw cocycles is as in Section \ref{ssec:class_exp}. On the Selmer side, we find a set of raw cocycles $\mathfrak{R}(w_j)$ for each $j \le n_1$.

\textbf{(5) A choice of governing expansions:}
For each distinct $i_a = i_a(j_1, j_2)$ marked in the third part of the definition, we take $\mathfrak{G}(i_a)$ to be a set of governing expansions over $X$ with $i_a = i_a(\mathfrak{G}(i_a))$. For every $S$ of the form $S(j_1, j_2) - \{i_b, i\}$ for some $i \in S(j_1, j_2)$ other than $i_a(j_1, j_2)$ or $i_b$, and for every $\bar{x} \in \overline{X}_S$, we assume that the expansion 
\[ \phi_{\bar{x}}(\mathfrak{G}(i_a(j_1, j_2)))\] exists. We also assume it is trivial when restricted to $\text{Gal}(\overline{\QQ}_v/\QQ_v)$ for $v$ coming from a certain set of places:
\begin{itemize}
\item On the class side, we presume that the expansion is trivial at $2$, at $\infty$, at all places in $\pi_{[d] - S}(\bar{x})$, and at all primes dividing the discriminant of $K/\QQ$.
\item On the Selmer side, we presume that the expansion is trivial at $2$, at $\infty$, at all places in $\pi_{[d] - S}(\bar{x})$, and at all primes dividing the conductor of $E/\QQ$.
\end{itemize}

\textbf{(6) A choice of inertia elements:}
Take $M_r/\QQ$ to be the least number field containing $L(\psi_k(\mathfrak{R}, \,x))$ whenever it exists for any $x \in X$, any $\mathfrak{R}$ as defined in (4), and any $k \le m$. Take $M$ to be the least number field extending $M_r$ that also contains the field of definition of each expansion found in any $\mathfrak{G}(i_a(j_1, j_2))$. $M/\QQ$ has ramification degree at most two at any prime; for each prime $p$ where it ramifies, choose some $\sigma_p$ in $\text{Gal}(M/\QQ)$ so that $\{1, \, \sigma_p\}$ is the inertia group of some prime dividing $p$ in $M$.
\end{defn} 

We will use $\mathscr{P}$ to denote an assignment of the additive-restrictive input. To define the additive-restrictive system associated to $\mathscr{P}$, we will first need to understand the role of the choice of inertia elements.

\begin{defn*}
Suppose we have some choice of additive-restrictive input $\mathscr{P}$, and choose $w$ in $D^{\vee}_{(2)}$ on the class side and in $D_{(1)}$ on the Selmer side where $\mathfrak{R}(w)$ has been chosen. Choose $S \subseteq [d]$ over which $\mathfrak{R}(w)$ is consistent or $i_a(j_1, j_2)$-consistent. Suppose $\bar{x} \in \overline{X}_S$ satisfies 
\[\text{rk}(\mathfrak{R}(w))(x) \ge |S| +1\quad\text{for all}\,\,\,x \in \widehat{x}(\emptyset).\]

In the case of consistency, for any $i \in [d] - S$, we call $\mathfrak{R}(w)$ \emph{acceptably ramified at} $(\bar{x}, \, i)$ if
\[\sum_{x \in \widehat{x}(\emptyset)} \beta(x, x_0) \circ \psi_{|S| + 1}\big(\mathfrak{R}(w),\, x\big)(\sigma_{\pi_i(\bar{x})}) \,=\, 0.\]

In the case of $i_a$-consistency, choose $i \in [d] - S$, and suppose there is some $\bar{z} \in \overline{X}_{S \cup \{i\}}$ satisfying $\bar{x} \in \widehat{z}(S)$ at which $\phi_{\bar{z}}(\mathfrak{G}(i_a))$ is defined and has $\pi_1(\pi_{i_a}(\bar{z})) \ne\pi_0(\pi_{i_a}(\bar{z}))$. We then call $\mathfrak{R}(w)$ \emph{acceptably ramified at} $(\bar{x}, \, i)$ if
\[\sum_{x \in \widehat{x}(\emptyset)} \beta(x, x_0) \circ \psi_{|S| + 1}\big(\mathfrak{R}(w),\, x\big)(\sigma_{\pi_i(\bar{x})})\, =\, \phi_{\bar{z}}(\sigma_{\pi_i(\bar{x})}).\]
\end{defn*}

We can now define our additive-restrictive system.
\begin{defn}
\label{def:RGAR}
Take $\mathscr{P}$ to be an additive-restrictive input as above, and take all notation as in Definition \ref{def:AR_input}. Choose $j_1, j_2 \le n_0$; on the Selmer side, we assume that $j_1$ is less than $j_2$. We will define an additive-restrictive system $\mathfrak{A}(j_1, j_2) = (\overline{Y}_S,\, \overline{Y}_S^{\,\circ},\, F_S, \,A_S)$ as follows.

First, on the class side, take $\overline{Y}_{\emptyset}^{\,\circ}(j_1, j_2)$ to be the set of $x \in X$ on which the natural pairings
\[2^{k-1}\overline{\text{Cl}}^{\vee} K(x)[2^k] \,\times\, 2^{k-1} \overline{\text{Cl}}\, K(x)[2^k] \longrightarrow \mathbb{F}_2\]
agree with the pairings $\text{Art}_{(k)}$ for each $k$ in the range $1 < k < m$. On the Selmer side, take $\overline{Y}_{\emptyset}^{\,\circ}$ to be the set of $x$ on which the natural pairings 
\[2^{k-1}\overline{\text{Sel}}^{\,2^k}E(x) \,\times\, 2^{k-1}\overline{\text{Sel}}^{\,2^k}E(x) \longrightarrow \mathbb{F}_2\]
agree with the pairings $\text{Ctp}_{(k)}$ for each $k$ in the range $0 < k < m$.

Next, choose $S \subseteq [d]$. We will now define a set $\overline{Y}_S^{\,\circ}(j_1, j_2)$ in $\overline{Y}_S(j_1, j_2)$. We have done this already for $S$ empty. Next, if $S$ is not contained in $S(j_1, j_2)$, or if $S$ has cardinality greater than $|S(j_1, j_2)| - 2$, we take $\overline{Y}_S^{\,\circ}(j_1, j_2) = \overline{Y}_S(j_1, j_2)$.

Now suppose $S \subset S(j_1, j_2)$ with cardinality at most $|S(j_1, j_2)| - 2$. On the class side, we say that $\bar{x} \in \overline{Y}_S$ is in $\overline{Y}^{\,\circ}_S$ if we have the following:
\begin{itemize}
\item For $j \le n_{|S| + 1}$ other than $j_1$, we have that $\mathfrak{R}(w_{aj})$ is minimal at $\bar{x}$.
\item We have that $\mathfrak{R}(w_{aj_1})$ agrees with $\mathfrak{G}\big(i_a(j_1, j_2)\big)$ at $\bar{x}$.
\item For $j \le n_{|S| + 1}$, we have that $\mathfrak{R}(w_{aj})$ is acceptably ramifeid at $(\bar{x},\, i)$ for all $i$ in $S(j_1, j_2) - S$.

\end{itemize}
On the Selmer side, we say that $\bar{x} \in \overline{Y}_S$ is in $\overline{Y}^{\,\circ}_S$ if we have the following:
\begin{itemize}
\item For $j \le n_{|S|}$ other than $j_1, j_2$, we have that $\mathfrak{R}(w_j)$ is minimal at $\bar{x}$.
\item We have that $\mathfrak{R}(w_{j_1})$ agrees with $\mathfrak{G}(i_a(j_1, j_2))$ at $\bar{x}$. 
\item If $|S| < m$, then for $j \le n_{|S|}$ other than $j_2$, we have that $\mathfrak{R}(w_j)$ is acceptably ramified at $(\bar{x}, i)$ for all $i$ in $S(j_1, j_2) - S$. 
\end{itemize}
 
Suppose $\bar{x}$ is in $\overline{Y}_S(j_1, j_2)$ for some subset $S$ of $S(j_1, j_2)$ of cardinality at most $|S(j_1, j_2)| - 2$. Then Proposition \ref{prop:psi_coh} and \ref{prop:R_G_agree} imply that $\psi(\mathfrak{R}(w), \,\bar{x})$ or $\psi(\mathfrak{R}(w), \,\bar{x}) + \phi_{\bar{x}}(\mathfrak{G}(i_a))$ is a cocycle for each $w$ considered in the above definition. Call this cocycle $\psi$.

On the class side, $\psi$ is a quadratic character. The acceptable ramification conditions prevent $\psi$ from being ramified at any prime in $\pi_S(\bar{x})$, so it is an unramified character over any $K(x)$ with $x \in \widehat{x}(\emptyset)$. As $\text{rk}(\mathfrak{R}(w))(x) > |S|$ for each $x \in \widehat{x}(\emptyset)$, and from the local triviality assumptions we made in part (5) of Definition \ref{def:AR_input}, we find that $\psi$ is trivial over any $K(x)$ at all primes where $K(x)/\QQ$ ramifies besides those in $\pi_S(x)$. If $\psi$ is trivial over $K(x)$ at all primes in $\pi_S(x)$, we then have that $\psi$ corresponds to an element of $D^{\vee}_{(2)}$. We have $|S|$ bits describing the behavior at $\pi_S(x)$, and an element in $D^{\vee}_{(2)}$ can be described with $n_2 +1$ bits. Finally, acceptable ramification can be described with $|S(j_1, j_2) - S|$ bits. The conditions of each of these bits are additive, and we have one set of conditions for each of the $n_2$ vectors $w_{aj}$, so we can take $\overline{Y}_S^{\,\circ}(j_1, j_2)$ to be the kernel of some additive map
\[\overline{Y}_S(j_1, j_2) \rightarrow (\Z/2\Z)^{n_2(n_2 + m + 2)}.\]

On the Selmer side, $\psi$ can be thought of as a pair of quadratic characters. The cocycle is unramified at all primes in $\pi_S(\bar{x})$. Choose some $x \in \widehat{x}(\emptyset)$. If $\psi$ is trivial at all the primes in  $\pi_S(x)$, we find that $\psi$ is a $2$-Selmer element (again using part (5) of Definition \ref{def:AR_input} if necessary). Each local condition is described with two bits, and it takes $n_1 + 2$ bits to describe an element in $D_{(1)}$. Finally, acceptable ramification can be described with two bits at each $i$ in $S(j_1, j_2) - S$. This gives a set of $n_1 + 2m + 6$ bits to describe the conditions accrued from one $w_j$. Varying $j$ in $[n_1] - \{j_2\}$, we see we can take  $\overline{Y}_S^{\,\circ}(j_1, j_2)$ to be the kernel of some additive map
\[\overline{Y}_S(j_1, j_2) \rightarrow (\Z/2\Z)^{(n_1-1)(n_1 + 2m + 6)}.\]

This defines the additive-restrictive sequence associated with $(j_1, j_2)$.
\end{defn}

\begin{prop}
\label{prop:AR_main}
Take $\mathscr{P}$ to be an additive-restrictive input defined either with respect to an elliptic curve or imaginary quadratic field, and choose some $\mathfrak{A}(\mathscr{P})(j_1, j_2)$ as defined above. Take $S = S(j_1, j_2)$, and take $\bar{x} \in \overline{X}_S$. Suppose that, for each $i$ in $S$, there is some $\bar{z}_i \in \widehat{x}(S - \{i\})$ so that
\[\bar{z}_i \in \overline{Y}_{S- \{i\}}^{\,\circ}\big(\mathfrak{A}(\mathscr{P})(j_1, j_2)\big).\]
Then $\widehat{x}(\emptyset)$ is a subset of $\overline{Y}_{\emptyset}^{\,\circ}$. Furthermore, write $(p_{0b},\, p_{1b})$ for $\pi_{i_b}(\bar{x})$ and $i_a$ for $i_a(j_1, j_2)$. Then, on the class side, we have
\begin{align*}
&\sum_{x \in \widehat{x}(\emptyset)} \left[ \frac{L\big(\psi_{m}(\mathfrak{R}(w_{aj_3}), \,x)\big)/K(x)}{w_{bj_4}(x)}\right] \\
&\qquad\quad= \begin{cases} \phi_{\bar{z}_{i_b}}(\mathfrak{G}(i_a)) \big(\text{\emph{Frob}}(p_{0b}) \cdot \text{\emph{Frob}}(p_{1b})\big) &\text{ if } (j_3, j_4) = (j_1, j_2) \\ 0 &\text{ otherwise}\end{cases}
 \end{align*}
 for all $j_3, j_4 \le n_m$. On the Selmer side, we instead have
 \begin{align*}
&\sum_{x \in \widehat{x}(\emptyset)} \big\langle \psi_{m}(\mathfrak{R}(w_{j_3}), \,x),\, w_{j_4}(x)\big\rangle_{CT} \\
&\qquad\quad= \begin{cases} \phi_{\bar{z}_{i_b}}(\mathfrak{G}(i_a)) \big(\text{\emph{Frob}}(p_{0b}) \cdot \text{\emph{Frob}}(p_{1b})\big) &\text{ if } (j_3, j_4) = (j_1, j_2) \\ 0 &\text{ otherwise}\end{cases}
 \end{align*}
 for $j_3 < j_4 \le n_m$.
 \end{prop}
 \begin{proof}
 Take $x_0$ to be the element of $\widehat{x}(\emptyset)$ not in any $\bar{z}_i$. We first need to check that the Cassels-Tate pairings or Artin pairings corresponding to $x_0$ is given by $\text{Ctp}_{(k)}$ or $\text{Art}_{(k)}$ for $k < m$. On the class side, we do this by considering the value of these pairings on each $(w_{aj_3}, \,w_{bj_4})$ or, for $j_4 = j_2$, on $(w_{aj_3},\, t_b + w_{bj_2})$. The value of the pairing at these tuples determines the pairing everywhere by bilinearity. But, given the minimality restrictions on the $\mathfrak{R}(w)$, we see that the first part of Theorem \ref{thm:Cl_gov} implies that the Artin pairings for $k < m$ at $x_0$ equal the sum of the Artin pairings at all other vertices in $\widehat{x}(\emptyset)$. This is enough to give that $x_0$ is in $\overline{Y}_{\emptyset}^{\,\circ}$.

The pairings at $k = m$ follow similarly except at $(w_{aj_1},\, t_b + w_{bj_2})$. At this tuple, though, the second part of the theorem applies, and we get the claimed result on the Artin pairing.

On the Selmer side, we can follow the same argument, finding the sum of the Cassels-Tate pairing over $x$ in $\widehat{x}(\emptyset)$. We again use pairs of the form $(w_{j_3}, \, w_{j_4})$ and $(w_{j_3},\, t_2 + w_{j_2})$; we can assume that $j_3$ and $j_4$ are both not equal to $j_2$. We see that it is enough to prove that the Cassels-Tate pairing behaves as we expect on pairs of this form since the pairing is alternating.  Theorem \ref{thm:Sel_gov} then gives us that $x_0$ is in $\overline{Y}_{\emptyset}^{\, \circ}$ and that the sum of the Cassels-Tate pairing at level $m$ over $x \in \widehat{x}(\emptyset)$ obeys the given formula. 
\end{proof}

\section{Ramsey-Theoretic results}
\label{sec:rams}
In Proposition \ref{prop:AR_main}, we found a condition on $\bar{x} \in \overline{X}_S$ under which the sum
\[\sum_{x \in \widehat{x}(\emptyset)} \big\langle \psi_m(\mathfrak{R}(w_{j_3}),\, x),\, w_{j_4}(x)\big\rangle_{CT} \in \mathbb{F}_2\]
was determined by an Artin symbol in the field of definition of some governing expansion, with an analogous form found on the class side. This information is not enough to determine the value of the pairing at any particular $x \in \widehat{x}(\emptyset)$. However, if we have enough choices of $\bar{x}$ where we can find this sum, we can still usually prove that the value of the pairing is forced to be $1$ on about half the vertices in $\overline{Y}_{\emptyset}^{\,\circ}$.

The first question is whether there is even one choice of such a $\bar{x}$ whose vertices lie in $\overline{Y}_{\emptyset}^{\,\circ}$. This is a question in Ramsey theory; we can prove that such a $\bar{x}$ exists if $\overline{Y}_{\emptyset}^{\,\circ}$ is large enough. This is the $r =2$ case of the following proposition.

\begin{prop}
\label{prop:subgrid}
Take $d \ge 2$ to be an integer, take $2^{-d-1} > \delta > 0$, and take $X_1, \dots, X_d$ to be finite sets with cardinality at least $n > 0$. Suppose that $Y$ is a subset of $X = X_1 \times \dots \times X_d$ of cardinality at least $\delta \cdot |X|$. Then, for any positive $r$ satisfying
\[r \le \left(\frac{\log n}{5 \log \delta^{-1}}\right)^{1/(d-1)},\]
there exists a choice of sets $Z_1, \dots, Z_d$, each of cardinality $r$, such that
\[Z_1 \times \dots \times Z_d \subseteq Y.\]
\end{prop}
\begin{proof}
We can find subsets $X'_i$ of the $X_i$ so $|X'_i| = n$ and so $Y$ has density at least $\delta$ in $X'_1 \times \dots \times X'_d$. Because of this, we can without loss of generality assume that $X_1, \dots, X_d$ have cardinality exactly $n$.

For any positive integers $d$ and $r$ and any $Y \subset X$, write $N(r, Y)$ for the number of ways of choosing subsets $Z_i$ of $X_i$ for all $i \le d$, each of cardinality $r$, so that $Z_1 \times \dots \times Z_d \subset Y$. Write $N_d(n, r, \delta)$ for the minimum of $N(r, Y)$ over all $Y$ of cardinality at least $\delta \cdot |X|$. To prove the proposition, we will show that for $d > 0, \delta > 0$, and $n \ge r \ge 2$ satisfying
\begin{equation}
\label{eq:subgrid}
 (2^{-d -1} \delta)^{2r^{d-1}} \cdot nr^{-1} \ge 1,
\end{equation}
we have
\begin{equation}
\label{eq:subgrid_ind}
N_d(n, r, \delta) \ge  (2^{-d-1}\delta)^{\frac{r^{d+1} - r}{r-1}}\frac{n^{rd}}{(r!)^d}.
\end{equation}
The condition of the proposition is stricter than \eqref{eq:subgrid}, so this will be sufficient to show the proposition.

We prove the claim by induction. Setting $d = 1$, we find
\[N_1(n, r, \delta) \ge \frac{(\delta n - r)^r}{r!}\]
For $r \le \frac{1}{2}\delta n$, this gives 
\[N_1 \ge (\delta/2)^{r} \frac{n^r}{r!},\]
and this gives us the base case for \eqref{eq:subgrid_ind}.

Now consider the case of $d > 1$, and choose $Y$ with $N(r, Y)$ minimal. Take $X_{\text{thick}}$ to be the subset of $x \in X_1$ so that 
\[Y_x = Y \cap \big(\{x\} \times X_2 \times \dots \times X_d\big) \]
has density at least $\delta/2$ in $ \{x\} \times X_2 \times \dots \times X_{d}$. $X_{\text{thick}}$ has density at least $\delta/2$ in $X_1$.

Take $\mathscr{Z}$ to be the set of choices of subsets $Z_2, \dots, Z_d$, $Z_i \subseteq X_i$ such that each $Z_i$ has cardinality $r$. We have
\[|\mathscr{Z}| \le \frac{1}{(r!)^{d-1}} n^{r(d-1)}.\]
For 
\[\mathbf{z} = (Z_2, \dots, Z_d) \in \mathscr{Z},\]
take $n_{\mathbf{z}}$ to be the number of $x \in X_{\text{thick}}$ such that $Y$ contains 
\[\{x\} \times Z_2 \times \dots \times Z_d.\]
Then
\[N_d(n, r, \delta) = N(r, Y) \ge \sum_{\substack{\mathbf{z} \in \mathscr{Z} \\ n_{\mathbf{z}} \ge r}} \frac{1}{r!}(n_\mathbf{z} - r)^r  \ge \sum_{\mathbf{z} \in \mathscr{Z}}\frac{1}{2^rr!}n_{\mathbf{z}}^r - \frac{r^r}{r!}.\]
We have
\[\sum_{\mathbf{z} \in \mathscr{Z}} n_{\mathbf{z}} \ge |X_{\text{thick}}| \cdot N_{d-1}(n, r, \delta/2)\,\,\ge\,\, \frac{\delta}{2} n (2^{-d-1}\delta)^{\frac{r^d - r}{r-1}} \frac{n^{r(d-1)}}{(r!)^{d-1}}\]
so
\[\frac{\sum_{\mathbf{z} \in \mathscr{Z}} n_{\mathbf{z}}}{\sum_{\mathbf{z} \in \mathscr{Z}} 1} \ge \frac{\delta}{2}n(2^{-d-1}\delta)^{\frac{r^d - r}{r-1}} \ge 4n(2^{-d-1}\delta)^{\frac{r^d - 1}{r-1}} .\]
With Cauchy-Schwarz, we then get
\[N_d(n, r, \delta) \ge \frac{n^{r(d-1)}}{(r!)^{d-1}}  \left(-\frac{r^r}{r!} + \frac{1}{2^rr!}4^rn^r(2^{-d-1}\delta)^{\frac{r^{d+1} - r}{r-1}}\right).\]
But, for $r \ge 2$,
\[\frac{r^{d+1} - r}{r-1} \le 2r^d,\]
so \eqref{eq:subgrid} implies
\[N_d(n, r, \delta) \ge  (2^{-d-1}\delta)^{\frac{r^{d+1} - r}{r-1}}\frac{n^{rd}}{(r!)^d},\]
as claimed. This is thus true for all $d$ by induction, proving the proposition.
\end{proof}

As a first consequence of this proposition, we will show that, if a subset $Z$ of $X$ is large enough, then a function from $Z$ to $\mathbb{F}_2$ with ``generic differential''  will typically be $1$ on about half of $Z$. The next definition and proposition formalize this notion.

\begin{defn}
\label{def:GSZ}
Take $X_1, \dots, X_d$ to be disjoint finite nonempty sets, and take $X$ to be their product. Choose a nonempty subset $S$ of $[d]$ of cardinality at least two, and choose some $Z \subseteq X$ so $\pi_{[d] - S}(Z)$ is a point. Taking $F$ to be a function from $Z$ to $\mathbb{F}_2$, we define a function
\[dF: \,\big\{\bar{x} \in \overline{X}_S \,:\,\, \widehat{x}(\emptyset) \subseteq Z\big\}\longrightarrow \mathbb{F}_2\]
by
\[dF(\bar{x}) = \begin{cases} \sum_{x \in \widehat{x}(\emptyset)} F(x) &\text{ if }\,\,\, |\widehat{x}(\emptyset)| = 2^{|S|} \\ 0 &\text{ otherwise.}\end{cases}
\]
Write $\mathscr{G}_S(Z)$ for the image of this map $d$. In addition, for $\epsilon > 0$, write $\mathscr{G}_S(\epsilon,\, Z)$ for the set of $g \in \mathscr{G}_S(Z)$ expressible in the form $g = dF$ for some $F$ that equals $1$ on more than $(0.5 + \epsilon)|Z|$ or fewer than $(0.5 - \epsilon)|Z|$  points in $Z$.
\end{defn}

\begin{prop}
\label{prop:sbst_FG}
Taking $X$ and $Z$ as in the previous definition, choose $\delta > 0$ so that
\[|Z| \ge \delta \cdot |\pi_S(X)|.\]
Suppose $|X_i| \ge n$ for each $i \in S$. Then, for $\epsilon > 0$,
\[\frac{|\mathscr{G}_S(\epsilon,\, Z)|}{|\mathscr{G}_S(Z)|} \le \exp\left( |\pi_S(X)| \cdot \left(-\delta \epsilon^2 \,+\, 2^{|S|+2}\cdot n^{-1/2^{|S|}}\right)\right).\]
\end{prop}
\begin{proof}
Take $Z'$ to be a maximal subset of $Z$ so that there is no $\bar{z} \in \overline{X}_S$ satisfying $|\widehat{z}(\emptyset)| = 2^{|S|}$ and $\widehat{z}(\emptyset) \subseteq Z'$. We see that the kernel of the map $d: \mathbb{F}_2^Z \rightarrow \mathscr{G}_S(Z)$ then has size at most $2^{|Z'|}$. From applying \eqref{eq:subgrid} with $r = 2$, we also have
\[ |Z'| \le |\pi_S(X)| \cdot 2^{|S| + 2} \cdot N^{-1/2^{|S|}}.\]
Then we must have
\[|\mathscr{G}_S(Z)| \ge 2^{|Z|} \cdot \exp\left(-|\pi_S(X)| \cdot 2^{|S| + 2} \cdot N^{-1/2^{|S|}}\right)\]

On the other hand, from Hoeffding's inequality, the number of $F$ equaling $1$ on more than $(0.5 + \epsilon)|Z|$ or fewer than $(0.5 - \epsilon)|Z|$ points in $Z$ is bounded by
\[2^{|Z| + 1} \exp\big(-2\epsilon^2 |Z|\big)\]
by Hoeffding's inequality \cite[Theorem 1]{Okam59}. Then $\mathscr{G}_S(\epsilon,\, Z)$ is bounded by
\[|\mathscr{G}_S(\epsilon,\, Z)| \le 2^{|Z| } \exp\big(-2\epsilon^2 |Z|\big).\]
Taking ratios of these estimates then gives the result.
\end{proof}

We run into two issues when we try to apply Proposition \ref{prop:sbst_FG} directly to Proposition \ref{prop:AR_main}. The first is that we do not a priori have any control on the form of $Z = \overline{Y}_{\emptyset}^{\,\circ}$. We can restrict an element of $\mathscr{G}_S(\pi_S(X_S))$ to $\mathscr{G}_S(Z)$, but the preimages of the various $\mathscr{G}_S(\epsilon,\, Z)$  will depend on the choice of $Z$. Furthermore, in the context of Proposition \ref{prop:AR_main}, it is not enough that $\widehat{x}(\emptyset)$ lie in $\overline{Y}_{\emptyset}^{\,\circ}$ to conclude $dF(\bar{x}) = g(\bar{x})$ for the relevant $F$ and $g$; we must have instead that $\widehat{x}(T)$ meet $\overline{Y}_{T}^{\,\circ}$ for each proper $T$ in $S$.

Fortunately, thanks to the structure already found for additive-restrictive systems, both of these issues can be solved with a little more work. First, we have a regularity condition on $\overline{Y}_{\emptyset}^{\,\circ}$ proved in Proposition \ref{prop:AR_main}, where we found that $x_0$ could be proved to be in this set by finding a nice cube $\bar{x} \in \overline{X}_S$ with all other vertices in this set. Because of this, we do not need to consider all possible $Z$. Furthermore, thanks to Proposition \ref{prop:ars_density}, we can find a minimal density for $\overline{Y}_S^{\,\circ}$ in terms of the density of $\overline{Y}_{\emptyset}^{\,\circ}$, and this is enough to circumvent the second issue. The end result is the following proposition.
\begin{prop}
\label{prop:staff}
There is an absolute positive constant $A$ so that we have the following:

Take $X$ and $S$ as in the the previous definition. For $a \ge 2$ and $\epsilon > 0$, define
\[\mathscr{G}_S(\epsilon, \,a, \,X)\]
to be the set of $g \in \mathscr{G}_S\big(\pi_S(X_S)\big)$ for which there is some $Z \subseteq X$, some $F: Z \rightarrow \mathbb{F}_2$, some additive restrictive system $\mathfrak{A}$ on $X$, and some subset $\overline{Z}_S$ of $\overline{X}_S$ so that
\begin{itemize}
\item The image of $Z$ under $\pi_{[d]-S}$ is a point.
\item For each $T \subseteq S$, we have $|A_T(\mathfrak{A})| \le a$.
\item If $\bar{x}$ is in $\overline{Z}_S$, then $\widehat{x}(\emptyset) \subseteq Z$ and
\[dF(\bar{x}) = g(\bar{x}).\]
\item We have equalities
\[\overline{Z}_S = \bigcap_{T \subsetneq S} \big\{\bar{x} \in \overline{X}_S\,:\,\, \widehat{x}(T) \cap \overline{Y}_T^{\,\circ}(\mathfrak{A}) \ne \emptyset \big\}\]
and
\[Z = \overline{Y}_{\emptyset}^{\,\circ}(\mathfrak{A}).\]
\item The function $F$ is $1$ on more than $0.5|Z| + \epsilon|\pi_S(X)|$ or fewer than $0.5|Z| - \epsilon|\pi_S(X)|$ of the points in $Z$.
\end{itemize}
Write $n$ for $\min_{i \in S} |X_i|$. Then, if $\epsilon$ is less than $a^{-1}$ and
\begin{equation}
\label{eq:Gea_cond}
\log n \ge A \cdot 6^{|S|} \log \epsilon^{-1},
\end{equation}
we have
\[\frac{|\mathscr{G}_S(\epsilon, \, a,\, X)|}{|\mathscr{G}_S(\pi_S(X_S))|}  \le \exp\big(-|\pi_S(X)| \cdot n^{-1/2}\big). \]
\end{prop}

\begin{proof}
Consider a function $g$ coming from $F$, $Z$, and $\mathfrak{A}$ as in the proposition.
For any $x_0 \in Z$, define $Z(x_0)$ as the set of $x$ in $Z$ for which there is some $\bar{x} \in \overline{X}_S$ with $x, x_0 \in \widehat{x}(\emptyset)$ such that, if $T$ is a proper subset of $S$ and $\bar{y}$ is an element of $\widehat{x}(T)$ that contains the vertex $x_0$, then $\bar{y}$ is in $\overline{Y}_T^{\,\circ}$. From Proposition \ref{prop:ars_density}, we see that there is some sequence $x_1, \dots, x_r$ of points in $Z$ so that
\begin{equation}
\label{eq:Zxi_dom}
Z(x_j) - Z(x_{j-1}) - \dots - Z(x_1)
\end{equation}
has density at least $(0.5 \cdot a^{-1}\epsilon)^{3^{|S|}} \ge \epsilon^{3^{|S| + 1}}$ for $j \ge 1$ and so that the complement
\[Z - Z(x_r) - \dots - Z(x_1)\]
has density at most $\epsilon/2$ in $W$. Each $Z(x_j)$ is determined by the sequence of structures
\[Z(x_j) \cap \pi^{-1}_{S - \{i\}}(x_j)\]
as $i$ varies through $S$; this info can be specified by
\[|\pi_S(X)| \cdot \sum_{i \in S} \frac{1}{|X_i|} \le |\pi_S(X)| \cdot |S| \cdot n^{-1}\]
bits. There are at most $\epsilon^{-3^{|S| + 1}}$ elements $x_j$, so, given $x_1, \dots, x_r$ the info of all the $Z(x_j)$ can be specified with at most 
\[\epsilon^{-3^{|S| + 1}}|\pi_S(X)| \cdot |S| \cdot n^{-1}\]
bits. Writing $Z'(x_j)$ for the expression \eqref{eq:Zxi_dom}, we find that there must be a $j$ so that $F$ equals $1$ on at least
\[|Z'(x_j)|\cdot 0.5(1 + \epsilon)\]
vertices in $Z'(x_j)$.

The conditions on $Z'(x_j)$ imply that, if $x$ is in $Z'(x_j)$, then there is a cube $\bar{x} \in \overline{X}_S$ with $x, x_j \in \widehat{x}(\emptyset)$ such that $dF(\bar{x}) = g(\bar{x})$. Using the additivity of $dF$ and $g$ we find then that, if $\bar{x} \in \overline{X}_S$ has $\widehat{x}(\emptyset)$ contained in $Z'(x_j)$, then $dF(\bar{x}) = g(\bar{x})$. Then Proposition \ref{prop:sbst_FG} implies that the number of $g$ in $\mathscr{G}_S(\epsilon, \,a,\, X)$ corresponding to this choice of $Z'(x_j)$ is bounded by

\[|\mathscr{G}_S(\pi_S(X_S))| \cdot \exp\left(|\pi_S(X_S)| \cdot \left(-\epsilon^{4 + 3^{|S| + 1}} + 2^{|S| + 2} \cdot n^{-1/2^{|S|}}\right)\right).\]
For sufficient $A$, we use \eqref{eq:Gea_cond} to bound this by
\[|\mathscr{G}_S(\pi_S(X_S))| \cdot \exp\left(-|\pi_S(X_S)| \cdot \epsilon^{5 + 3^{|S| + 1}} \right).\]

Summing this over all possible choices of the $(x_1, \dots, x_r)$, over all choices of the $Z(x_i)$, and over all the choices of $j$, we find that the ratio being estimated by the proposition is bounded by
\[r|\pi_S(X_S)|^r \exp\left(|\pi_S(X_S)| \cdot \left(-\epsilon^{5 + 3^{|S| + 1}} + \epsilon^{-3^{|S| +1}}N^{-1}|S|\right)\right).\]
For sufficient $A$, this is less than
\[r|\pi_S(X_S)|^r \exp\left(-|\pi_S(X_S)| \cdot \epsilon^{6 + 3^{|S| + 1}} \right).\]
If $A$ is sufficiently large, we find that $n^{-1/2} > \log |\pi_S(X_S)|/ |\pi_S(X_S)|$, and the ratio is bounded by
\[\exp\left(-|\pi_S(X_S)| \cdot \epsilon^{7 + 3^{|S| + 1}} \right),\]
which is within the bounds of the proposition for sufficient $A$.
\end{proof}

\section{Prime divisors as a Poisson point process}
\label{sec:pois}
Take $N$ to be a large real number, and take $n$ to be a positive integer chosen uniformly from $[1, N]$. Taking $p_1 < \dots < p_r$ to be the prime divisors of $n$, it is commonly understood  that the values $\log \log p_i$ behave more or less like random variables uniformly chosen from the interval $[0, \log \log N]$. This model breaks down for the prime divisors near the endpoints of the interval, but it is otherwise fairly robust (see \cite{Gran07}, for example).

This is a very convenient model to have. In order for our argument to apply to an integer $n$, we need there to be some $i \ge \sqrt{r}$ with
\[\log \log p_{i+1} - \log \log p_i \ge \log \log \log p_{i+1} +  \frac{1}{2}\log \log \log \log N\]
that obeys some technical conditions, so we rely on the fact that most integers have a gap this large. It is far easier to prove that such a gap usually exists by working in the corresponding Poisson point process than by directly dealing with the prime factors.

In proving the model, the key object to understand is
\[I_k(u) = \int_{\substack{t_1,\dots, t_k \ge 1 \\ t_1 + \dots + t_k \le u}} \frac{dt_1}{t_1} \dots \frac{dt_k}{t_k}\]
where $u > 1$ and $k \ge 1$. This integral dates back to Ramanujan, with recent work done by Soundararajan \cite{Soun12}. We clearly have
\[I_k(u) \le \left(\int_1^u \frac{dt}t\right)^k = (\log u)^k.\]
Our first result is a better estimate for this integral.
\begin{lem}
\label{lem:Iku}
For $u \ge 3$ and $k \ge 1$, we have
\[\left| I_k(u) - \,\frac{e^{-\gamma \alpha}}{\Gamma(1 + \alpha)}(\log u)^k \right| = \mathcal{O}\left( (\alpha + 1)(\log u)^k\frac{(\log \log u)^3}{\log u}\right)\]
where $\alpha = k/\log u$, $\gamma$ is the Euler-Mascheroni constant, and the implicit constant is absolute. 
\end{lem}

Next, we give three properties a well-behaved sample of points on an interval should have. We have given these properties the names \emph{comfortable spacing}, \emph{regularity}, and \emph{extravagant spacing}.

\begin{defn*}
Take $L > 2$, take $n$ to be a positive integer satisfying $|L - n| < L^{3/4}$, and take $X_1, \dots, X_n$ to be independent random variables, each distributed uniformly on $[0, L]$. For $i \le n$, take $U_{(i)}$ to be the $i^{th}$ order statistic of this sample set; that is, take $U_{(i)}$ to be value of the $i^{th}$ smallest $X_i$. 
\begin{itemize}
\item For $\delta > 0$ and $0 \le L_0 < L$, we call the sample \emph{$\delta$-comfortably spaced} above $L_0$ if, for all $i < n$ such that $U_{(i)} \ge L_0$, we have
\[U_{(i+1)} - U_{(i)}\, \ge \,\delta \exp(-U_{(i)}).\]
\item For $C_0 > 0$, we call the sample \emph{$C_0$-regular} if, for all $i \le n$,
\[\left|U_{(i)} - i \right|\, <\, C_0^{1/5} \cdot \max(i, \,C_0)^{4/5}.\]
\item We call the sample \emph{extravagantly spaced} if, for some $m \ge \sqrt{n}$, we have
\[\exp(U_{(m)}) \,\ge\, U_{(m)} \cdot (\log L)^{1/2} \cdot \left(\sum_{i \ge 1}^{m-1} \exp(U_{(i)})\right).\]
\end{itemize}
\end{defn*}

\begin{prop}
\label{prop:bad_samp_bhv}
Take $X_1, \dots, X_n$ to be a sample as in the above definition.
\begin{enumerate}
\item For $\delta > 0$ and $0 \le L_0 < L$, the probability that this sample is not $\delta$-comfortably spaced above $L_0$ is bounded by
\[\mathcal{O}\big(\delta \cdot \exp(-L_0)\big)\]
with absolute implicit constant.
\item There is a positive constant $c$ so that, for $C_0 > 0$, the probability that this sample is not $C_0$-regular is bounded by
\[\mathcal{O}\big(\exp(-c\cdot C_0)\big)\]
with absolute implicit constant.
\item There is a positive constant $c$ so that the probability that the sample is not extravagantly spaced is bounded by
\[\mathcal{O}\left(\exp\left(-c \cdot \sqrt{\log L} \right)\right)\]
with absolute implicit constant.
\end{enumerate}
\end{prop}

Once this result is proved, we can move on to its number theoretic analogue. For $N, D$ positive real numbers and $r$ a positive integer, we define $S_r(N, D)$ to be the set of squarefree positive integers less than $N$ with exactly $r$ prime factors and no prime factors less than $D$.

 \begin{defn}
\label{def:spacing}
 Take $N > 30$ and $D > 3$ to be real numbers satisfying $(\log N)^{1/4} > \log D$, and take $r$ to be a positive integer satisfying
 \begin{equation}
 \label{eq:r_ND_bnd}
 \left| r - \log\left(\frac{\log N}{\log D}\right)\right| \le \log\left(\frac{\log N}{\log D}\right)^{2/3}.
 \end{equation}
For $n \in S_r(N, D)$, write $(p_1, \dots, p_r)$ for the primes dividing $n$ in order from smallest to largest.
\begin{itemize}
\item For $\delta > 0$ and $D_1 > D$, we call $n$ \emph{comfortably spaced} above $D_1$ if, for all $i < r$ such that $p_{i} > D_1$, we have
\[4D_1 < 2p_i < p_{i+1}.\]
\item For $C_0 > 1$, we call $n$ \emph{$C_0$-regular} if, for all $i \le \frac{1}{3}r$,
\[\left|\log \log p_i - \log \log D - i \right|\, <\, C_0^{1/5} \cdot \max(i, \,C_0)^{4/5}.\]
\item We call $n$ \emph{extravagantly spaced} if, for some $m$ in  $\left(0.5 \cdot r^{1/2},\,\, 0.5 \cdot r\right)$, we have
\begin{equation}
\label{eq:extravagant}
\log p_m \,\ge\, \log\left( \frac{\log p_m}{\log D}\right)\cdot (\log \log \log N)^{1/2} \cdot \left(\sum_{i \ge 1}^{m-1} \log p_i \right).
\end{equation}
\end{itemize}
\end{defn}

\begin{thm}
\label{thm:SkND}
Choose $N$, $D$, and $r$ as in the previous definition. Choose $n$ uniformly at random from the set $S_r(N, D)$.
\begin{enumerate}
\item For $D_1 > 3$, the probability that $n$ is not comfortably spaced above $D_1$ is
\[\mathcal{O}\big((\log D_1)^{-1}\big) + \mathcal{O}\big((\log N)^{-1/2}\big)\]
with absolute implicit constant.
\item There is a positive constant $c$ so that, for $C_0 > 0$, the probability that $n$ is not $C_0$ regular is
\[\mathcal{O}\big(\exp(-c \cdot C_0)\big) + \mathcal{O}\left(\exp(-c(\log \log N)^{1/3})\right) \]
with absolute implicit constant.
\item There is a positive constant $c$ so that the probability that $n$ is not extravagantly spaced is bounded by
\[\mathcal{O}\left(\exp\left(-c \cdot (\log \log \log N)^{1/2} \right)\right)\]
with absolute implicit constant.
\end{enumerate}
\end{thm}

\subsection{Proof of Lemma \ref{lem:Iku}}
\begin{proof}
The structure of our argument comes largely from \cite{Soun12}. We take a branch of the logarithm that is holomorphic away from the nonpositive reals and which is real on the positive reals. Following \cite{Soun12}, we have
\begin{equation}
\label{eq:Iku_soun}
I_k(u) = \frac{1}{2\pi i} \int_{c - i\infty}^{c + i\infty} \frac{e^s}{s} \left(\int_{1}^{\infty} \frac{e^{-ts/u}}{t}dt \right)^k ds
\end{equation}
for any positive $c$. We recognize the inner integral as the exponential integral function $E_1(s/u)$, which can be rewritten for $s/u$ off the negative real axis as
\[E_1(s/u) = \int_{s/u}^{\infty} \frac{e^{-z}}{z}dz,\]
where the integral is along any path that does not cross the negative real axis \cite{Abra72}. We also have
\[E_1(s/u) = -\gamma - \log s/u + \sum_{n=1}^{\infty} \frac{(-s/u)^n}{n \cdot n!}.\]
 If $|s/u| < 1$ and $s/u$ is off the negative real axis, we find
\[|E_1(s/u)| \le - \log|s/u| + A\]
with $A$ some absolute constant. If $|s/u| \ge 1$, we instead find
\[|E_1(s/u)| \le |e^{-s/u}|  + A.\]
From this, if $u > k$, we can write
\[\int_{c + iR}^{c + i\infty} \frac{e^s}{s} \big(E_1(s/u) \big)^k ds= \int_{c + iR}^{-\infty + iR} \frac{e^s}{s} \big(E_1(s/u) \big)^k ds \]
for $R > 0$.  Choose $|c + iR| < u$ with $R > e^A$. We have
\[\left|\int_{c+iR}^{-\infty + iR} \frac{e^s}{s} \big(E_1(s/u) \big)^k ds\right| \le \int_{-\infty}^{c} \frac{e^t}{R} \left(\left(e^{-t/u} + A\right)^k + \left(\log u\right)^k\right) dt.\]
Assuming $u > 2k$ and that $u$ is greater than some constant determined from $A$, and choosing $c = 1$, this is bounded by
\[\mathcal{O}\left( \frac{1}{R}(\log u)^k\right).\]
for some choice of constant $A > 0$. Repeating this for negative $R$, we find 
\[\left|I_k(u) - \int_{1-iR}^{1+iR} \frac{e^s}{s} \big(E_1(s/u) \big)^k ds\right| = \mathcal{O}\left( \frac{1}{R}(\log u)^k\right)\]
for $R$ sufficiently large and $|1 + iR| < u$.

 If $z$ has positive real part, we also have
\[\frac{1}{\Gamma(z)} = \frac{1}{2\pi i}  \int_{1 - i\infty}^{1 + i\infty} e^s s^{-z}ds\]
(see \cite[Ch. IX, Misc. Ex. 24]{Cops35}). Then
\[\frac{e^{-\gamma \alpha}}{\Gamma(1 + \alpha)}  = \frac{1}{2\pi i} \int_{1 - i\infty}^{1 + i\infty}  \frac{e^s}{s} e^{-\alpha (\log s + \gamma)}ds.\]
For $R > 1$, we have
\[\left|\int_{1 + iR}^{1 + i \infty} \frac{e^s}{s} e^{-\alpha(\log s + \gamma)}ds\right| = \left|\int_{1 + iR}^{-\infty + iR} \frac{e^s}{s} e^{-\alpha(\log s + \gamma)}ds \right| \]
\[\le \int_{-\infty}^1 \frac{e^s}{R} ds = \frac{e}{R}.\]

Take $R = \log u$ and take $s$ on the segment $[1-iR,\, 1 + iR]$. We have
\begin{align*}
(E_1(s/u))^k &= (\log u)^k \left( 1 - \frac{\log s + \gamma + \mathcal{O}(s/u)}{\log u}\right)^k \\
&  = (\log u)^k \exp\left(-k \frac{\log s + \gamma}{\log u} + \mathcal{O}\left(\frac{k s/u}{\log u} + \frac{k \log^2 s}{\log^2 u}\right)\right) \\
&  = (\log u)^k\exp\left(-k \frac{\log s + \gamma}{\log u} \right) + (\log u)^k \cdot \mathcal{O}\left(\frac{k s/u}{\log u} + \frac{k \log^2 s}{\log^2 u}\right)
\end{align*}
Then, if $R = \log u$ and $u > 2k$ is sufficiently large,
\[I_k(u) = (\log u)^k \frac{e^{-\gamma \alpha}}{\Gamma(1 + \alpha)}+ \mathcal{O}\left((\log u)^k \cdot\left( \frac{k \cdot \log^3 R}{\log^2 u} + \frac{1}{R}\right)\right),\]
within the bounds of the lemma. If $u \le 2k$ or $u$ is small, the lemma is implied by $I_k(u) \le (\log u)^k$.

\end{proof}

\subsection{Proof of Proposition \ref{prop:bad_samp_bhv}}
\begin{proof}[Proof of (1)]
For $i, j \le n$, the probability that $X_i > L_0$ is less than $X_j$ and that the gap $X_j - X_i$ is uncomfortable is
\[\frac{1}{L^2}\int_{L_0}^L \int_{X_i}^{X_i + \delta \cdot e^{-X_i}} dX_j \, dX_i = \frac{1}{L^2} \int_{L_0}^L \delta \cdot e^{-X_i}dX_i \le \frac{1}{L^2}\delta \cdot e^{-L_0}.\]
There are $n(n-1)$ choices of the pair $i, j$, and this number is $\mathcal{O}(L^2)$, so the probability that some pair gives an uncomfortable gap is
\[\mathcal{O}\big(\delta \cdot \exp(-L_0)\big).\]
\end{proof}
\begin{proof}[Proof of (2)]
We note it is sufficient to show that there is a positive constant  $A$ so that, for $C_0 \ge 1$, the probability that
\begin{equation}
\label{eq:UiAC}
\left| U_{(i)} - i\right| \, < \, A\cdot C_0^{1/5} \cdot \max(i,\, C_0)^{4/5}
\end{equation}
does not hold for some $i$ is bounded by $\mathcal{O}(\exp(-C_0))$ with absolute implicit constant.

If $C_0 \ge L$, the proposition is trivial, so we assume $C_0 < L$. We also assume $(L/C_0)^{1/5}$ is an integer. If it is not, we can rechoose $C_0$ from the interval $[C_0, \,32C_0]$ so that this is the case. For $j > 1$, define a sequence
\[\alpha_j = \min\big(C_0 \cdot j^5,\,\, L\big).\]
Take $k_j$ to be the number of $X_i$ less than $\alpha_j$. We can think of $k_j$ as the result of running $k_{j+1}$ Bernoulli trials with success rate $\alpha_j/\alpha_{j+1}$. Then, for any $j$, the probability that
\begin{equation}
\label{eq:yes_binom}
\left|k_j - \frac{\alpha_j}{\alpha_{j+1}} k_{j+1} \right| \le \sqrt{(C_0 + j)k_{j+1}}
\end{equation}
is not satisfied by the sample is bounded by $\exp(- C_0 - j)$ by Hoeffding's inequality. Then the probability that this inequality is not satisfied somewhere is bounded by
\[\sum_{j \ge 1} \exp(-C_0 - j) \le \exp(- C_0).\]

So suppose that \eqref{eq:yes_binom} is satisfied for all $j$. We then claim that, with the proper choice of constant $A$, the sample satisfies \eqref{eq:UiAC} at all $i$. First, we claim that there is some positive constant $B > 0$ so that
\begin{equation}
\label{eq:reg_interm}
\left| k_j - \frac{n}{L}\alpha_j\right| \le B \cdot C_0^{1/5} \cdot \alpha_j^{4/5}.
\end{equation}
for all $j$. This is clear if $\alpha_j = L$, as in this case $k_j = n$. 

We can then proceed by induction. Suppose \eqref{eq:reg_interm} holds for all $j > m$ and that we wish to prove it for $j = m$. Then
\begin{align*}
\left| k_m - \frac{n}{L}\alpha_m\right|  \,\,&\le\,\, \left| k_m - \frac{\alpha_m}{\alpha_{m+1}}k_{m+1}\right| \,\,\, + \,\, \frac{\alpha_m}{\alpha_{m+1}}  \left| k_{m+1} - \frac{n}{L}\alpha_{m+1}\right| \\
&\le \sqrt{(C_0 + m)k_{m+1}} \,\,+ \left(\frac{m}{m+1}\right)^5 \cdot B \cdot C_0^{1/5} \cdot \alpha_{m+1}^{4/5} \\
&= \sqrt{(C_0 + m)k_{m+1}}  \,\,+ \,\,\,\frac{m}{m+1} \,\,\cdot B \cdot C_0^{1/5} \cdot \alpha_{m}^{4/5}
\end{align*}
To prove the inequality, we then just need
\begin{equation}
\label{eq:reg_interm2}
\sqrt{(C_0 + m)k_{m+1}}  \le \frac{1}{m+1} B \cdot C_0^{1/5} \cdot \alpha_{m}^{4/5}.
\end{equation}
The square value of the left hand side of \eqref{eq:reg_interm2} has upper bound
\[\big( C_0 \,\,+\,\, m\big)\cdot \big(C_0 \cdot m^5 \,\,+\,\, B \cdot C_0 \cdot (m+1)^4\big),\]
which can be expanded to a sum of four monomials.
The square value of the right hand side of \eqref{eq:reg_interm2} has lower bound
\[\frac{1}{4} B^2 \cdot C_0^2 \cdot m^6.\] 
For each of the four monomials from the left hand side, we can choose $B_0$ so that the monomial is bounded by
\[\frac{1}{16} B^2 \cdot C_0^2 \cdot m^6\]
if $B > B_0$. For example,
\[C_0 \cdot m^6 \le \frac{1}{16} B^2 \cdot C_0^2 \cdot m^6\]
holds for $B_0 \ge 4$. If $B$ is greater than each $B_0$, then \eqref{eq:reg_interm2} necessarily holds, finishing the induction step. Then \eqref{eq:reg_interm} holds for all $m$.

We now turn to \eqref{eq:UiAC}. If $i \le k_1$, we have
 \[\left|U_{(i)} - i\right| \,\le\, \alpha_1 + k_1 = \mathcal{O}(C_0),\] 
in the bound of the inequality. Now, suppose $i > k_1$. Then $i$ is in some interval $(k_j, \,k_{j+1}]$. We have
\begin{align*}
\left|U_{(i)} - i\right| &= \max\big(U_{(i)} - i,\,\, i - U_{(i)}\big) \\
&\le \max\big(\alpha_{j+1} - k_j,\,\, k_{j+1} - \alpha_j \big) \\
&\le \left|\alpha_j - \alpha_{j+1}\right| + \max\big(\left|\alpha_j - k_j\right|,\, \left|\alpha_{j+1} - k_{j+1}\right|\big)
\end{align*}
We can bound $\alpha_{j+1} - \alpha_j$ by $\mathcal{O}(j^4 \cdot C_0)$. Using $i \ge \alpha_j$, we find
\[\alpha_{j+1} - \alpha_j = \mathcal{O}\left(C_0^{1/5} \cdot i^{4/5}\right).\]
Using \eqref{eq:reg_interm} and the estimate $\frac{L}{n} = 1 + \mathcal{O}(L^{-1/4})$, we also have
\[ \max\big(\left|\alpha_j - k_j\right|,\, \left|\alpha_{j+1} - k_{j+1}\right|\big) = \mathcal{O}\left(C_0^{1/5} \cdot i^{4/5}\right).\]
Then \eqref{eq:UiAC} is satisfied for some sufficiently large constant $A$, giving the part.
\end{proof}

\begin{proof}[Proof of (3)]
For the third part, we note that we can assume that $L$ is larger than some arbitrarily large positive constant $L_0$. 
Define sequences
\[k_j = \lfloor 4^{-k} \cdot n \rfloor \quad\text{and}\quad k'_j = \lfloor 0.5 \cdot 4^{-k} \cdot n\rfloor\]
for $j \ge 0$.  Take $M$ to be the maximal $M$ such that $k_{M+1} \ge \sqrt{n}$. Suppose 
\[u_0 > u_1 > \dots > u_{M+1}.\]
is a sequence of real numbers such that $u_0 = L$ and such that
\[\left|u_j - k_j\right| \le k_j^{5/6}\]
if $j \le M + 1$. We say a sample obeys condition $\mathbf{U}$ if $U_{(k_j )}$ equals $u_j$ for all $j \le M+1$.

For $m \ge \sqrt{n}$ in the interval $(k'_j, k_j]$, we say that $E_m$ is satisfied if
\[U_{(m)} - U_{(m-1)}\,\ge \,\, \log 2 + \log k_j + \frac{1}{2}\log \log L.\]
We say further that $E'_m$ is satisfied if
\[\exp(U_{(m)}) \,\ge\, 2\cdot U_{(m)} \cdot (\log L)^{1/2} \cdot \left(\sum_{i > k_{j+1}}^{m-1} \exp(U_{(i)})\right).\]
Finally, we say that $R_m$ is satisfied if
\[U_{(m)} \ge \frac{3}{2}U_{(k_{j+1})}.\]

For $m$ as above, we have
\begin{align*}
\mathbb{P}\left(E_m \,\big| \, \mathbf{U}\right) &= \left(1 - \frac{\log 2 + \log k_j +\frac{1}{2}\log\log L}{u_j - u_{j+1}}\right)^{k_j - k_{j+1} - 1}\\
&\ge \left(1 - \frac{\log 2 + \log k_j +\frac{1}{2}\log\log L}{\frac{3}{4}k_j - 2k_j^{5/6}}\right)^{\frac{3}{4}k_j - 1}\\
&\ge \exp\left(- \log 2 - \log k_j - \frac{1}{2}\log \log L - o(1) \right)\\
&\ge \frac{e^{-o(1)}}{2 \cdot k_j \cdot \sqrt{\log L}} \ge \frac{1}{3 \cdot k_j \cdot \sqrt{\log L}}
\end{align*}
for a sufficiently large choice of $L_0$.

Next, we note that there is a small positive constant $c_1$ so that, for sufficiently large $L$,
\[\mathbb{P}\left(R_m \,\big|\, \mathbf{U}\right) \ge 1 - e^{-c_1 \cdot k_j}\]
This is an easy consequence of Hoeffding's inequality. We also see that
\[\mathbb{P}\left(E'_m\,\big|\, \mathbf{U},\, R_m,\, E_m\right) \ge \frac{I_{m - k_{j+1}}(\exp(0.5 \cdot u_{j+1}))}{(0.5 \cdot u_{j+1})^{m - k_{j+1}}}.\]
Via Lemma \ref{lem:Iku}, for sufficient $L_0$, this can be bounded from below by some small positive constant $c_2$. Then
\[\mathbb{P}\left(E'_m\, \big|\, \mathbf{U}\right) \ge \frac{c_2}{3 \cdot k_j \cdot \sqrt{\log L}}\, -\, e^{-c_1 \cdot k_j}.\]
For sufficient $L_0$, this is bounded by
\[\frac{c_2}{4 \cdot k_j \cdot \sqrt{\log L}}.\]

We say that $T_j$ is satisfied if $E_m$ is satisfied for some $m$ in $(k'_j, k_j]$, and we say $T'_j$ is satisfied if $E'_m$ is satisfied for some $m$ in this interval. Then we have

\[\mathbb{P}\left(T_j \,\big| \,\mathbf{U}\right) = \mathcal{O}\left(\frac{1}{\sqrt{\log k_j}}\right).\]
For sufficient $L_0$, this is bounded from above by $0.5$. At the same time, we have
\begin{align*}
\mathbb{P}\left(T'_j \,\big| \,\mathbf{U}\right) &\ge\, \sum_{m > k'_j}^{k_j} \mathbb{P}\left(E'_m \,\big|\, \mathbf{U}, \,\overline{E_{m+1}}, \, \dots, \,\overline{E_{k_j}}\right)\,\cdot\,  \mathbb{P}\left(\overline{E_{m+1}},\, \dots, \, \overline{E_{k_j}} \,\big|\, \mathbf{U}\right) \\
&\ge\, \sum_{m > k'_j}^{k_j} \mathbb{P}\left(E'_m \,\big|\, \mathbf{U}\right) \cdot \mathbb{P}\left(\overline{T_j} \,\big|\, \mathbf{U}\right) \\
&\ge\, \frac{c_2(k_j - k'_j)}{8 \cdot k_j \cdot \sqrt{\log L}} \,\,\ge\,\, \frac{c_3}{\sqrt{\log L}}
\end{align*}
for a sufficiently small constant $c_3$.

Given $\mathbf{U}$, the $T'_j$ are independent events. Therefore, given $\mathbf{U}$, the probability that none of the  $T'_j$ are true for $j \le M$ is at most
\[\left(1 -  \frac{c_3}{\sqrt{\log L}}\right)^M = \mathcal{O}\left(\exp\big(-c \cdot \sqrt{\log L}\big)\right)\]
for a sufficiently small constant $c$.

Now, if our sample is $\log L$-regular, then the $U_{(k_j)}$ will all be within $k_j^{5/6}$ of $k_j$ for sufficiently large $k_j$. Then the probability that the sample is not $\log L$-regular or that no $T'_j$ holds is at most
\[\mathcal{O}\left(\exp(-c \cdot \log L)\right) +  \mathcal{O}\left(\exp\big(-c \cdot \sqrt{\log L}\big)\right) =  \mathcal{O}\left(\exp\big(-c \cdot \sqrt{\log L}\big)\right).\]

If the sample is $\log L$-regular, we find that
\[\exp(U_{(k'_j)}) \ge 2\cdot U_{(k_j)} \cdot (\log L)^{1/2} \cdot \left(\sum_{i \le k_{j+1}} \exp(U_{(i)})\right)\]
is true for all $j \le M$ if $L_0$ is sufficiently large. Assuming this, we find that $E_m'$ implies that
\[\exp(U_{(m)}) \,\ge\,  U_{(m)} \cdot (\log L)^{1/2} \cdot \left(\sum_{i =1}^{m-1} \exp(U_{(i)})\right).\]
Because of this, if $L_0$ is sufficiently large, and if the sample is $\log L$-regular and satisfies $T'_j$ for some $j \le M$, we must have that the sample is extravagantly spaced. This gives the proposition.
\end{proof}

\subsection{Proof of Theorem \ref{thm:SkND}}
\label{ssec:pois2}
For positive $x$, we define
\[F(x) = \sum_{p \le x} \frac{1}{p},\]
the sum being over the primes no greater than $x$.

Using the prime number theorem, we know there are constants $A, c > 0$ so that, for all $x \ge 1.5$,
\[\big|F(x) - \log \log x - B_1\big| \le A \cdot e^{-c \sqrt{\log x}},\]
where $B_1$ is the Mertens' constant.

With this in mind, suppose that $T$ is a collection of tuples of primes of length $r$. We define $\text{Grid}(T) \subseteq \mathbb{R}^r$ to be the union
\[\bigcup_{(p_1, \dots, p_r) \in T} \prod_{i \le r} \left[F(p_i) - \frac{1}{p_i} - B_1,\,\, F(p_i)-B_1\right].\]
Now suppose $V \subseteq \mathbb{R}^r$ contains
\[\log \log T = \big\{ \left(\log \log p_1, \dots, \log \log p_r \right) \,:\,\, (p_1, \dots, p_r) \in T\big\}\]
For $(x_1, \dots, x_r) \in \mathbb{R}^k$, we define
\[\tau(x_1, \dots, x_r) \]
\[= \prod_{i \le r} \left[x_i - A_1 \cdot \exp(-c \cdot e^{x_i/2}), \,\,x_i + A_1 \cdot \exp(-c \cdot e^{x_i/2})\right].\]
We then define
\[V^{\text{big}} = \bigcup_{x \in V} \tau(x)\]
and
\[V^{\text{small}} = \{ x \in \mathbb{R}_{\ge}(-B_1)^r \,:\,\, \tau(x) \subseteq V\},\]
where $\mathbb{R}_{\ge}(-B_1)$ is the set of reals $\ge -B_1$.

For a proper choice of the constants $A_1$ and $c$, we then see that, if $T$ is the maximal set of prime tuples such that $\log \log T$ is contained in $V$, we have
\begin{equation}
\label{eq:sm_grd_big}
V^{\text{small}} \subseteq \text{Grid}(T) \subseteq V^{\text{big}}.
\end{equation}
This equation is extremely useful, as we have
\[\text{Vol}\big(\text{Grid}(T)\big) = \sum_{(p_1, \dots, p_r) \in T} \frac{1}{p_1 \cdot \dots \cdot p_r}.\]

For example, for $u$ a positive real and $r$ a positive integer, take $V_r(u) \subseteq \mathbb{R}^r$ to be the set of $(x_1, \dots, x_r)$ satisfying
\[e^{x_1} + \dots + e^{x_r} \le u.\]
We see that
\[\exp\big(x + A_1 \exp(-c \cdot e^{x/2})\big) - \exp(x) \le \kappa\]
for some $\kappa$ depending on $A_1$ and $c$ but not $x$. Then $V_r(u)^{\text{big}}$ is contained in $V_r(u + r\kappa)$, while $V_r(u)^{\text{small}}$ contains
\[V_r(u - r \kappa) \cap \mathbb{R}_{\ge}( -B_1)^r.\]
At the same time, we see that
\[\text{Vol}\left( V_r(u) \cap \mathbb{R}_{\ge}(B)^r\right) = I_r(e^{-B}u).\]
Thus, for $N, D > 0$ and $r$ a positive integer, we have
\[I_r\left(\frac{\log N - r \kappa}{\exp(F(D) - B_1)}\right) \le \sum_{\substack{p_1, \dots, p_r > D \\ p_1 \cdot \dots \cdot p_r < N}} \frac{1}{p_1 \cdot \dots \cdot p_r}\le I_r\left(\frac{\log N + r \kappa}{\exp(F(D) - B_1)}\right).\]
From this, if $r^2 < A \log N$ for some appropriate constant $A$ and $\log \log N - F(D) + B_1 > 1$, we can calculate
\begin{align*}
\sum_{\substack{p_1, \dots, p_r > D \\ p_1 \cdot \dots \cdot p_r < N}} \frac{1}{p_1 \cdot \dots \cdot p_r} &= I_r\left(\frac{\log N}{\exp(F(D) - B_1)}\right)\\
&+ \mathcal{O}\left(\frac{r^2}{\log N} \cdot \big(\log \log N - F(D) + B_1\big)^{r-1}\right).
\end{align*}
We call this sum $F_r(N, D)$.

Choose $\epsilon > 0$. We restrict to the case that  that $\log\log N > (1 + \epsilon) \log \log D \ge 0$, that the ratio
\[u = \frac{\log N}{\exp(F(D) - B_1)}\]
is at least $3$, and that
\[\epsilon \cdot r < \log u < \epsilon^{-1} \cdot r.\]
We define
\[G_r(N, D) = \sum_{\substack{p_1, \dots, p_r > D \\ p_1 \cdot \dots \cdot p_r < N}} \log(p_1 \cdot \dots \cdot p_r)\]
and
\[H_r(N, D) = \sum_{\substack{p_1, \dots, p_r > D \\ p_1 \cdot \dots \cdot p_r < N}} 1.\]
We claim that, subject to the restrictions above, we have

\begin{equation}
\label{eq:GrND}
G_r(N, D) = rN \cdot I_{r-1}(u) + \mathcal{O}\left(\frac{N}{\log N} (\log u)^{r + 3}\right)
\end{equation}
and
\begin{equation}
\label{eq:HrND}
H_r(N, D) = \frac{rN}{\log N} \cdot I_{r-1}\left(u \right) + \mathcal{O}\left(\frac{N}{\log^2 N} (\log u)^{r + 3}\right)
\end{equation}
with implicit constants depending only on $\epsilon$.

To see this, we calculate
\[G_r(N, D) = r \sum_{\substack{p_1, \dots, p_{r-1} > D \\ P < N/D}} \sum_{p > D}^{N/P} \ln p\]
\[ = r \sum_{\substack{p_1, \dots, p_{r-1} > D \\ P < N/D}} NP^{-1}\left(1 + \mathcal{O}\big(e^{-c \sqrt{\log N/P}}\big)\right) - \sum_{p < D} \ln p,\]
\[ =rN \cdot F_{r-1}(N/D, D) - r \cdot H_r(N/D, D) \cdot \sum_{p < D} \ln p\ \]
\[+ \mathcal{O}\left(rNe^{-c\sqrt{\log D} }\big(F_{r-1}(N, D) - F_r(N_0, D)\big) + rNe^{-c \sqrt{\log N/N_0}}F_r(N_0, D)\right)\]
for any choice of $N > N_0 > D$, where we write $P$ for $p_1 \cdot \dots \cdot p_{r-1}$. Choosing $N_0 = Ne^{-(c^{-1} \log \log N)^2}$  fits this term into the error term of \eqref{eq:GrND}. The bounds on $N$ and $D$ allow us to put $F_r(N, D) - F_r(N/D, D)$ in the error term, and we can then estimate $F_r(N, D)$ as above.

We note that, even without the assumptions on $N, D$, the above argument gives
\[G_r(N, D) = \mathcal{O}\big(rN \cdot F_{r-1}(N, D)\big).\]
Then we can calculate
\begin{align*}
H_r(N, D) &= \frac{G_r(N, D)}{\log N} + \int_D^N \frac{G_r(x, D)dx}{x \log^2x}\\
&= \frac{G_r(N, D)}{\log N} + \mathcal{O}\left(\int_D^N \frac{rx \cdot F_{r-1}(N, D)dx}{x\log^2 x}\right)\\
&= \frac{G_r(N, D)}{\log N} + \mathcal{O}\left(\frac{rN}{\log^2 N} (\log u)^{r-1}\right).
\end{align*}
From this we get the two equations.

We see that $H_r(N, D)/r!$ is a slight overestimate for the size of $S_r(N, D)$, with a correction term needed for nonsquarefree integers. This term is of relative size at most
\[\mathcal{O}\left(\sum_{p > D} \frac{1}{p^2}\right) = \mathcal{O}\left(1/D\right)\]
compared to the main term; there is a constant $D_0$ so that, for $D > D_0$, we have
\[|S_r(N, D)| \ge \frac{1}{2 \cdot r!}H_r(N, D)\]
for $N, D, r$ as above. Then, from
\[\frac{1}{2 \cdot r!} H_r\big(N, \,\max(D, \,D_0)\big) \le |S_r(N, D)| \le \frac{1}{ r!}H_r(N, D),\]
we have the following proposition.
\begin{prop}
\label{prop:SrND_est}
Take $1 > \epsilon > 0$, and suppose $N, D$ are real numbers satisfying
\[\log \log N > (1 + \epsilon) \log \log D > 0\]
and that $r$ is a positive integer satisfying
\[1 \le \epsilon \cdot r  \le \log u \le \epsilon^{-1} \cdot r\]
where $u = \frac{\log N}{\exp(F(D) - B_1)}$. Then there are positive constants $C, c$ depending only on $\epsilon$ such that
\[c \cdot \frac{N}{\log N} \frac{(\log u)^{r-1}}{(r-1)!} < \big|S_r(N, D)\big| < C \cdot \frac{N}{\log N} \frac{(\log u)^{r-1}}{(r-1)!} \]
for all sufficiently large $N$.
\end{prop}

Taking $\epsilon$, $N$, $D$, and $r$ as in this proposition, and taking $k \le r$, we define
\[S_{r, \,k}(N, D)\]
to be the subset of $S_r(N, D)$ of elements $n$ with exactly $k$ prime divisors smaller than
\[N_1 = \exp\left(\sqrt{\log N \cdot \exp(F(D) - B_1)}\right).\]
\begin{prop}
\label{prop:SrkND}
Given $\epsilon$, $N$, $D$, and $r$ as above, the density of the set
\[\bigcup_{|r - 0.5k| \,>\, r^{2/3}} S_{r, \,k}(N, D)\]
in $S_r(N, D)$ is bounded by
\[\mathcal{O}\left(\exp(-c(\log \log N)^{1/3})\right)\]
for some $c > 0$, with $c$ and the implicit constant depending only on $\epsilon$.

Now suppose $|0.5r - k| \le r^{2/3}$, and suppose that $T_1$ and $T_2$ are collections of tuples $(p_1, \dots, p_k)$ of distinct primes less than $N_1$ in increasing order. Writing $S_r(N, D, T)$ for the subset of $n$ in $S_{r, \,k}(N, D)$ whose $k$ smallest prime factors $(p_1, \dots, p_k)$ lie in $T$, we have
\[\frac{\big| S_{r,\, k}(N, D, T_1)\big|}{\big| S_{r,\, k}(N, D, T_2)\big|} = \mathcal{O}\left(\frac{\text{\emph{Vol}}\big(\text{\emph{Grid}}(T_1)\big)}{\text{\emph{Vol}}\big(\text{\emph{Grid}}(T_2)\big)}\right)\]
with the implicit constant depending only on $\epsilon$.
\end{prop}
\begin{proof}
We have
\[\left| S_{r,\, k}(N, D)\right| = \sum_{D < p_1 < \dots < p_k < N_1} \left|S_{r-k}\left(\frac{N}{p_1 \cdot \dots \cdot p_k}, D\right)\right|.\]
If we choose $N$ to be large enough, the conditions on this force $N_1^k < \sqrt{N}$. This means the ratio of $(\log \log N - F(D) + B_1)^{r-1}$ and $(\log \log N/P- F(D) + B_1)^{r-1}$ is bounded by some absolute constant for any choice of $P = p_1 \cdot \dots p_k$. The same is true of the ratio of $\log N$ and $\log N/P$. Then there are constants $c, C > 0$ depending only on $\epsilon$ such that
\begin{equation}
\label{eq:choose_k}
c \cdot \frac{1}{P} \left|S_{r}(N, N_1)\right| < \left|S_{r}(N/P, N_1)\right| < C \cdot \frac{1}{P} \left|S_{r}(N, N_1)\right|
\end{equation}
for sufficient $N$. From this, we find
\[\left| S_{r,\, k}(N, D)\right| = \mathcal{O}\left(\frac{N}{\log N} \frac{2^{-r-1}(\log \log N - F(D))^{r-1}}{(r- k + 1)! k!}\right)\]
for $r > k$. Hoeffding's inequality gives the first part of the proposition if we remove the case $r = k$ from the union. The case $r= k$ is insignificant, with its contribution necessarily limited by $N_1^r$, so we have the first part of the proposition. The second part of the proposition is just a recast of \eqref{eq:choose_k} in terms of the corresponding statements for $S_{r,\, k}(N, D)$, and we have the proposition.
\end{proof}

With this proposition, parts two and three of Theorem \ref{thm:SkND} are straightforward. First, we restrict to considering $S_{r,\, k}(N, D)$ with $|k - 0.5r| \le r^{2/3}$. We need not consider other $k$, as the union of all other $S_{r,\,k}(N, D)$ fits in the error bound.

Take $T_2$ to be the set of $k$-tuples of distinct primes from the interval $(D, N_1)$. We find
\[\text{Vol}(\text{Grid}(T_2)) \ge c (\log \log N_1 - \log \log D)^k.\]
For part (2), take $T_1$ to be the set of non-$C_0$-regular prime tuples in $T_2$. The biggification of the grid of $T_1$ consists of samples that are not $C_0 - \kappa$ regular for some constant $\kappa > 0$ not depending on $C_0$. From Proposition \ref{prop:bad_samp_bhv}, we then find that the volume of this bigification is bounded by
\[\mathcal{O}\left( \exp(-c \cdot C_0) \cdot (\log \log N_1 - \log \log D)^k\right).\]
Then Proposition \ref{prop:SrkND} gives the part.

For part (3), take $T_1$ to be the set of prime tuples so, for $m > k^{1/2}$,
\[\log p_m \le \log\left(\frac{\log p_m}{\log D}\right) \cdot (\log \log \log N)^{1/2}\cdot \left(\sum_{i \ge 1}^{m-1} \log p_i \right).\]
The bigification of this grid consists of tuples $(x_1, \dots, x_k)$ so, for $m > k^{1/2}$,
\[ e^{x_m} \le A x_m \cdot (\log \log \log N)^{1/2} \cdot \left(\sum_{i = 1}^{m-1} e^{x_i}\right)\]
for some absolute constant $A$. Repeating part (3) of the proof Proposition \ref{prop:bad_samp_bhv} to take account of the $A$, we find that the volume of this bigification is
\[\mathcal{O}\left( \exp\big(-c \cdot (\log \log \log N)^{1/2}) \cdot (\log \log N_1 - \log \log D)^k\right).\]
Then Proposition \ref{prop:SrkND} gives the part.

For the first part of the theorem, we opt to start from scratch. The number of uncomfortably spaced examples is bounded by
\[\sum_{p > D_1}^{N} \sum_{q > p}^{2p} \left| S_{r-2}(N/pq,\, D)\right| \]
We can restrict this to the case that $p < N^{1/4}$ and $p \ge N^{1/4}$. The former case has size boundable by
\[\mathcal{O}\left( \left| S_{r-2}(N,\, D)\right|  \cdot  \sum_{p > D_1}^N  \sum_{q > p}^{2p}  \frac{1}{pq}\right) = \mathcal{O}\left( \left| S_r(N,\, D)\right|  \cdot  (\log D_1)^{-1} \right)\]
The sum over $p \ge N^{1/4}$ can be bounded by $\mathcal{O}(N/\log N)$, so we have the result if we can show
\[\frac{(\log \log N - F(D) + B_1)^{r-1}}{(r-1)!} \ge \sqrt{\log N}\]
for all sufficiently large $N$. This is a simple consequence of \eqref{eq:r_ND_bnd} and the restriction $\log D < (\log N)^{1/4}$, and we have the theorem.
\qed

\section{Equidistribution of Legendre symbols}
\label{sec:Leg}
From elementary work of R{\'e}dei and Reichardt \cite{ReRe33}, we know that the rank of the $4$-class group of a quadratic field with discriminant $\Delta$ can be determined from the kernel of a matrix of Legendre symbols $\left(\frac{d}{p}\right)$, where $d$ varies over the divisors of $\Delta$ and $p$ varies over the odd prime divisors of $\Delta$. For quadratic twists of elliptic curves $E$ with full two torsion and no rational cyclic subgroup of order four, we can also give the $2$-Selmer rank as the kernel of a certain matrix of Legendre symbols.

In \cite{Swin08}, under the assumption that the associated Legendre symbol matrices were uniformly distributed among all posibilites, Swinnerton-Dyer found the distribution of $2$-Selmer ranks among the set of all twists. Kane then proved that the actual distribution of Selmer ranks agreed with the distribution found by Swinnerton-Dyer \cite{Kane13}. In contrast to the work of Fouvry and Kl{\"u}ners on $4$-class groups \cite{Fouv07} and the work of Heath-Brown on the congruent number problem \cite{Heat94}, Kane's work relied on the fact that the distribution over Legendre symbol matrices had already been found, thus streamlining the argument in a way that hadn't previously been possible. However, in line with the prior papers, Kane's eventual argument was that the moments of the $2$-Selmer groups were consistent only with the claimed distribution.

As an alternative to Kane's approach, we can prove Swinnerton-Dyer's assumption is correct. If we set up our sets of integers correctly, we can prove that the matrices of Legendre symbols are essentially equidistributed; this is the content of Proposition \ref{prop:eqd_lgn_raw} and Theorem \ref{thm:Lgn_perm}. Both of these results concern points chosen from a product space of increasing intervals of primes; we will make the translation to from arbitrary integers to such product spaces in Section \ref{ssec:box}.

Our first task is to concretely define what we mean by matrices of Legendre symbols.
\begin{defn}
\label{def:Lgn_matrix}
Take $P_0$ to be an arbitrary set of prime numbers, and take $P = \{-1\} \cup P_0$. Choosing $r > 0$, take $\mathscr{M}$ to be some subset of 
\[\big\{\{i,\,j\}\,:\,\, i, j \in [r]\}\]
and take $\mathscr{M}_P$ to be some subset of $[r] \times V$. Also take $a$ to be an arbitrary function from $\mathscr{M} \cup \mathscr{M}_P$ to $\pm 1$.

Take $X_1, \dots, X_r$ to be disjoint sets of odd primes not meeting $P$, and write $X$ for the product of the $X_i$. We then define $X(a)$ to be the set of $(x_1, \dots, x_r)$ in $X$ satisfying
\[\left(\frac{x_i}{x_j}\right) = a\big(\{i, \,j\}\big) \text{ for all } i < j \text{ with } \{i, \,j\} \in \mathscr{M}\]
and
\[\left(\frac{d}{x_j}\right) = a\big((i, \, d)\big)\text{ for all }  (i, \,d) \in \mathscr{M}_P.\]
\end{defn}
Our goal is to find situations where the order of $|X(a)|$ is well approximated by $2^{-|\mathscr{M}_P \cup \mathscr{M}|} \cdot |X|$. To do this unconditionally, we need to account for the possibility of Siegel zeros in the $L$-functions of the associated quadratic characters. We use the following definition.
\begin{defn}
\label{def:Sieg}
For $c > 0$, take $\text{Sieg}(c)$ to be the set of squarefree integers $d$ so that the quadratic character $\chi_d$ associated with $\QQ(\sqrt{d})/\QQ$ has Dirichlet $L$-function satisfying
\[L(\chi_d, s) = 0 \quad\text{for some}\quad 1 \ge s \ge 1 - c(\log 2d)^{-1}.\]
We can order $\text{Sieg}(c)$ by increasing magnitude, getting a sequence $d_1, d_2, \dots$. By Landau's theorem (see \cite[Theorem 5.28]{Iwan04}), we can choose an absolute $c$ that is sufficiently small so that
\[d_i^2 < \big| d_{i+1}\big|\]
for all $i \ge 1$. Fix such a choice of $c$ for this entire paper. We call a given $d$ \emph{Siegel-less} if it is outside the sequence of such $d_i$ 
\end{defn}
We can now state our first result.

\begin{prop}
\label{prop:eqd_lgn_raw}
Choose positive constants $c_1, c_2, c_3, c_4, c_5, c_6, c_7, c_8$. We presume that $c_3 > 1$, that $c_5 > 3$, and that
\[\frac{1}{8} > c_8 + \frac{c_7 \log 2}{2} + \frac{1}{c_1} + \frac{c_2c_4}{2}.\] 
Then there is a constant $A$ depending only on the choice of these constants so that we have the following:

Choose sets $X_1, \dots, X_r$ and a function $a$ with associated set $P$ as in Definition \ref{def:Lgn_matrix}, and choose a sequence of real numbers
\[A < t < t_1 < t_1' < t_2 < t_2' < \dots < t_r < t_r'\]
such that $X_i$ is a subset of $(t_i,\, t'_i)$ for each $i \le r$. Choose $1 \le k \le r$ such that, if $\mathscr{M}_P$ contains some element of the form $(i,\, d)$, we have $i > k$. For $i > k$, we assume that
\[X_i = \left\{t_i < p < t'_i \,:\,\, \left(\frac{d}{p}\right) = a\big((i, d)\big) \text{ for all } (i, d) \in \mathscr{M}_P\right\}.\]

For $D_P$ any product of elements in $P$ and $D_X$ any product over $i$ of at most one element from each $X_i$, we assume that $D_PD_X$ is Siegel-less if $|D_PD_X| > t$.

We also assume that we have the following:
\begin{enumerate}
\item We assume that all primes in $P$ are less than $t_1'$.
\item  We assume that
\[t_1' > r^{c_1}\quad\text{and}\quad t'_{k} < \exp({t_1'}^{c_2}).\]
\item We assume that, for $1\le i \le r$, we have
\[|X_i| \ge \frac{2^{c_3i} \cdot t'_i}{(\log t'_i)^{c_4}} \quad\text{and}\quad |P| \le \log t'_i - i.\]
\item If $k \ne r$, we assume that 
\[t'_{k+1} > \exp\big((\log t_1')^{c_5}\big),\, \,\exp\big(t^{c_6}\big).\]
\item We assume that $k < c_7 \log t_1'$ and that, for any $i \le r$ and any $j$ satisfying $r \ge j \ge i -2 + c_7\log t'_i$, we have
\[ \exp\big((\log t_i')^{c_5}\big)  < t'_j.\]

\end{enumerate}
Then
\[\left||X(a)| \,-\, 2^{-|\mathscr{M}|} |X|\right| \,\le\,\, {t_1'}^{-c_8}\cdot 2^{-|\mathscr{M}|}  |X|.\]

\end{prop}

The lower bounds assumed of $t_1$ in the above proposition are essential for the proposition to be correct; we do not have sufficiently strong control on the Legendre symbols involving only small primes to give the equidistribution result we want. However, there is a combinatorial trick that allows us to circumvent this bad behavior:
\begin{defn*}
Given $X = X_1 \times \dots \times X_r$ and $a$ as in Definition \ref{def:Lgn_matrix}, and given a permutation $\sigma: [r] \rightarrow [r]$, we take
\[X(\sigma, \,a) = (X_{\sigma(1)} \times X_{\sigma(2)} \times \dots \times X_{\sigma(r)})(a).\]
Given $k_2 \le r$, we take $\mathscr{P}(k_2)$ to be the set of permutations of $[r]$ that are the identity outside of $[k_2]$.
\end{defn*}

There are two key points to make about this definition. First, when matrices of Legendre symbols are used to find $2$-Selmer groups or $4$-class groups, the order of the prime factors of the quadratic twist or discriminant do not effect the eventual rank. Because of this, the Selmer or class structure of a point in $X(\sigma, \, a)$ is not affected by the choice of $\sigma$.

Second, the application of $\sigma$ to $X$ has the effect of mixing the bad corner of the Legendre symbol matrix in with the rest of the matrix. As a result, the average of $|X(\sigma,\, a)|$ over all choices of $\sigma$ is almost independent of the choice of $a$, as we detail in the next theorem. 

\begin{thm}
\label{thm:Lgn_perm}
Choose positive constants $c_1, \dots, c_{12}$. We presume that $c_3 > 1$, that $c_5 > 3$, that
\[\frac{1}{8} > c_8 + \frac{c_7 \log 2}{2} + \frac{1}{c_1} + \frac{c_2c_4}{2},\]
and that
\[c_{10}\log 2 + 2c_{11} + c_{12} < 1 \quad\text{and}\quad c_{12} + c_{11} < c_9.\]  
Then there is a constant $A > 0$ depending only on the choice of these constants so that we have the following:

Choose a sequence of real numbers
\[ t_1 < t_1' < t_2 < t_2' < \dots < t_r < t_r',\]
and a positive number $t > A$, and take $X_i$ to be the set of primes in the interval $(t_i, \,t'_i)$. Write $X$ for the product of the $X_i$, and take $a$ to be as in Definition \ref{def:Lgn_matrix}.  Choose integers $k_0, k_1, k_2$ satisfying $0 \le k_0 \le k_1 < k_2 \le r$ and $t'_{k_0 + 1} > t$. We will write $t'$ for $t'_{k_0 + 1}$.

For $D_P$ any product of elements in $P$ and $D_X$ any product over $i$ of at most one element from each $X_i$, we assume that $D_PD_X$ is Siegel-less if $|D_PD_X| > t$. We also assume $k_2 > A$.

We also make the following assumptions:
\begin{enumerate}
\item We assume that all primes in $P$ are less than $t'$.
\item We assume that
\[t' > r^{c_1}\quad\text{and}\quad t'_{k_1} < \exp({t'}^{c_2}).\]
\item For $i > k_0$, we assume that
\[|X_i| \ge \frac{2^{|P| + c_3i} \cdot k_2^{c_9} \cdot t'_i}{ (\log t'_i)^{c_4}} \quad\text{and}\quad |P| \le \log t'_i - i.\]
\item If $k_1 \ne r$, we assume that 
\[t'_{k_1 + 1} > \exp\big((\log t_1')^{c_5}\big),\,\,\exp\big(t^{c_6}\big)\]
\item We assume that $k_1 - k_0 < c_7 \log t'$ and that, for any $k_0 < i \le r$ and any $j$ satisfying $r \ge j \ge i -2 + c_7\log t'_i$, we have
\[ \exp\big((\log t_i')^{c_5}\big)  < t'_j.\]
\item We assume that
\[c_{10} \log k_2 > |P| + k_0 \quad\text{and}\quad c_{11} \log k_2 > \log k_1.\]
\end{enumerate}

Then, for any choice of the subsets $\mathscr{M}$ and $\mathscr{M}_p$, we have
\[\sum_{a \in \mathbb{F}_2^{\mathscr{M} \cup \mathscr{M}_p}} \left|2^{-|\mathscr{M} \cup \mathscr{M}_P|}\cdot k_2! \cdot |X| \, -\,  \sum_{\sigma \in \mathscr{P}(k_2)} \big|X(\sigma,\, a)\big| \right| \]
\[\le \left(\left(k_2^{-c_{12}} + {t'}^{-c_8}\right)  \cdot k_2! \cdot |X|\right) .\] 

\end{thm}

\subsection{Equidistribution via Chebotarev and the Large Sieve}
There are two main methods to predict the distribution of $\left(\frac{d}{p}\right)$ over a given set of $p$. First, if $d$ is small relative to the set of $p$, then we can use the Chebotarev density theorem to predict the distribution. On the other hand, if $d$ is similarly sized to the set of $p$, we can use the large sieve results of Jutila to predict the distribution of these symbols for most $d$ \cite{Juti75}. The key to proving this proposition is combining these two tools properly.

We start with the form of the Chebotarev density theorem that we will be using. For the proof of Proposition \ref{prop:eqd_lgn_raw}, we only need to apply this proposition with $L = \QQ$. However, the full power of this proposition, including the cumbersome form of the error term, will be necessary to prove our results for $2^k$-Selmer groups and class groups.

\begin{prop}
\label{prop:2Cheb}
Suppose $M/\QQ$ is a Galois extension and $G = \emph{\text{Gal}}(M/\QQ)$ is a $2$-group. Suppose $M$ equals the composition $KL$, where $L/\QQ$ is Galois of degree $d$ and $K/\QQ$ is an elementary abelian extension, and where the discriminant $d_L$ of $L/\QQ$ and the discriminant $d_K$ of $K/\QQ$ are relatively prime. Take $d_{K_0}$ to be the maximal absolute value of the discriminant of a quadratic subfield $K_0$ of $K$.

Take $F: G \rightarrow [-1, 1]$ to be a class function of $G$ with average over $G$ equal to zero. Then there is an absolute constant $c > 0$ such that
\[\sum_{p \le x} F\left(\left[\frac{M/\QQ}{p} \right]\right) \cdot \log p \]
\[= \mathcal{O}\left(x^{\beta} \cdot |G| \,+\, x \cdot |G| \cdot  \exp\left(\frac{-cd^{-4} \log x}{\sqrt{\log x} + 3d \log \big| d_{K_0}d_L\big|}\right) \big(d^2 \log \big|xd_{K_0}d_L\big|\big)^4\right)\]
for $x \ge 3$, where $\beta$ is the maximal real zero of any Artin $L$-function defined for $G$, the term being ignored if no such zero exists. The implicit constant here is absolute.
\end{prop}
\begin{proof}
Take $\rho$ to be a nontrivial irreducible representation of $G$. As it is a $p$-group, $G$ is nilpotent and is hence a monomial group, so the Artin conjecture is true for $\rho$. That is, the Artin $L$-function $L(\rho, s)$ is entire. The representations $\rho \otimes \rho$ and $\rho \otimes \overline{\rho}$ also satisfy the Artin conjecture, so
\[L(\rho \otimes \overline{\rho}, s)\]
is entire except for a simple pole at $s = 1$, and
\[L(\rho \otimes \rho, s)\]
is entire unless $\rho$ is isomorphic to $\overline{\rho}$.

Then \cite[Theorem 5.10]{Iwan04} applies for $L(\rho, s)$. We also see that \cite[(5.48)]{Iwan04}holds for this $L$ function by the argument given in \cite{Iwan04} after this equation. Then Theorem 5.13 of \cite{Iwan04} applies for this $L$-function. We note that $\rho$ is defined on $\text{Gal}(K_0L/\QQ)$ for some quadratic extension $K_0/\QQ$ inside $K$, so its degree is bounded by $2d$ and the conductor of $L(\rho, s)$ is bounded by the discriminant of $K_0L/\QQ$, which is bounded by
\[d_{K_0}^d \cdot d_L^2.\]
Then \cite[Theorem 5.13]{Iwan04} gives
\[\sum_{p \le x} \chi_{\rho}\left(\left[\frac{M/\QQ}{p} \right]\right) \cdot \log p\]
\[=  \mathcal{O}\left(x^{\beta} \,+\, x \cdot  \exp\left(\frac{-cd^{-4} \log x}{\sqrt{\log x} + 3d \log \big| d_{K_0}d_L\big|}\right) \big(d^2 \log \big|xd_{K_0}d_L\big|\big)^4\right).\]
Now, we can write $F$ in the form $\sum_{\rho} a_{\rho}\chi_{\rho}$, the sum being indexed by the nontrivial irreducible representations of $G$. Then
\begin{align*}
\sum_{\rho} \big| a_{\rho}\big| &= \sum_{\rho} \left| \frac{1}{|G|} \sum_{g \in G} F(g)\cdot \overline{\chi_{\rho}}(g)\right| \\
& \le \sum_{\rho} \left(\frac{1}{|G|} \sum_{g \in G} F(g) \cdot \overline{F}(g)\right)^{1/2} \cdot \left(\frac{1}{|G|} \sum_{g \in G} \chi_{\rho}(g) \overline{\chi_{\rho}}(g)\right)^{1/2} \\
&\le \sum_{\rho} 1 \le |G|.
\end{align*}
We then get the proposition.
\end{proof}

We now give the form of the large sieve we use.
\begin{prop}
\label{prop:jutila}
Take $X_1$ and $X_2$ to be disjoint sets of odd primes with upper bounds $t'_1$ and $t'_2$ respectively. Then, for any $\epsilon > 0$, we have
\[\sum_{x_1 \in X_1} \left| \sum_{x_2 \in X_2} \left(\frac{x_1}{x_2}\right) \right| = \mathcal{O}\left( t_1' \cdot {t_2'}^{3/4 + \epsilon} +\, t_2' \cdot {t_1'}^{3/4 + \epsilon}\right),\]
with the implicit constant depending only on the choice of $\epsilon$.
\end{prop}
\begin{proof}
By \cite[Lemma 3]{Juti75}, we have
\[\sum_{x_1 \in X_1} \left( \sum_{x_2 \in X_2} c_2(x_2)\left(\frac{x_1}{x_2}\right) \right)^2 = \mathcal{O}\bigg( t'_1 \cdot |X_2| + {t_1'}^{1/2} \cdot {t_2'}^2 \log^6 t_2'\bigg)\]
for any function $c_2: X_2 \rightarrow \pm 1$. Then, for $c_1: X_1 \rightarrow \pm 1$, Cauchy's theorem gives
\[\sum_{\substack{x_1 \in X_1 \\ x_2 \in X_2}} c_1(x_1)c_2(x_2) \left(\frac{x_1}{x_2}\right) =  \mathcal{O}\bigg( t'_1 \cdot   {t_2'}^{1/2} \,+\, {t_1'}^{3/4} \cdot {t_2'} \log^3 t_2'\bigg).\]
Choosing $a_1, a_2 \in \pm 1$, we apply the above estimate to the subsets of $X_1, X_2$ where $x_1 \equiv a_1 \, (4)$ and $x_2 \equiv a_2 \, (4)$. For each of the four possibilities of $(a_1, a_2)$, we have that $\left(\frac{x_1}{x_2}\right)\left(\frac{x_2}{x_1}\right)$ is constant by quadratic reciprocity, and we deduce the proposition.
\end{proof}

\begin{proof}[Proof of Proposition \ref{prop:eqd_lgn_raw}]
We will show that, subject to the assumptions of the proposition, we have
\[\left||X(a)| \,-\, 2^{-|\mathscr{M}|} |X|\right| \,\le\,\, r \cdot{ t_1'}^{-c_8 - \frac{1}{c_1}} \cdot 2^{-|\mathscr{M}|}A  |X|.\]
The bound on $t_1'$ shows that this implies the proposition. We proceed by induction on $r$. 
The statement is obvious for $r = 1$, where $\mathscr{M}$ is always empty. 

Now, suppose we wish to show it for $X_1 \times \dots \times X_r$, once we know the result for every product of length $r-1$. To this end, for $x_1 \in X_1$, take $X_i(a, \, x_1)$ to be the subset of elements $x_i$ in $X_i$  satisfying
\[\left(\frac{x_1}{x_i}\right) = a(\{1, \, i\})\]
should $\{1,\, i\}$ lie in $\mathscr{M}$.

If $\{1,\, i\}$ is  in $\mathscr{M}$ for $i \le k$, we apply Proposition \ref{prop:jutila} to say that
\[\sum_{x_1 \in X_1} \left| \sum_{x_i \in X_i} \left(\frac{x_1}{x_i}\right) \right| = \mathcal{O}\left( t_i' \cdot {t_1'}^{3/4 + \epsilon}\right).\]
Then, for any $\epsilon > 0$, the bounds on the size of the $X_i$ then force
\[\sum_{x_1 \in X_1} \left| \sum_{x_i \in X_i} \left(\frac{x_1}{x_i}\right) \right| < {t_1'}^{-\frac{1}{4}+c_4c_2 + \epsilon } \cdot |X_1| \cdot |X_i|\]
for sufficiently large $A$. Choosing constants $c_a, c_b > 0$ with
\[c_a + c_b < \frac{1}{4} - c_4c_2,\]
we expect that, for all $x_1 \in X_1$ besides at most $k \cdot t_1'^{-c_a} \cdot |X_1|$ exceptions, we have
\[\big||X_i(a, \, x_1)|\, -\, 0.5 |X_i|\big| < t_1'^{-c_b} \cdot |X_i| \quad\text{for all} \quad\{1,\, i\} \in \mathscr{M}.\]
Write $X_1^{\text{bad}}$ for the set of exceptional $x_1$. We will choose
\[c_a > c_7 \log 2   + c_8 + \frac{1}{c_1}\]
and
\[c_b > c_8 + \frac{1}{c_1}.\]
Given the conditions on the constants, it is always possible to find such $c_a, c_b$.

Meanwhile, suppose $\{1,\, i\}$ is in $\mathscr{M}$ with $i > k$. We apply Proposition \ref{prop:2Cheb} to the field $M$ generated by $\sqrt{x_1}$ and by all $\sqrt{d}$ with $d$ in $P$. Take $F:\text{Gal}(M/\QQ) \rightarrow [-1, +1]$ to equal $1 - 2^{-|P| - 1}$ for $\sigma$ corresponding to the Frobenius class of the elements of $X_i(a, \, x_1)$, and to otherwise equal $-2^{-|P|-1}$. We are interested in bounding
\[\sum_{p \le t'_i} F\left(\left[\frac{M/\QQ}{p} \right]\right) \cdot \log p.\]
Choose a constant $c_c > 0$. By Siegel's theorem \cite[Theorem 5.28]{Iwan04}, we can choose a large enough constant $A$ so 
\[(1- \beta) >t^{-c_c}\]
if $\beta$ is an exceptional real zero of the $L$ function corresponding to some $\chi_D$ with $|D| < t$. Then, when applying Proposition \ref{prop:2Cheb}, we find
\[x^\beta < t'_i \exp\big(-(\log t'_i)^{1 - \frac{c_c}{c_6}} \big).\]
From $\log |d_{K_0}| = \mathcal{O}\big((\log t_1')^2\big)$, we can also bound
\[\exp\left(\frac{-cd^{-4} \log x}{\sqrt{\log x} + 3d \log \big| d_{K_0}d_L\big|}\right) \le \exp\big(-(\log t_i')^{1/3 + \epsilon}\big)\]
for some constant $\epsilon > 0$. From $|G| \le t'_1$, we then find that we can write
\[\sum_{p \le t'_i} F\left(\left[\frac{M/\QQ}{p} \right]\right)  \cdot \log p  \le t'_i\exp\big(-(\log t_i')^{1/3 + \epsilon}\big)\]
for sufficient $A$. By reweighting this series, we can also show that
\[\sum_{p \le t'_i} F\left(\left[\frac{M/\QQ}{p} \right]\right) s \le t'_i\exp\big(-(\log t_i')^{1/3 + \epsilon}\big).\]
Then, for $i \ge k$, we always find
\[\big||X_i(a, \, x_1)|\, -\, 0.5 |X_i|\big| < t_1'^{-1} \cdot |X_i|.\]

Write $X^{\text{bad}}(a)$ for the subset of $X(a)$ with $x_1$ in $X_1^{\text{bad}}$. Our first task is to bound this set. Choose $x_1 \in X^{\text{bad}}(a)$, and add it to $P$, shifting its conditions from $\mathscr{M}$ to $\mathscr{M}_P$. Consider the product
\[X_2 \times \dots \times X_{k} \times X_{k+1}(a,\, x_1) \times \dots \times X_r(a,\, x_1).\]
This product has length $r-1$. Once we shift up $k$, it obeys all the conditions of the proposition, so the induction step tells us that the subset of $X(a)$ starting with $x_1'$ has size at most
\[2^{-|\mathscr{M}| + k + 1} \frac{|X|}{|X_1|}.\]
Then $X^{\text{bad}}(a)$ has size bounded by
\[2^{k+1} \cdot {t'_1}^{-c_a} \cdot 2^{-|\mathscr{M}|}|X|.\]
On the good side, we instead look at the product
\[X_2(a,\, x_1) \times \dots \times X_r(a,\, x_1).\]
From this and the induction step, we find that the subset of $X(a)$ starting with a good $x_1$ has size at most
\[2^{-|\mathscr{M}|} \frac{|X|}{|X_1|} \cdot (1+(r-1) \cdot{ t_1'}^{-c_8 - \frac{1}{c_1}}) \cdot (1 + {t_1'}^{-c_b})^k \cdot (1 + {t_1'}^{-1})^r\]
and at least
\[2^{-|\mathscr{M}|} \frac{|X|}{|X_1|}\cdot (1-(r-1) \cdot{ t_1'}^{-c_8 - \frac{1}{c_1}}) \cdot (1 - {t_1'}^{-c_b})^k \cdot (1 - {t_1'}^{-1})^r.\]

We have $k < c_7 \log t'_1$, so the term $(1 + {t_1'}^{-c_b})^k$ gives error fitting into $ { t_1'}^{-c_8 - \frac{1}{c_1}}$ from our lower bound on $c_b$.  The term $(1 + {t_1'}^{-1})^r$ gives error fitting into this whenever $1 > c_8 + 2c_1^{-1}$, which is always satisfied.

Next, we see that the contribution from $X^{\text{bad}}(a)$ fits into the error bound whenever $2^k \cdot {t'_1}^{-c_a}$ fits into ${ t_1'}^{-c_8 - \frac{1}{c_1}}$. From $k < c_7 \log t'_1$ and the bounds on $c_a$, we find that this is also the case. This gives the proposition.
\end{proof}

\subsection{Combinatorial approaches to small primes}
We now wish to prove Theorem \ref{thm:Lgn_perm} from Proposition \ref{prop:eqd_lgn_raw}. This transition is entirely combinatorial: the key is the following proposition.

\begin{prop}
\label{prop:perm_helps}
Choose $X$, $P$, $\mathscr{M}$, and $\mathscr{M}_P$ as in Definition \ref{def:Lgn_matrix}. We assume $\mathscr{M}$ and $\mathscr{M}_P$ are maximal given $r$ and $P$. 

Choose integers $0 \le k_0 \le k_1 \le k_2 \le r$ so that
\[2^{|P| + k_0 + 1}\cdot k_1^2 < k_2.\]
For $\sigma$ a permutation of $[r]$ and $a$ as in Definition \ref{def:Lgn_matrix}, take $X_C(\sigma, \, a)$ to be the set of $x = (x_1, \dots, x_r)$ in $X$ so that
\[\left(\frac{d}{x_j}\right) = a\big((\sigma^{-1}(j),\, d)\big) \quad\text{for all}\quad (j,\, d) \in [k_1] \times P\quad \text{and}\]
\begin{align*}
\left(\frac{x_i}{x_j}\right) = a\big(\{\sigma^{-1}(i),\, \sigma^{-1}(j)\}\big) \,\,\,&\text{whenever}\,\,\, i, j\le k_1 \text{ and } \sigma^{-1}(i) \le \sigma^{-1}(j)\\
& \text{and either } i \le k_0 \text{ or } j \le k_0.
\end{align*}
Write $m_C$ for the number of Legendre symbol conditions specified; that is,
\[m_C = k_1|P| + \frac{1}{2}(k_0^2 - k_0) + k_0(k_1 - k_0).\]

Then, for any $x \in X$, we have
\[\sum_{a \in \mathbb{F}_2^{\mathscr{M} \cup \mathscr{M}_P}} \bigg( 2^{-m_C}\cdot k_2! \, - \, \big\{\sigma \in \mathscr{P}(r_2)\,:\,\, x \in X_C(\sigma, \, a)\big\}\bigg)^2 \]
\[\le\frac{ 2^{|P| + k_0 + 1} \cdot k_1^2}{k_2}\, \cdot\, 2^{-2m_C\, +\, |\mathscr{M} \cup \mathscr{M}_P|} \cdot k_2!^2.\]
\end{prop}

\begin{proof}
Write $W(a)$ for $\big\{\sigma \in \mathscr{P}(r_2)\,:\,\, x \in X_C(\sigma, \, a)\big\}$. We see that the average size of $W(a)$ over all $a$ is $2^{-m_C} \cdot r_2!$, as the condition that $x$ is in $X_C(\sigma,\, a)$ for a given $x$ and $\sigma$ is given  by $m_C$ binary conditions on $a$.

We now consider the average of $|W(a)|^2$. We see that $|W(a)|^2$ is the number of permutation pairs $(\sigma_1, \sigma_2)$ so that $x$ is in $X_C(\sigma_1,\, a)$ and $X_C(\sigma_2,\, a)$. Write $W(\sigma_1, \, \sigma_2)$ for the set of $a$ so that $x$ is in both of these sets.

The maximal number of conditions on an $a$ in $W(\sigma_1,\, \sigma_2)$ is $2m_C$; a lower bound on the number of conditions depends on $\sigma_1$ and $\sigma_2$. Take $d_1$ to be the number of $i \in [r]$ so that $\sigma^{-1}_1(i)$ and $\sigma^{-1}_2(i)$ are both at most $k_1$. Then we see that $W(\sigma_1,\, \sigma_2)$ is determined by at least
\[2m_c - d_1(|P| +k_0) \]
conditions. So 
\[|W(\sigma_1,\, \sigma_2)| \le 2^{-2m_c + d_1(|P| +k_0) + |\mathscr{M} \cup \mathscr{M}_P|}\]

At the same time, the number of ways to choose a permutation $\pi$ of $[k_2]$ so that 
\[\bigg|\pi\big([k_1]\big) \,\cap\, [k_1]\bigg|\,\ge\, d\]
is bounded by the number of ways to choose two cardinality $d$ subsets from $[k_1]$ and a bijection between these sets and a bijection between their complements in $[k_2]$. This is bounded by
\[d! \cdot \binom{k_1}{d}^2 \cdot (k_2 - d)!  \le \left(\frac{k_1^2}{k_2}\right)^d \cdot k_2!\]

Then the mean value of $|W(a)|^2$ is bounded by
\[\sum_{d \ge 0} 2^{-2m_c + d(|P| +k_0)}  \left(\frac{k_1^2}{k_2}\right)^d \cdot k_2!^2.\]
This is a geometric sum; combining this with our calculation of the mean of $|W(a)|$ then gives the proposition.

\end{proof}

\begin{proof}[Proof of Theorem \ref{thm:Lgn_perm}]
Without loss of generality, we may assume that $\mathscr{M}$ and $\mathscr{M}_P$ are both maximal as in Proposition \ref{prop:perm_helps}. We also define $m_C$ and $X_C(\sigma,\, a)$ as in that proposition, and we assume that $X_1, \dots, X_{k_0}$ are singletons $x_1, \dots, x_{k_0}$.

We prove the theorem by bounding
\begin{align*}
&\sum_a \bigg| k_2! |X| \, -\,   2^{m_C}\cdot  \sum_{\sigma \in \mathscr{P}(k_2)} |X_C(\sigma,\, a)|\bigg| \\
&\qquad+\,\,  \sum_{\sigma \in \mathscr{P}(k_2)} \sum_a \bigg| 2^{m_C} \cdot |X_C(\sigma,\, a)| \,-\, 2^{|\mathscr{M} \cup \mathscr{M}_P|} |X(\sigma,\, a)|\bigg|.
\end{align*}
The former we can bound via the previous proposition by
\[k_2^{-c_{12}} \cdot |X| \cdot 2^{|\mathscr{M} \cup \mathscr{M}_P|} \cdot k_2!\]
For the second sum, fix a $\sigma$ and a choice of $a$ outside of the values referenced in the definition of $X_C(\sigma,\, a)$. There are then $2^{m_c}$ choices of $a$, and these $a$ partition $X$ into sets 
\[X_C(\sigma,\, a) = \{x_1\} \times \dots \times \{x_{k_0 }\} \times X_{k_0 + 1}(a) \times \dots \times X_{k_1}(a) \times X_{k_1+1} \times \dots \times X_r,\]
with $X_i(a)$ the subset of $X_i$ consistent with the choice of $P$ and $x_1, \dots, x_{k_0}$.

Given an $k_0 < i \le k_1$, the union of all $X_C(\sigma,\, a)$ for which 
\[|X_i(a)| \le \frac{1}{ 2^{|P| + k_0} \cdot k_2^{c_9}}\cdot |X_i|\]
has order at most $k_2^{-c_9}|X|$.  Because of this, we can restrict the sum to be over only $(\sigma,\, a)$ that do not satisfy this inequality at all such $i$, introducing an error with magnitude bounded by
\[k_1k_2^{-c_9} \cdot |X| \cdot 2^{|\mathscr{M} \cup \mathscr{M}_P|} \cdot k_2!\]
Once restricted, each summand can be bounded by Proposition \ref{prop:eqd_lgn_raw} to be less than
\[{t}'^{-c_8} \cdot 2^{m_C} |X|,\]
giving the theorem.
\end{proof}

\subsection{Boxes of integers}
\label{ssec:box}

The results of this section, like the results of the first half of this paper, all apply to points in the product spaces of sets of primes. In this section, we finally give the definitions and results that allow us to move from the set of positive integers less than a certain bound to such a product space. As before, $S_r(N, D)$ denotes the set of squarefree integers less than $N$ with exactly $r$ prime factors, of which all are greater than $D$. 
\begin{defn}
\label{def:box}
Take $N \ge D_1 \ge D \ge 3$ to be real numbers, and take $r$ to be a positive integer satisfying \eqref{eq:r_ND_bnd}. Let $W$ be a subset of elements $S_r(N, D)$ that are comfortably spaced above $D_1$ (cf. Definition \ref{def:spacing}).

Let $k \le r$ be a nonnnegative integer, and choose a sequence of increasing primes
\[D < p_1 < \dots < p_k < D_1\]
Take 
\[D_1 < t_{k+1} < t_{k+2} < \dots < t_r\]
to be an increasing sequence of real numbers. For $i > k$, define
\[t'_i = \left(1 \,+\, \frac{1}{e^{i-k} \cdot \log D_1}\right)\cdot t_i.\]
Take
\[X = X_1 \times \dots \times X_r,\]
where $X_i = \{p_i\}$ for $i \le k$ and $X_i$ is the set of primes in the interval $(t_i, \,t'_i)$ for $i > k$.

If $t'_i < t_{i+1}$ for all $r > i > k$, we see that there is a natural bijection from $X$ to a subset of $S_r(N, D)$; abusing notation, we write $X$ for this subset too. We call $X$ a \emph{box meeting} $W$ if $X \cap W$ is nonempty.
\end{defn}

The restriction to comfortably spaced $W$ means that, if $X \cap W$ is nonempty, then we automatically have that the $X_i$ are disjoint sets and none of them contain any prime below $D_1$. This is very convenient.

\begin{prop}
\label{prop:box_smoothing}
Take $N \ge D_1 \ge D \ge 3$ with $\log \log N \ge 2 \log \log D_1$ and $r$ satisfying \eqref{eq:r_ND_bnd}, and take $W$ to be a subset of $S_r(N, D)$ that is comfortably spaced above $D_1$. Suppose $V$ is any other subset of $S_r(N, D)$, and suppose there are constants $\delta, \epsilon > 0$ such that
\[ |W| >  (1 - \epsilon)\cdot \big|S_r(N, D)\big|\]
and so that, for any box $X$ meeting $W$, we have
\[(\delta - \epsilon)\cdot \big|B_X\big| <  \big|V \cap B_X\big| < (\delta + \epsilon)\cdot \big|B_X\big|.\]
Then
\[ \big|V\big| \,=\, \delta \big|S_r(N, D)\big|  \,+\, \mathcal{O}\bigg(\left(\epsilon +(\log D_1)^{-1}\right) \cdot \big|S_r(N, D)\big|\bigg)\]
with absolute implicit constant.
\end{prop}

\begin{proof}
Take $\mathscr{D}_k$ to be the space of tuples 
\[\mathbf{t} = (p_1, \dots, p_k, t_{k+1}, \dots, t_r)\]
corresponding to boxes meeting $W$; we write the corresponding box as $X(\mathbf{t})$. Consider
\[\int_{\mathscr{D}_k} \big| V \cap X(\mathbf{t})\big|\cdot \frac{dp_1 \dots dp_k dt_{k+1} \dots dt_{r} }{t_{k+1} \dots t_{r}},\]
where the measure corresponding to $dp_1 \dots dp_k$ is one on every prime tuple and zero otherwise. If $n \in W$ has exactly $k$ prime factors less than $N_1$ and corresponds to the tuple $(q_1, \dots, q_r)$, then $n$ is in $X(\mathbf{t})$ if
\[(q_1, \dots, q_k) = (p_1, \dots, p_k)\]
and, for $i > k$,
\[t_i \le p_i \le t_i\left(1 + \frac{1}{e^{i-k} \cdot\log D_1}\right).\]
Then, if 
\begin{equation}
\label{eq:not_too_big}
\prod_{i = k+1}^{r} \left(1 + \frac{1}{e^{i-k}\cdot \log D_1}\right)n < N,
\end{equation}
we have that the measure of the subset of $\mathscr{D}_k$ corresponding to boxes containing $n$ is
\[\prod_{i = k+1}^{r}  \log\left(1 + \frac{1}{e^{i-k}\cdot \log D_1}\right).\]
If $n$ is outside $W$ but in $S_r(N, D)$ with exactly $k$ prime factors below $N_1$, or if $n$ is in $W$ but does not satisfy \eqref{eq:not_too_big}, then the measure of boxes containing $n$ is bounded by this product. Any $n$ not satisfying \eqref{eq:not_too_big} is in the range
\[N\cdot \big(1 - A \cdot \log(D_1)^{-1}\big) \le n \le N\]
where $A$ is some positive constant.

Taking $H_r(N, D)$ as in Section \ref{ssec:pois2}, and using \eqref{eq:HrND} together with Proposition \ref{prop:SrND_est}, we find
\[\frac{1}{r!} H_r(N, D) - \frac{1}{r!}H_r\big((1- c)N, D \big) = \mathcal{O}\left(c \,+\, \frac{(\log \log N)^4}{\log N}  \right)\cdot  \big|S_r(N, D)\big|\]
with implicit constant absolute whenever $c$ is in $(0, 1)$. From this, the number of $n$ not satisfying \eqref{eq:not_too_big} is bounded by $\mathcal{O}\left(|S_r(N, D)|/\log D_1\right)$.

We then see that
\[\sum_{k \ge 0} \prod_{i = k+1}^{r}  \log\left(1 + \frac{1}{e^{i-k} \log D_1}\right)^{-1} \int_{\mathscr{D}_k} \big| V \cap B(\mathbf{t})\big|\cdot \frac{dp_1 \dots dp_k dt_{k+1} \dots dt_{r} }{t_{k+1} \dots t_r}\]
is at least as large as 
\[\big| V \cap W\big| - \mathcal{O}\left((\log D_1)^{-1} \cdot \big|S_r(N, D)\big| \right)\]
and is no larger than $|V|$. The estimates on $\big| V \cap X(\mathbf{t})\big|$ relative to $\big| X(\mathbf{t})\big|$ then give the proposition.

\end{proof}

In this proposition, $W$ should be considered to be some ``nice'' subset of $S_r(N, D)$. We have already given three different notions of niceness with comfortable spacing, regularity, and extravagant spacing. Since our main results rely on using an effective, unconditional form of the Chebotarev density theorem, there is one more form of not-niceness that we must avoid. Whenever possible, we must avoid $L$ functions that have Siegel zeros.
\begin{prop}
\label{prop:no_sieg}
Take $d_1, d_2, \dots$ to be a potentially infinite sequence of distinct squarefree integers satisfying
\[d_i^2 < \big| d_{i+1}\big|.\]
Take $d_i'$ to be the product of the primes dividing $d_i$ that are greater than $D$, and take $\mathbf{d}'$ to be the subset of the $d_i'$ for which $|d_i|$ is greater than $D_1$. Take $N \ge D_1 \ge D \ge 3$ with $\log \log N \ge 2 \log \log D_1$ and $r$ satisfying \eqref{eq:r_ND_bnd}, and define
\[V = \bigcup_{X \cap\, \mathbf{d}'\cdot \Z \ne \emptyset} X.\]
Here, the union is over all boxes of $S_r(N, D)$ that contain some element $n$ divisible by an element of $\mathbf{d}'$. We assume $\log D_1 > 2D \log D$. Then
\[|V| = \mathcal{O}\left(\frac{1}{\log D_1} \cdot \big|S_r(N, D)\big| \right)\]
\end{prop}
\begin{proof}
Choose $d'_i$ in $S_r(N, D)$, writing it as a product $p_1 \cdot \dots \cdot p_m$. Suppose some element of $X$ is divisible by $d'_i$. Taking $n$ in $X$, we see that there are prime factors $q_1, \dots, q_m$ of $n$ such that
\[q_i = p_i \quad\text{if} \quad p_i < D_1 \quad\text{and}\]
\[\frac{1}{2} q_i < p_i < 2q_i\quad\text{otherwise}.\]
If $d'_i < N^{2/3}$, there is then an absolute constant $A$ so that the number of $n$ sharing a box with a multiple of $d'_i$ is bounded by
\[A^m \cdot \prod_{p_i \le D_1} p_i^{-1} \cdot \prod_{p_i > D_1} (\log p_i)^{-1} \cdot \left|S_r(N, D)\right| =\mathcal{O}\left(\left(\frac{1}{\log d'_i}\right)\cdot \big|S_r(N, D)\big| \right).\]
We can also bound the contribution from $d'_i \ge N^{2/3}$ by $\mathcal{O}(N/\log N)$.

We remove the first elements from the sequence $d_1, d_2, \dots$, renumbering so that $|d_1| > D_1$. We then get $|d_i| > D_1^{2^i}$, so $|d'_i| > D_1^{2^{i-1}}$.
Then the contribution from the $d'_i$ with $d'_i < N^{2/3}$ is
\[\mathcal{O}\left(|S_r(N, \, D)\big| \cdot \sum_{i > 1} \frac{1}{2^i \log D_1}\right),\]
within the bound.  The contribution from $d'_i > N^{2/3}$ is $\mathcal{O}(N/\log N)$, which is also within the bound. This proves the proposition.
\end{proof}

\begin{defn*}
We call a box \emph{Siegel-free above} $D_1$ if it is not contained in the set $V$ defined in the above proposition with respect to the sequence defined in Definition \ref{def:Sieg}.

In addition, we will call a box $C_0$-regular if it contains some $C_0$-regular element of $S_r(N, D)$, and we will call it extravagantly spaced if it contains some extravagantly spaced element.
\end{defn*}

In line with the previous two proposition, we see that we can ignore boxes that are not Siegel-free above $D_1$ so long as $D_1$ is sufficiently large.
\begin{defn*}
Choose absolute constants $c_{13}, c_{14}> 0$, and choose $r$, $N$, $D$ satisfying \eqref{eq:r_ND_bnd} and
\[D \le \log \log \log N.\]
Taking $X$ to be a (comfortably spaced) box of $S_r(N, D)$ with
\[D_1 = D^{(\log \log N)^{c_{13}}},\]
we call $X$ \emph{acceptable} if it is $C_0$-regular for
\[C_0 = c_{14}\cdot \log \log \log N\]
and if it is Siegel-free above $D_1$.
\end{defn*}

We now have all the tools to reprove Kane's results on $2$-Selmer groups from the Markov chain analysis of Swinnerton-Dyer. To reprove the results of  Fouvry-Kl{\"u}ners on $4$-class groups of imaginary quadratic fields, we would repeat this argument starting from the Markov chain analysis of Gerth in \cite{Gert84}.
\begin{cor}
\label{cor:Kane}
There is an absolute $c > 0$ so that we have the following:

Take $E/\QQ$ to be any elliptic curve with full rational $2$-torsion and no rational cyclic subgroup of order four, and take $P^{\text{\emph{Alt}}}(j\,|\,n)$  and $R_{E,\, 1}(n)$ as in the introduction. Take $R_0$ to be the set of squarefree integers. Then, for any $n_1 \ge 0$ and $N > 3$, we have
\[\left|\frac{\big|[N] \cap R_{E,\, 1}(n_1)\big|}{|[N] \cap R_0|}\, -\,\, 0.5 \lim_{n \rightarrow \infty}P^{\text{\emph{Alt}}}(n_1\,|\,2n + n_1)\right|\]
\[ = \mathcal{O}\left(\frac{1}{(\log \log N)^{c}}\right),\]
with the implicit constant depending only on the choice of $E$.
\end{cor}
\begin{proof}

By applying Theorem \ref{thm:SkND} and Proposition \ref{prop:no_sieg} to Proposition \ref{prop:box_smoothing}, we find that it suffices to prove this result on acceptable boxes of twists in $S_r(N, D)$ with $D$ larger than the largest bad prime of $E$.

We will apply Theorem \ref{thm:Lgn_perm} to our acceptable box with $k_0$ the minimal integer so $t'_{k_0 + 1}$ is larger than $D_1$. We take $t = D_1$, so the Siegel-less condition holds. Choose $k_1$ minimal so that $t'_{k_1 + 1}$ is larger than $\exp(D_1^{c_6})$, and take $k_2 = r$. Finally, take $P$ to be the set of all primes less than $D$ and $-1$, and take $\mathscr{M}$ and $\mathscr{M}_P$ maximal. We need to check that, for sufficiently large $N$ and some appropriate choice of $c_1, \dots, c_{14}$, the six conditions of Theorem \ref{thm:Lgn_perm} hold.
\begin{enumerate}
\item The first condition always holds.
\item The second condition holds for sufficient $N$ if $c_2 > c_6$.
\item The third condition holds for sufficient $N$ if
\[c_4 > 2 + c_3\log 2 + \frac{c_9}{c_{13}}\]
and $c_{14}$ is sufficiently small relative to the other constants.
\item The fourth condition holds for sufficient $N$.
\item The fifth condition holds for sufficient $N$ if $c_2 < c_7$ and $c_{14}$ is sufficiently small relative to the other constants.
\item The sixth condition holds for sufficient $N$ if $c_{10}, c_{11} > c_{13}$ and $c_{14}$ is sufficiently small relative to the other constants.
\end{enumerate}
It is a pleasantly mundane exercise to prove that there are positive constants $c_1,\dots, c_{14}$ that satisfy all the inequalities stated above and in Theorem \ref{thm:Lgn_perm}.

Then, considered up to permutation, the Legendre symbol matrices found in our acceptable box are equidistributed with error within the bound of the corollary.
Since the $2$-Selmer rank depends only on the permutation class, we can now apply Swinnerton-Dyer's work in \cite{Swin08}. This paper does not give error estimates, but we can find them with just a little extra work on the Markov chain described in  \cite[(20)]{Swin08}.

There is some $A, \epsilon > 0$ so we have the following: choose $k = 0, 1$, choose $n > 0$, and consider the Markov chain $Y$ described by \cite[(20)]{Swin08} with initial state $2n + k$. Under this Markov chain, if $T$ is the first passage time of our Markov chain to state $k$, we can bound the expected value of $(1 + \epsilon)^T$ by $A^n$. Similarly, if we start another Markov chain $X$ initially equaling the stationary, and if $T$ is the minimal time when $Y_T$ meets $X_T$, we find that  we can bound the expected value of $(1 + \epsilon)^T$ by $A^{n+1}$. Then by the logic of \cite[Theorem 1.8.3]{Norr97}, we find that there is some constant $C$ so, in the notation of the final section of \cite{Swin08},
\[\big|Q(d, M, CM) - \alpha_d\big| = \mathcal{O}\big(\exp(-M)\big).\]
Plugging this estimate into the final equation of \cite{Swin08} then shows that, among all Legendre symbol matrices corresponding to a twist with $r$ prime factors, the proportion corresponding to rank $d$ is $0.5 \lim_{n \rightarrow \infty}P^{\text{Alt}}(j\,|\,2n + j)$ with maximal error $\mathcal{O}\big(\exp(-cr)\big)$ for some constant $c > 0$, easily within our error term. This gives the corollary.

\end{proof}

\section{Proofs of the main theorems}
\label{sec:proofs}
In the previous section, we reduced distributional questions over the squarefree integers to distributional questions over acceptable boxes. In this section, we extend this logic to more and more specialized product spaces. Our goal is to reduce to product spaces on which a combination of Proposition \ref{prop:AR_main}, Proposition \ref{prop:staff}, and the Chebotarev density theorem suffice to prove the equidistribution results we want for $2^k$-Selmer groups and class groups. This will be enough to prove Theorems \ref{thm:Sel_main} and \ref{thm:Cl_main}.

With Proposition \ref{prop:AR_main}, the notions of $2^k$-Selmer groups and $2^k$-class groups have become essentially interchangeable. In this section, we will state all our results and arguments on the Selmer side; straightforward adjustments to the argument would give the results on the class side.

We begin by stating the explicit form of Theorem \ref{thm:Sel_main} that we will prove in this section. 
\begin{thm}
\label{thm:Sel_main_exp}
There is an absolute constant $c > 0$ so that, for any elliptic curve $E/\QQ$ with full $2$-torsion and no rational cyclic subgroup of order four, there is a choice of $A > 0$ so that, for any choice of $N > 0$, any choice of $m \ge 1$, and any sequence $n_1 \ge n_2 \ge \dots \ge n_{m+1}$ of nonnegative integers of the same parity, we have
\[\Bigg| \left| [N]\, \cap  \bigcap_{k = 1}^{m+1} R_{E,\, k}(n_k) \right|\,\,-\,\,  P^{\text{\emph{Alt}}}(n_{m+1} \,|\, n_m) \cdot \left|[N]\, \cap \bigcap_{k=1}^m R_{E,\, k}(n_k) \right|\Bigg|\]
\[ \le AN\cdot (\log \log \log \log N)^{-\frac{c}{m^26^m}}\]
whenever the latter expression is well defined and positive.
\end{thm}
From this theorem, we can derive the following explicit form of Corollary \ref{cor:main}.

\begin{cor}
\label{cor:drown}
Take $c$ to be a positive constant less than $\frac{\log 2}{\log 6}$, and take $E/\QQ$ to be an elliptic curve as in the previous theorem. Then there is some $N_0 > 0$ depending on $E$ and $c$ so that, for all $N > N_0$, we have
\begin{equation}
\label{eq:crnk2_est}
\bigg| \big\{ d \in [N]\,:\,\, \text{\emph{corank }} \text{\emph{Sel}}^{2^{\infty}} E^{(d)} \ge 2 \big\}\bigg|  \le \frac{N}{ \big(\log\log\log\log\log N\big)^{c}}.
\end{equation}
\end{cor}
\begin{proof}
We consider a Markov process whose states are the nonnegative integers. At each step, we take the transition probability from state $n$ to state $j$ to be $P^{\text{Alt}}(j \,|\, n)$. In this Markov chain, we note that the probability of stepping to $0$ after $2$ is $0.5$; the probability of stepping to either $0$ or $2$ after any other even state is at least two thirds; and the probability of $1$ after any other odd state is at least $0.5$. (All these facts follow from the formula for $P^{\text{Alt}}(j\,|\,n)$ given in  \cite{Heat94}). Then, independent of the initial probability distribution, the chance that the process is in a state other than $0$ or $1$ after $m$ steps is bounded by $\mathcal{O}(2^{-m})$.

Choose $c < c' < \frac{\log 2}{\log 6}$, and take
\[m = \left\lfloor \frac{c'}{\log 2} \log \log \log \log \log \log N \right\rfloor.\]
We can assume that this is positive. From Corollary \ref{cor:Kane} and the formulas from \cite{Heat94}, we see the proportion of $d \in X_N$ such that $E^{(d)}$ has $2$-Selmer rank exceeding $m + 2$ is bounded by $\mathcal{O}\left(2^{-c_1m^2}\right)$ for some constant $c_1 > 0$, in the range of the corollary's estimate for sufficient $N_0$.

We see that the set being bounded in \eqref{eq:crnk2_est} is contained in

\[[N] \,\cap \bigcup_{n_m \ge 2} R_{E,\, m}(n_m) \,=\, [N]\, \cap \bigcup_{n_1 \ge \dots \ge n_m \ge 2}\,\, \bigcap_{k =1}^m R_{E,\, k}(n_k)\]
For sufficient $N_0$ and some constant $c_2 > 0$, we also have
\begin{align*}
 \left| [N] \,\cap  \bigcap_{k = 1}^m R_{E,\, k}(n_k)\right| & \le \prod_{k=1}^{m-1} P^{\text{Alt}}(n_{k+1}\,|\, n_k) \cdot \bigg|[N]\, \cap R_{E, \, 1}(n_1)\bigg|  \\
&+ AmN\cdot (\log \log \log \log N)^{-\frac{c_2}{m^26^m}}
\end{align*}
for any sequence $n_1 \ge \dots \ge n_m \ge 2$. By summing this over all paths with $n_1 \le m$ and using our Markov chain result, we find that the set in \eqref{eq:crnk2_est} has maximal size
\[\mathcal{O}\left(2^{-m}N\right) + Am^{m+1}N \cdot (\log \log \log \log N)^{-\frac{c_2}{m^26^m}},\]
which is within the bound of the corollary for sufficiently large $N_0$.

\end{proof}
 
We now proceed to the proof of Theorem \ref{thm:Sel_main_exp}. To prove the result, we recast it in increasingly specialized situations. The first and easiest of these recasts is to move from an equidistribution result on integers less than $N$ to boxes in $S_r(N, D)$.
\begin{prop}
\label{prop:proof_A}
There is an absolute constant $c > 0$ so that, for any choice of $E/\QQ$ as above, there is some $A > 0$ so that we have the following:

Take $D$ one greater than the largest bad prime of $E$. Choose a positive real $N > 30$, and take
\[D_1 = D^{(\log \log N)^{1/10}}.\]
Choose $r$ satisfying \eqref{eq:r_ND_bnd}, and let $X$ be any box of some $S_r(N, D)$ with this $D_1$ that is extravagantly spaced, Siegel free above $D_1$, and $\sqrt{\log \log \log N}$ regular. Then,  for any choice of $m \ge 1$ and any sequence $n_1  \ge \dots \ge n_{m+1}$ of nonnegative integers of the same parity, we have
\[\Bigg| \left| X\, \cap  \bigcap_{k = 1}^{m+1} R_{E,\, k}(n_k) \right|\,\,-\,\,  P^{\text{\emph{Alt}}}(n_{m+1} \,|\, n_m) \cdot \left|X\, \cap \bigcap_{k=1}^m R_{E,\, k}(n_k) \right|\Bigg|\]
\[ \le A|X|\cdot (\log \log \log \log N)^{-\frac{c}{m^26^m}}\]
whenever the right hand side is defined and positive.
\end{prop}

\begin{proof}[Proof that Proposition \ref{prop:proof_A} implies Theorem \ref{thm:Sel_main_exp}]
With this proposition, Theorem \ref{thm:Sel_main_exp} is a consequence of applying Theorem \ref{thm:SkND} and Proposition \ref{prop:no_sieg} to Proposition \ref{prop:box_smoothing}.
\end{proof}

As $x$ in $X$ varies, the tuples $w = (T_1, T_2, \Delta_1, \Delta_2)$ corresponding to $2$-Selmer elements change. Our next step is to restrict our attention to sets $X(a)$, where we no longer have this problem. This reduction is technically cumbersome, as some choices of $a$ will prevent us from finding sets of variable indices as in part (3) of Definition \ref{def:AR_input}. We begin with the notation we will need.
\begin{defn*}
Take $E$, $X$, $N$, and $m$ as in Proposition \ref{prop:proof_A}, and assume the extravagant spacing of $X$ is between indices $k_{\text{gap}}$ and $k_{\text{gap}} + 1$. Take $P$ to be the union of the prime numbers less than $D$ with $\{-1\}$.  In the context of Definition \ref{def:Lgn_matrix}, take $\mathscr{M}$ and $\mathscr{M}_P$ maximal, and let $a$ be any function in $\mathbb{F}_2^{\mathscr{M} \cup \mathscr{M}_P}$.

Under these circumstances, any $\bar{x} \in \overline{X}_{[r]}$ entirely contained in $X(a)$ is quadratically consistent, so we can define additive-restrictive input as in Definition \ref{def:AR_input}. Take $\text{Ctp}_{(1)}, \dots, \text{Ctp}_{(m-1)}$ to be a choice of lower pairings as in part (1) of this definition, choose a basis $w_1, \dots, w_{n_1}$ and the integer $n_m$ as in part (2), and choose variable indices as in part (3). We assume $i_b > k_{\text{gap}}$; writing $S_{\text{pre-gap}}$ for the union of the $S(j_1, j_2) - \{i_b\}$, we assume
\[S_{\text{pre-gap}} \subseteq \big[0.5k_{\text{gap}},\, k_{\text{gap}}\big].\]
Take $P^-_{\text{pre-gap}}$ to be an element of $\prod_{i \in [k_{\text{gap}}] - S_{\text{pre-gap}}} X_i$. We assume that $a$ is consistent with the choice of $P^-_{\text{pre-gap}}$.

Then all of the data we have chosen so far will be called \emph{inital data for Proposition \ref{prop:proof_B}}.
\end{defn*}
We will write
\[X_i(a, \, P^-_{\text{pre-gap}})\]
for the subset of $X_i$ consistent with $a$ and the data of $P^-_{\text{pre-gap}}$, and take $X(a,\, P^-_{\text{pre-gap}})$ for the subset of $X(a)$ equaling $P^-_{\text{pre-gap}}$ on $[k_{\text{gap}}] - S_{\text{pre-gap}}$. Finally, given a choice of sequence pairings $\text{Ctp}_{(1)}, \dots, \text{Ctp}_{(k)}$, take
\[X(a, \, P^-_{\text{pre-gap}},\, k)\]
for the subset of $i(a, \, P^-_{\text{pre-gap}})$ whose first $k$ Cassels-Tate pairings agree with the given sequence.

\begin{prop}
\label{prop:proof_B}
There is a constant $c > 0$ so we have the following:

Choose initial data for Proposition \ref{prop:proof_B} as above. Writing
\[n_{\max}= \left\lfloor \sqrt{\frac{c}{m^26^m} \log \log \log \log \log N} \right\rfloor,\]
we assume $n_{\max}$ is defined, positive, and greater than $n$. We also assume that we have
\begin{equation}
\label{eq:proof_B_Xi}
\big|X_i(a,\, P^-_{\text{\emph{pre-gap}}})\big| > 4^{-k_{\text{\emph{pre-gap}}}} \cdot |X_i|.
\end{equation}
for $i \in S_{\text{\emph{pre-gap}}}$.

 Finally, take $\text{\emph{Ctp}}_{(m)}$ to be any  $n_m \times n_m$ alternating matrix with coefficents in $\mathbb{F}_2$.  Then there is some constant $A > 0$ depending only on $E$ so that 
 \[\bigg| \big| X(a,\, P^-_{\text{\emph{pre-gap}}},\, m)\big| \,-\,  2^{-\frac{n_m(n_m - 1)}{2}}  \cdot \big| X(a,\, P^-_{\text{\emph{pre-gap}}},\, m-1)\big|\bigg| \]
 \[ \le A \cdot \big|X(a,\,  P^-_{\text{\emph{pre-gap}}})\big| \cdot (\log \log \log \log N)^{-\frac{c}{m\cdot 6^m}}.\]
\end{prop}

\begin{proof}[Proof that Proposition \ref{prop:proof_B} implies Proposition \ref{prop:proof_A}]

This implication would be easy if we could prove the above Proposition for arbitrary choices of $a$, $P^-_{\text{pre-gap}}$, and the pairnigs. However, there are three kinds of bad $(a, P^-_{\text{pre-gap}})$ to consider. First, we need to avoid the case where $n$ is not less than $n_{\max}$. Second, we need to avoid $a$ such that, for some choice of pairings, we cannot find variable indices suitable for the initial data. Finally, we need to avoid $(a, P^-_{\text{pre-gap}})$ for which \eqref{eq:proof_B_Xi} does not hold for some $i \in S_{\text{pre-gap}}$.  We claim that the union of $X(a,\, P^-_{\text{pre-gap}})$ over all three kinds of bad $(a, \, P^-_{\text{pre-gap}})$ fits into the error term of Proposition \ref{prop:proof_A}.

We first claim that the union of $X(a)$ for which $n_m \ge n_{\max}$ fits into the error term of Proposition \ref{prop:proof_A}. This is a consequence of the argument of Corollary \ref{cor:Kane} and the formulas in \cite{Heat94}.

Next, consider the set of $a$ for which, for some choice of pairings $\text{Ctp}_{(k)}$ and a basis, there is no choice as in the lemma for the variable indices $S(j_1, j_2)$. We claim the union of the $X(a)$ over the set of $a$ for which this holds also fits in this error bound. 

First, we claim that the proportion of $a$ for which there are $2$-Selmer elements $w_1, w_2$ so that either $w_1$ or $w_2$ is non-torsion and
\begin{equation}
\label{eq:four_seven}
\big|(T_1(w_1) + T_2(w_2)) \cap [0.5k_{\text{gap}},\, k_{\text{gap}}]\big| > (0.25 + 2^{-10n_{\max}})\cdot k_{\text{gap}}
\end{equation}
has density at most
\[\mathcal{O}\left((15/16)^r + \exp\big(2^{-20n_{\max}} \cdot k_{\text{gap}}\big)\right)\]
in the space $\mathbb{F}_2^{\mathscr{M} \cup \mathscr{M}_P}$. Here, $T_1+T_2$ denotes the symmetric difference.

Call $a$ generic if there is no non-torsion $2$-Selmer tuple $w$ of $X(a)$ for which $T_1(w)$, $T_2(w)$, and $[r]$ are not linearly independent sets with respect to symmetric difference, and if there are no pair of non-torsion $2$-Selmer tuples $(w_1, w_2)$ with $w_1 +w_2$ also non-torsion, but where $T_1(w_1)$, $T_2(w_1)$, $T_1(w_2)$, $T_2(w_2)$, and $[r]$ are not linearly independent. From Lemmas 4-6 of \cite{Swin08}, we see that the proportion of $a$ that are not generic due to the condition on $w$ is bounded by
\[\mathcal{O}\left(2^{2|P|} \cdot (3/4)^r\right).\]
For the condition on $(w_1, w_2)$, we can use Lemma 7 from \cite{Swin08} after noting that the condition $u_1' = u_2''$ can be weakend to $u_1'/u_2'' \in X_S$ with no change in the argument. Then, from this lemma, the proportion of non-generic $a$ is bounded by
\[A^{|P|} \cdot (15/16)^r\]
for some absolute $A > 1$.

Now suppose $w$ is a generic tuple as above. From genericity, we can prove that the local conditions at the $r$ primes coming from $X$ are independent, and we find that the proportion of $a$ so that $w$ is a $2$-Selmer tuple for $X(a)$ is bounded by $\mathcal{O}\left(4^{-r}\right)$. Similarly, if $(w_1, w_2)$ is generic as above, the probability that $w_1$ and $w_2$ are both $2$-Selmer for $X(a)$ is bounded by $\mathcal{O}\left(16^{-r}\right)$.

Then Hoeffding's inequality is sufficient to complete the estimate of the density of $a$ in $\mathbb{F}_2^{\mathscr{M} \cup \mathscr{M}_P}$ not satisfying \eqref{eq:four_seven} for some $w_1, w_2$.

For any  $a$ other than those in this set, it is easy to find sets of variable indices if $n_{\text{max}}$ is larger than some constant determined by $E$. First, choose some $i_b > k_{\text{gap}}$, and add torsion to the basis as necessary so $i_b$ is not in any $T_i(w_j)$. Then each $S(j_1, j_2) -\{i_b, i_a(j_1, j_2)\}$ can be taken to be any subset of size $m$ inside of
\[ T_2(w_{j_2}) \cap \big([r] - T_1(w_{j_2})\big) \cap \bigcap_{j \ne j_2} \bigg([r] - \big(T_1(w_j) \cup T_2(w_j)\big)\bigg).\]
The assumptions on $a$ give that this intersection has density about $4^{-n_1}$ on the integers in the interval $[0.5k_{\text{gap}}, k_{\text{gap}}]$, which will be larger than $m$ for sufficient $n_{\text{max}}$. We can find $i_a$ similarly.

If $k_2 < 0.5 k_{\text{gap}}$, we see that permutations of the first $k_2$ indices do not change whether \eqref{eq:four_seven} holds for a given $a$. Then, from Theorem \ref{thm:Lgn_perm}, we find that our argument implies that the union of $X(a)$ over all $a$ for which it may be impossible to find a set of variable indices fits into the error of Proposition \ref{prop:proof_A}.

Next, we claim that the union of $X(a,\, P_{\text{pre-gap}}^{-})$ over all $(a,\, P_{\text{pre-gap}}^{-})$ for which \eqref{eq:proof_B_Xi} does not hold for some $i$ fits into the error of Proposition \ref{prop:proof_A}. We will work in the context of Proposition \ref{prop:eqd_lgn_raw}. To do this, add the primes $p_1, \dots, p_k$ of the box to the set $P$; taking $X_i(a, P)$ to be the subset of $X_i$ consistent with $P$ and the choice of $a$, we will attempt to apply the argument of the proposition to
\[X_1(a, P) \times \dots \times X_r(a, P).\]
This will only work if no $X_i(a, P)$ is smaller than $\frac{1}{(\log t_1')^{c'}} \cdot |X_i|$ for some choice of the constant $c'$. For a good choice of constants, outside a set of choices of $a$ over which the union of the $X(a)$ fits into the error of Proposition \ref{prop:proof_A}, we always have
\[X_i(a, P)\ge \frac{1}{(\log t_1')^{c'}} \cdot |X_i|.\]
Suppose we have such an $a$. Then a choice of $P_{\text{pre-gap}}^{-}$ for which \eqref{eq:proof_B_Xi} does not hold would be exceptional in the sense of the proof of Proposition \ref{prop:eqd_lgn_raw}. Per that proof, the union of all such exceptional sets fits into the error of Proposition \ref{prop:proof_A}.

Finally, we note that there are at most $2^{mn_{\max}^2}$ sequences of pairings $\text{Ctp}_{(k)}$. Writing $X_{aP^-}$ for $ X(a,\, P^-_{\text{pre-gap}})$, the claim of the proposition then implies
\[\Bigg| \left| X_{aP^-} \,\cap   \bigcap_{k = 1}^{m+1} R_{E,\, k}(n_k) \right| \,-\, P^{\text{Alt}}(n_{m+1}\,| n_m) \cdot \left| X_{aP^-} \cap \, \bigcap_{k = 1}^{m} R_{E,\, k}(n_k) \right|\Bigg|\]
\[ \le A \cdot 2^{mn_{\max}^2} \cdot  |X_{aP^-}|\cdot (\log \log \log \log N)^{-\frac{c}{m \cdot 6^m}}.\]
A computation shows that the sum of this error over all $a$ and $P^-_{\text{pre-gap}}$ is then within the error of Proposition \ref{prop:proof_A}. This gives the lemma.
\end{proof}

Now that we have a set of variable indices, the next structure to add is a set of governing expansions as in part (5) of Definition \ref{def:AR_input}. The requirements on these governing expansions are quite stringent, making this step the most interesting part of the reduction of Theorem \ref{thm:Sel_main_exp}. We first need notation for the extra structure.
\begin{defn*}
Choose initial data for Proposition \ref{prop:proof_B} that obeys the conditions of Proposition \ref{prop:proof_B}. Choose subsets $Z_i$ of $X_i$ for each $i$ in  $S_{\text{pre-gap}}$. For each set $S(j_1, \, j_2)$ of variable indices, choose a set of governing expansions $\mathfrak{G}(i_a(j_1, \,j_2))$ on the product $Z_{\text{pre-gap}}$ of the $Z_i$. For any set $S$ of the form $S(j_1, j_2) - \{i_b\}$ and any $\bar{x} \in \overline{\big(Z_{\text{pre-gap}}\big)}_S$, we assume 
\[\phi_{\bar{x}}(\mathfrak{G}(i_a(j_1, j_2)))\]
exists.

For $x \in Z_{\text{pre-gap}}$, write $L(x)$ for the composition of all quadratic fields ramified only at $\infty$, the places of $P$, and the places of $P^-_{\text{pre-gap}}$. Write $M(j_1, j_2)$ for the composition of the fields of definition for the set of $\phi_{\bar{x}}$ with $\bar{x} \in \overline{\big(Z_{\text{pre-gap}}\big)}_{S(j_1, j_2) - \{i_b\}}$. Also write $M_{\circ}(j_1, j_2)$ for the composition of the fields of definition for the set of $\phi_{\bar{x}}$ with $\bar{x} \in \overline{\big(Z_{\text{pre-gap}}\big)}_S$ for some proper subset $S$ of $S(j_1, j_2) - \{i_b\}$.

We assume that, for each $S(j_1, j_2)$, the field $M_{\circ}(j_1, j_2)/\QQ$ splits completely at all primes in $P$, in $P^-_{\text{pre-gap}}$, and in any $Z_i$ with $i$ outside $S(j_1, j_2) - \{i_b\}$.

Finally, take $M$ to be the composition of any $L(x)$ with the set of $M(j_1, j_2)$, and take $M_{\circ}$ to be the composition of any $L(x)$ with the set of $M_{\circ}(j_1, j_2)$. We write
\[X_i(M_{\circ})\]
to be the subset of primes $p$ in $X_i$ so $p$ is consistent with the choice of $a$ and $P^-_{\text{pre-gap}}$ and the prime $p$ splits completely in each $M_{\circ}(j_1, j_2)$. Note that $X_i(M_{\circ})$ is described alternatively as the subset of $X_i$ mapping under the Frobenius map to one specific central element of $\text{Gal}(M_{\circ}/\QQ)$. Finally, take
\[Z = \{P^-_{\text{pre-gap}}\} \times  Z_{\text{pre-gap}} \times \prod_{i > k_{\text{gap}}}  X_i(M_{\circ}).\]
\end{defn*}
\begin{prop}
\label{prop:proof_C}
There is an absolute constant $c> 0$ so we have the following:

Choose initial data for this proposition as above. Taking
\[M = \left\lfloor(\log \log \log \log N)^{1/5(m+1)}\right\rfloor,\]
we assume that $M$ is well defined and positive, and that each $Z_i$ has cardinality $M$. 

Then there is a constant $A > 0$ depending only on $E$ so that
\[ \bigg| \big| Z \cap X(a,\, P^-_{\text{\emph{pre-gap}}},\, m)\big| \,-\,  2^{-\frac{n_m(n_m - 1)}{2}}  \cdot \big|Z \cap  X(a,\, P^-_{\text{\emph{pre-gap}}},\, m-1)\big|\bigg| \]
 \[ \le A \cdot \big| Z \cap X(a,\, P^-_{\text{\emph{pre-gap}}})\big| \cdot (\log \log \log \log N)^{-\frac{c}{m\cdot 6^m}}.\]
\end{prop}
\begin{proof}[Proof that Proposition \ref{prop:proof_C} implies Proposition \ref{prop:proof_B}]
Choose initial data for Proposition \ref{prop:proof_B} obeying the conditions of Proposition \ref{prop:proof_B}. Write $V_{\text{pre-gap}}$ for the subset of $\prod_{i \in S_{\text{pre-gap}}} X_i$  consistent with $P^-_{\text{pre-gap}}$ and the conditions of $a$. Take
\[R = \left\lfloor \exp \exp\big(0.2 k_{\text{gap}}\big) \right\rfloor.\]
We can assume $R$ is positive. We also assume that
\[m < \log \log \log \log \log N,\]
as Proposition \ref{prop:proof_B} is otherwise vaccuous.

We will choose $t \ge 0$ and, for each $i \in S_{\text{pre-gap}}$, we will choose sequences of subsets
\[Z^{1}_i,\, \dots, Z^t_i \subseteq X_i(a,\, P^-_{\text{pre-gap}}),\]
with each set of cardinality $M$. We take
\[Z^{\ell}_{\text{pre-gap}} = \prod_{i \in S_{\text{pre-gap}}} Z^\ell_i.\]
We assume that these subsets obey the following conditions:
\begin{itemize}
\item For $\ell \ne \ell'$, we have that $Z^{\ell}_{\text{pre-gap}}$ and $Z^{\ell'}_{\text{pre-gap}}$ intersect at at most one point.
\item Each $Z^{\ell}_{\text{pre-gap}}$ is a subset of $V_{\text{pre-gap}}$, and any point in $V_{\text{pre-gap}}$ is in at most $R$ of the $Z^{\ell}_{\text{pre-gap}}$. 
\item The set $Z^{\ell}_{\text{pre-gap}}$ can be used as initial data for Proposition \ref{prop:proof_C}.
\end{itemize}
Furthermore, we assume that the sequence of $Z^{\ell}_{\text{pre-gap}}$ cannot be extended under these requirements to a sequence of $t+1$ subgrids. 

Write
\[X_{\text{pre-gap}} = \prod_{i \in S_{\text{pre-gap}}} X_i(a,\, P^-_{\text{pre-gap}}).\]
Take $V^{\text{bad}}_{\text{pre-gap}}$ to be the set of points in $V_{\text{pre-gap}}$ that are consistent with the choice of $a$ and $P^-_{\text{pre-gap}}$ and that are in fewer than $R$ of the $Z^{\ell}_{\text{pre-gap}}$. Write $\delta$ for the density of $V^{\text{bad}}_{\text{pre-gap}}$ in $X_{\text{pre-gap}}$. By a greedy algorithm, we can choose a subset $W$ of $V^{\text{bad}}_{\text{pre-gap}}$ of density at least $\delta/RM^{m+1}$ such that no point in $W$ is in more than two of the $Z^{\ell}_{\text{pre-gap}}$.

By adjoining splitting behavior at the primes in $P^-_{\text{pre-gap}}$ to the system constructed in Proposition \ref{prop:ARS_gov}, we can then define an additive-restrictive system on $X_{\text{pre-gap}}$ with $\overline{Y}^{\,\circ}_{\emptyset} = W$ and where, if  $\bar{x} \in \overline{Y}_{S_{\text{pre-gap}}}^{\,\circ}$, then the governing expansions defined at $\bar{x}$ are as required for Proposition \ref{prop:proof_C}. The maximal size of the abelian groups in this additive-restrictive system is bounded by $2^{k_{\text{gap}} + |P|}$. Then, by Proposition \ref{prop:ars_density}, the density of $\overline{Y}_{S_{\text{pre-gap}}}^{\,\circ}$ in $X_{\text{pre-gap}} \times X_{\text{pre-gap}}$ is at least
\[\left(\frac{\delta}{2^{k_{\text{gap}} |P|} \cdot RM^{m+1}}\right)^{3^{|S_{\text{pre-gap}}|}}.\]
We note $|S_{\text{pre-gap}}| \le (m+1)n_0^2$. In addition, for sufficently large $N$, we always have
\[|X_i(a, \, P^-_{\text{pre-gap}})| >\exp \exp(0.3 \cdot k_{\text{gap}})\]
for $i \in S_{\text{pre-gap}}$. Applying Proposition \ref{prop:subgrid} and the assumptions on $t$, we then have
\[M^{2m} > \frac{\exp(0.3 \cdot k_{\text{gap}})}{(m+1)3^{(m+1)n_m^2}\cdot (\exp(0.25\cdot k_{\text{gap}}) + \log \delta^{-1})} \]
for sufficiently large $N$. We can then bound $\delta$  by $\exp(- e^{0.25k_{\text{gap}}})$ for sufficiently large $N$. Then, following the logic of Proposition \ref{prop:eqd_lgn_raw}, we see that the subset of $x \in X(a, P^{-}_{\text{pre-gap}})$ for which $\pi_{S_{\text{pre-gap}}}(x)$ is in $V^{\text{bad}}_{\text{pre-gap}}$ fits easily into the error term of Proposition \ref{prop:proof_B}.

We associate grids $Z_{\text{pre-gap}}^{\ell}$ with fields $M^{\ell}$ and $M^{\ell}_{\circ}$ and a supergrid $Z^{\ell}$ as above. For $x \in X(a,\, P^-_{\text{pre-gap}})$ with $\pi_{S_{\text{pre-gap}}}(x)$ outside of $V^{\text{bad}}_{\text{pre-gap}}$, write $\Lambda(x)$ to be the number of $\ell \le t$ for which $x$ is in $Z^{\ell}$. Write $d_{ML}$ for the degree of $M^{\ell}$ over some $L(x)$ with $x \in Z^{\ell}_{\text{pre-gap}}$; from Proposition \ref{prop:gov_set_ind}, we find this degree does not depend on $\ell$ or $x$. For $i > k_{\text{gap}}$, write $X_i(L(x))$ for the subset of $X_i(a,\, P_{\text{pre-gap}})$ consistent with the choice of $x$. From the Chebotarev density theorem as presented in Proposition \ref{prop:2Cheb} and the definition of extravagant spacing, we then have
\[|X_i(M^{\ell}_{\circ})| = d_{ML}^{-1} \cdot |X_i(L(x))|\left(1 + \mathcal{O}\left(e^{-2k_{\text{gap}}}\right)\right)\]
for $i > k_{\text{gap}}$. Following Proposition \ref{prop:eqd_lgn_raw} then gives that the subset of
\[\prod_{i > k_\text{gap}} X_i(M^{\ell}_{\circ})\]
consistent with $a$ has order
\[d_{ML}^{-(r - k_{\text{gap}})} \cdot |X(a,\, P^-_{\text{pre-gap}}) \cap \pi^{-1}_{S_{\text{pre-gap}}}(x)| \cdot \left(1 + \mathcal{O}\left(e^{-k_{\text{gap}}}\right)\right).\]
From this, we calculate that $\Lambda(x)$ has average value
\[d_{ML}^{-(r - k_{\text{gap}})} R \cdot  \left(1 + \mathcal{O}\left(e^{-k_{\text{gap}}}\right)\right).\]
Similarly, from the requirements on $Z^{\ell} \cap Z^{\ell'}$ and Proposition \ref{prop:gov_set_ind}, we see $M_{\circ}^{\ell}M_{\circ}^{\ell'}$ has degree $d_{ML}^2$ over $L(x)$ for $Z^{\ell}$ and $Z^{\ell'}$ distinct grids containing $x$. Then the average square value of $\Lambda(x)$ is
\[\left(d_{ML}^{-2(r - k_{\text{gap}})} (R^2 - R) \,+\, d_{ML}^{-(r- k_{\text{gap}})} R\right) \cdot \left(1 +  \mathcal{O}\left(e^{-k_{\text{gap}}}\right)\right)\]
\[= d_{ML}^{-2(r - k_{\text{gap}})} \cdot R^2 \left(1 + \mathcal{O}\left(e^{-k_{\text{gap}}}\right)\right).\]
Then, outside a set of density $\mathcal{O}\left(e^{-0.5k_{\text{gap}}}\right)$ in the domain of $\Lambda$, we find that $\Lambda(x)$ over the mean value of $\Lambda$ is within $e^{-0.25 k_{\text{gap}}}$ of $1$. The effect of the set of low density fits into the error term of Proposition \ref{prop:proof_B}, and the variance between the $\Lambda(x)$ also fits into the error of this proposition. Then, to prove Proposition \ref{prop:proof_B}, it is enough to prove Proposition \ref{prop:proof_C} for each grid $Z^{\ell}$.

\end{proof}

\begin{proof}[Proof of Proposition \ref{prop:proof_C}]
Take $F$ to be a nonzero multiplicative character of the vector space of $n_m$ dimensional alternating matrices with coefficients in $\mathbb{F}_2$. For $x \in  Z \cap  X(a,\, P^-_{\text{pre-gap}},\, m-1)$, write $\text{CT}(x)$ for the Cassels-Tate pairing on $D_{(m)}$. To prove the proposition, it is enough to prove that
\begin{align*}
&\sum_{x \in  Z \cap  X(a,\, P^-_{\text{pre-gap}},\, m-1)} F(\text{CT}(x)) \\
&\qquad\qquad =  \mathcal{O}\left(\big| Z \cap X(a,\, P^-_{\text{pre-gap}})\big| \cdot (\log \log \log \log N)^{-\frac{c}{m\cdot 6^m}}\right)
\end{align*}
for each $F$.

Choose an $F$, and take $j_1 < j_2 \le n_0$ so that $F$ depends on the value of $\text{CT}(x)_{j_1j_2}$, and take $S = S(j_1, j_2)$. From Proposition \ref{prop:gov_set_ind}, we find that there is a natural bijection
 \[\text{Gal}(M(j_1, j_2)M_{\circ}/M_{\circ}) \cong \mathscr{G}_{S - \{i_b\}}(\pi_{S - \{i_b\}}(Z))\]
of $\mathbb{F}_2$ vector spaces, with our notation as in Definition \ref{def:GSZ}. For $\sigma$ in this Galois group, we take $X_{i_b}(\sigma)$ to be the subset of $X_{i_b}(M_{\circ})$ mapping under Frobenius to $\sigma$. From the Chebotarev density theorem, we find
\[|X_{i_b}(\sigma)| = 2^{-(M-1)^{m+1}} \cdot |X_{i_b}(M_{\circ})| \cdot \left(1 + \mathcal{O}\left(e^{-k_{\text{gap}}}\right)\right).\]
Choose $x_i \in X_i(M_{\circ})$ for $i$ above $k_{\text{gap}}$ besides $i_b$ such that the set of $x_i$ is consistent with $a$, writing this tuple as $P^-_{\text{post-gap}}$.  From Proposition \ref{prop:jutila} and Propoosition \ref{prop:eqd_lgn_raw}, we see that, outside a negligible set of choices of $P^-_{\text{post-gap}}$, if we write $X_{i_b}(P^-_{\text{post-gap}})$ for the subset of $X_{i_b}$ consistent with $a$, we have
\begin{equation}
\label{eq:Frob_fun}
\big|X_{i_b}(\sigma) \,\cap\, X_{i_b}(P^-_{\text{post-gap}})\big|\qquad\qquad\qquad\qquad\qquad
\end{equation}
\[= 2^{-(M-1)^{m+1}} \cdot \big|X_{i_b}(M_{\circ})\, \cap \,X_{i_b}(P^-_{\text{post-gap}})\big| \cdot \left(1 + \mathcal{O}\left(e^{-k_{\text{gap}}}\right)\right)\]
for each $\sigma$.

On the grid
\[Z_{AR}  = Z_{\text{pre-gap}}\, \times \, \big(X_{i_b}(M_{\circ}) \cap X_{i_b}(P^-_{\text{post-gap}}) \big),\]
we can find full additive-restrictive input as in Definition \ref{def:AR_input}. The corresponding additive-restrictive system has abelian groups with orders bounded by $2^{n_{\max}(n_{\max} + 2m + 6)}$. We now apply Proposition \ref{prop:staff} to the additve-restrictive system $\mathfrak{A}(\mathscr{P})(j_1, j_2)$. By Propositions \ref{prop:AR_main} and \ref{prop:staff}, if
\begin{equation}
\label{eq:PC_eps_1}
\epsilon < 2^{-n_{\max}(n_{\max} + 2m + 6)}
\end{equation}
and
\begin{equation}
\label{eq:PC_eps_2}
\log M \ge A \cdot 6^{m + 2} \log \epsilon^{-1},
\end{equation}
then there is a choice of  $\sigma_1,  \dots, \sigma_M$ in $\text{Gal}(M(j_1, j_2)M_{\circ}/M_{\circ})$ so, for any $\sigma$ in this Galois group and any choice of  $Z'_{AR} = Z_{\text{pre-gap}} \times \{x_1, \dots, x_M\}$ with
\[x_i \in X_{i_b}(\sigma + \sigma_i)\, \cap\, X_{i_b}(P^-_{\text{post-gap}})\quad\text{for all } i \le M,\]
we have
\[\sum_{x \in Z'_{AR}} F(\text{CT}(x)) \le \epsilon \cdot |Z'_{AR}|.\]
From the estimate \eqref{eq:Frob_fun}, we see that $Z_{AR}$ can be split into grids $Z'_{AR}$ with leftovers fitting into the error term of the proposition, so we have equidistribution on $Z_{AR}$ too.

For an appropriate constant $c' > 0$, we find that
\[\epsilon = (\log \log \log \log N)^{-\frac{c'}{(m+1)6^m}}\]
satisfies both \eqref{eq:PC_eps_1} and \eqref{eq:PC_eps_2} for $\epsilon$ sufficiently small. This gives the proposition, hence the proposition, hence the proposition, hence the theorem, hence the corollary.

\end{proof}

\bibliography{references}{}
\bibliographystyle{amsplain}

\end{document}